\documentclass[11pt,a4paper]{article}
\usepackage[a4paper]{geometry}
\usepackage{amssymb,latexsym,amsmath,amsfonts,amsthm}
\usepackage{graphicx}
\usepackage{epstopdf}
\usepackage{tikz}
\usepackage{pgflibraryshapes}
\usetikzlibrary{arrows,decorations.markings}
\usepackage{subfigure}
\usepackage{overpic}
\usepackage{fullpage}
\usepackage{comment,verbatim}
\usepackage{hyperref}
\usepackage{color}
\usepackage{mathrsfs}
\usepackage[title,titletoc]{appendix}
\usepackage{ulem}

\renewcommand{\Re}{\mathrm{Re}\,}
\renewcommand{\Im}{\mathrm{Im}\,}

\newcommand{\ud}{\,\mathrm{d}}

\newcommand{\Boh}{\mathcal{O}}
\def\red{\color{red}}
\def\blue{\color{blue}}

\newtheorem{theorem}{Theorem}[section]
\newtheorem{lemma}[theorem]{Lemma}
\newtheorem{proposition}[theorem]{Proposition}

\newtheorem{rhp}[theorem]{RH problem}

\newtheorem{Riemann-Hilbert Problem}{Definition}

\theoremstyle{definition}

\theoremstyle{remark}

\newtheorem{remark}[theorem]{Remark}

\numberwithin{equation}{section}

\begin{document}

\title{On The Eigenvalue Rigidity of the Jacobi Unitary Ensemble}

\author{Dan Dai\footnotemark[1] \ and Chenhao Lu\footnotemark[2]}

\renewcommand{\thefootnote}{\fnsymbol{footnote}}
\footnotetext[1]{Department of Mathematics, City University of Hong Kong, Tat Chee
Avenue, Kowloon, Hong Kong. E-mail: \texttt{dandai@cityu.edu.hk}}
\footnotetext[2]{Department of Mathematics, City University of Hong Kong, Tat Chee
Avenue, Kowloon, Hong Kong. E-mail: \texttt{chenhaolu3-c@my.cityu.edu.hk}}

\date{\today}

\maketitle

\begin{abstract}
In this paper, we prove an optimal global rigidity estimate for the eigenvalues of the Jacobi unitary ensemble. Our approach begins by constructing a random measure defined through the eigenvalue counting function. We then prove its convergence to a Gaussian multiplicative chaos measure, which leads to the desired rigidity result. To establish this convergence, we apply a sufficient condition from Claeys et al. \cite{CFL2021} and conduct an asymptotic analysis of the related exponential moments.
\end{abstract}

\tableofcontents

\section{Introduction and Statement of Results}

The Jacobi Unitary Ensemble (JUE) is a fundamental model in random matrix theory (RMT) that describes the statistics of a system of random particles (or real eigenvalues) confined to the interval $[-1, 1]$.  It is defined by the joint probability density function (PDF) for $N$ eigenvalues $\lambda_1, \lambda_2, \ldots, \lambda_N$ below (cf. Mehta \cite{Mehta-book}):
\begin{equation}\label{JUEdensity-classical}
\rho_N(\lambda_1,\cdots, \lambda_N) = \frac{1}{Z_N}\prod_{1\leq j<k\leq N} |\lambda_j-\lambda_k|^2 \prod_{i=1}^N (1-\lambda_i)^\alpha (1+\lambda_i)^\beta , \quad \alpha,\beta>-1,
\end{equation}
on $[-1, 1]^N$, where $Z_N$ is the normalization constant. This ensemble originated in statistics through the multivariate analysis of variance (MANOVA) for certain linear models (see \cite{Bai:Jack2010, Muirhead}), and is consequently known as the MANOVA ensemble in statistics. The JUE can be constructed in terms of random matrices. A standard approach involves two independent complex Wishart matrices. Let $a$ and $b$ be $n_1 \times N$ and $n_2 \times N$ matrices, respectively, with $n_1 \geq N$ and $n_2 \geq N$, whose $(n_1 +n_2)N$ entries are i.i.d. standard complex Gaussian random variables. Defining the $N \times N$ Wishart matrices $A =a^* a $ and $B = b^* b$, then the eigenvalues $x_1, x_2, \ldots, x_N$ of the matrix $A(A + B)^{-1}$ are distributed according to the JUE density \eqref{JUEdensity-classical} with $\lambda_j = 1-2x_j$. The parameters are given by $\alpha=n_1 - N$ and $\beta = n_2 - N$; for example, see Constantine \cite{Constantine}. In addition, Killip and Nenciu \cite{Kil:Nen2004} provided a tridiagonal matrix model whose eigenvalues are distributed according to the general $\beta$-Jacobi ensemble, which includes the JUE as a special example. In the physics literature, the JUE plays a significant role in areas such as quantum conductance and log-gas theory; for example, see \cite{Bee1997, FPJ}.


In random matrix theory, a central problem is the characterization of eigenvalue statistics. On a macroscopic scale, the limiting global density of the JUE is given by 
\begin{equation} \label{JUE-limiting-measure}
d\mu_{J}(x) = \frac{1}{\pi \sqrt{1-x^2}} dx, \qquad x \in (-1,1).
\end{equation}
In addition to JUE, other classical unitary models include the Gaussian unitary ensemble (GUE) and the circular unitary ensemble (CUE). For GUE, the limiting eigenvalue distribution is described by the well-known Wigner semicircle law:
\begin{equation} \label{eq:sc-law}
d\mu_{sc}(x) = \frac{2}{\pi} \,\sqrt{1-x^2} \, dx, \qquad x \in [-1,1].
\end{equation}
In contrast, the CUE consists of unitary matrices whose eigenvalues $e^{i\theta_j}$ lie on the unit circle. As the matrix size tends to infinity, the eigenangles $\theta_j$ follow a uniform distribution $\frac{1}{2 \pi} d\theta$ on $[0, 2 \pi]$. 

While these global eigenvalue distributions are well-established, the next natural step is to study eigenvalue properties on the microscopic scale. At this level, a key phenomenon known as {\it eigenvalue rigidity} is observed. Rigidity describes the property that individual eigenvalues lie close to their predicted classical locations, with fluctuations on a scale only slightly larger than the microscopic scale. To illustrate this in the context of CUE, consider the ordered eigenangles $0\leq \theta_1\leq \cdots \leq \theta_N < 2\pi$. Given the uniform limiting distribution of CUE, the classical location of the $j$-th eigenangle is $2 \pi j/N$. For any $\varepsilon>0$, Arguin, Belius and Bourgade \cite[Theorem 1.5]{ABB2017} show that 
\begin{equation} \label{eq:cue-optimal}
\lim_{N \to \infty}
\mathbb{P}\left( (2-\varepsilon)\frac{\log N}{N} \leq \max_{j=1, \cdots, N} \left| \theta_j-\frac{2\pi j}{N}\right| \leq (2+\varepsilon)\frac{\log N}{N} \right) = 1.
\end{equation}
This implies that the maximum fluctuation of the eigenangles $\theta_j$ around its classical location $2 \pi j/N$ is of order $O(\log N / N)$. This result is regarded as an example of {\it optimal global rigidity}, since the upper and lower bounds for the maximal fluctuation match up to the multiplicative constants $2\pm \varepsilon$, which can be made arbitrarily close by choosing $\varepsilon$ sufficiently small.

The rigidity phenomenon for the GUE is more intricate due to the different features of the bulk of the spectrum and the soft edges. The density of the semicircle law \eqref{eq:sc-law} is positive in the bulk $[-1+\delta, 1-\delta]$ and vanishes as a square root at the endpoints $\pm 1$. Let $\lambda_1 \leq \lambda_2 \leq \cdots \leq \lambda_N$  be the ordered eigenvalues of an $N \times N$ GUE matrix, and define the quantiles $\kappa_j$ by 
\begin{equation} \label{eq:gue-quantile}
\int_{-1}^{\kappa_j}d\mu_{sc}=\frac{j}{N}, \qquad j=1,\cdots, N.
\end{equation}
In the bulk, Gustavsson \cite{Gus2005} established the following central limit theorem (CLT): for any fixed $\delta > 0$ and $j\in( \delta N, (1-\delta) N)$,
\begin{equation}
2\sqrt{2} \sqrt{1-\kappa_j^2} \frac{N}{\sqrt{\log N} } (\lambda_j - \kappa_j) \xrightarrow{\ d \ } \mathcal{N}(0,1), \qquad \textrm{as } N \to \infty. 
\end{equation}
This above CLT shows that in the bulk, the fluctuation of eigenvalue $\lambda_j$ around its quantile $\kappa_j$ is typically of order $\mathcal{O}\left(\frac{\sqrt{\log N} }{N}\right)$. Beyond individual eigenvalue fluctuations, there is significant interest in characterizing the global eigenvalue fluctuation across the entire spectrum, including both the bulk and the soft edges. This problem has been investigated in a series of seminal papers \cite{Erd:Sch:Yau2009-cmp, Erd:Sch:Yau2009, Erd:Yau:Yin2012-ptrf, Erd:Yau:Yin2012}, which proved profound results valid for a broad class of random matrices, namely Wigner matrices. For GUE, which is a special example of complex Wigner matrices, the results in \cite[Thm. 2.2]{Erd:Yau:Yin2012} read: there exist constants $\alpha > \alpha' > 0$ and $C, c > 0$ such that
\begin{equation}
\mathbb{P}\left(  \max_{j=1, \ldots, N} \left\{ \sqrt{1-\kappa_j^2} \, |\lambda_ k - \kappa_j| \right\} \geq \frac{(\log N)^{\alpha \log \log N}}{N} \right) \leq C \exp \left( -c (\log N)^{\alpha' \log \log N} \right).
\end{equation}
This estimation suggests that the global eigenvalue fluctuations of a GUE matrix are highly likely to be smaller than $(\log N)^{\alpha \log \log N}/N$. This result was improved for the GUE by Claeys et al. \cite[Thm. 1.2]{CFL2021}, where an optimal bound is established: for any $\epsilon > 0$, 
\begin{equation} \label{eq:gue-optimal}
\lim_{N\to \infty} \mathbb{P}
\left[ \frac{1-\epsilon}{2} \frac{\log N}{N} \leq \max_{j=1, \ldots, N} \left\{ \sqrt{1-\kappa_j^2} \, |\lambda_ j - \kappa_j| \right\}\leq 
\frac{1+\epsilon}{2} \frac{\log N}{N} \right] = 1.
\end{equation}
This confirms that the maximal scaled fluctuation is precisely of order $O(\log N/N)$. Besides random matrix models, it is worth mentioning that similar optimal rigidity bounds have also been established for other ensembles, such as the sine $\beta$-process in \cite{Hol:Paq2018,Paq:Zei2018} and discrete $\beta$-ensembles in \cite{GH2019}.

Since GUE and CUE are among the simplest and most well-studied models in random matrix theory, it is natural to investigate whether optimal rigidity estimates analogous to \eqref{eq:cue-optimal} and \eqref{eq:gue-optimal} hold for more general unitary ensembles. In this paper, we pursue this direction by establishing an optimal eigenvalue rigidity result for the JUE. A key new feature of the JUE, compared to the GUE and CUE, is the presence of a hard edge. As seen in the limiting density \eqref{JUE-limiting-measure}, this density blows up as an inverse square root at the endpoints of the support $(-1,1)$. Therefore, obtaining a uniform global rigidity bound that remains valid all the way up to the hard edges is a non-trivial problem.


Results for the GUE and CUE indicate that, whether in the bulk or near a soft edge, the scaled eigenvalue fluctuations remain of order $O(\log N/N)$. An essential question is therefore how the presence of a hard edge affects this fluctuation scale. Our main result, Theorem~\ref{Main Theorem}, answers this question by establishing an optimal global rigidity estimate for the JUE of the same order, thereby demonstrating the robustness of this phenomenon even in the presence of hard edges. Moreover, we consider a generalization of the JUE defined by the following PDF:
\begin{equation}\label{JUEdensity}
\rho_N(\lambda_1,\cdots, \lambda_N) = \frac{1}{Z_N}\prod_{1\leq j<k\leq N} |\lambda_j-\lambda_k|^2 \prod_{i=1}^N (1-\lambda_i)^\alpha (1+\lambda_i)^\beta e^{t(\lambda_i)}, \quad \alpha,\beta>-1,
\end{equation}
where $t(x)$ is a real analytic function on $[-1, 1]$. The case $t(x) = 0$ corresponds to the classical JUE. For a general real analytic function $t(x)$, the global limiting eigenvalue density of the above modified JUE 
is still given by \eqref{JUE-limiting-measure}; see Kuijlaars and Vanlessen \cite{Kui-Van2002}. 

Our proof relies on the analysis of a normalized eigenvalue counting function. Let $\lambda_1\leq\cdots\leq \lambda_N$ be the ordered eigenvalues of the modified JUE \eqref{JUEdensity}. We define the function $h_N(x)$ as 
\begin{equation}\label{hN}
 h_N(x)=\sqrt{2}\, \pi\left(\sum_{1\leq j\leq N} \mathbf{1}_{\lambda_j\leq x} - N F(x) \right) , \qquad  x\in [-1,1],
\end{equation}
where $\mathbf{1}_{\lambda \leq x}$ is the indicator function, and $F(x)$ is the cumulative distribution function corresponding to \eqref{JUE-limiting-measure}, namely
\begin{equation}\label{distribution of eq measure}
F(x) = \int_{-1}^x \frac{1}{\pi \sqrt{1-s^2}} ds = 1-\frac{1}{\pi} \arccos x, \qquad  x\in [-1,1].
\end{equation}
The global fluctuation of linear statistics $\sum_{j=1}^N f(\lambda_j)$ for J$\beta$E as $N \to \infty$  was established by Dumitriu and Paquette \cite{DP2012}. For the JUE, an equivalent result was also obtained in \cite{BB-PAFA-2021, CG2021}, which will used in our subsequent analysis. More precisely, \cite[Corollary 2.2]{CG2021} states that for a real analytic function $f$ in a neighbourhood of $[-1,1]$, we have\footnote{A classical global fluctuation result for G$\beta$E was first proved by Johansson \cite{Johansson1998}. For unitary ensembles with general Gaussian-type weights, a similar result to \eqref{eq: JUE-Global-CLT} is also presented in \cite[Corollary 2.2]{CG2021}, with the same variance. However, due to the influence of the hard edge at $\pm 1$, the mean is not zero; instead, it is given by the expression in \eqref{eq: JUE-Global-CLTmean}.}
\begin{equation} \label{eq: JUE-Global-CLT}
-\frac{1}{\sqrt{2} \, \pi}\int_{-1}^{1} f'(x) h_N(x) dx =
\sum_{j=1}^N f(\lambda_j) - N\int_{-1}^1 f(x) d\mu_J(x) \xrightarrow{\ d \ } \mathcal{N}(\mu(f), \sigma^2(f)),
\end{equation}
where 
\begin{eqnarray} \label{eq: JUE-Global-CLTmean}
\mu(f) &=& \frac{\alpha + \beta}{2\pi} \int_{-1}^{1}\frac{f(x)}{\sqrt{1-x^2}} - \frac{\alpha}{2}f(1) - \frac{\beta}{2}f(-1) + \sigma^2(f; t), 
\end{eqnarray}
and $\sigma^2(f) = \sigma^2(f; f)$ with
\begin{eqnarray}
\sigma^2(f; g) &=& \int\int_{[-1, 1]^2} f'(x)g'(y)\frac{\Sigma(x, y)}{2\pi^2}dxdy, \\
\Sigma(x, y) &=& \log\left|\frac{1-xy+\sqrt{1-x^2}\sqrt{1-y^2}}{x-y}\right|. \label{correlation kernel of X}
\end{eqnarray}
Hence, $h_N(x)$ converges in distribution to a Gaussian log-correlated field $X(x)$ with correlation kernel \eqref{correlation kernel of X}. Subsequently, following the ideas in \cite{CFL2021}, we construct a random measure 
\begin{equation} \label{eq:dmu-N}
d\mu_N^\gamma(x) = \frac{e^{\gamma h_N(x)}}{\mathbb{E} e^{\gamma h_N(x)}}, \qquad x \in [-1,1],
\end{equation}
and prove its convergence to a Gaussian multiplicative chaos (GMC) measure as $N \to \infty$. The theory of multiplicative chaos, introduced by Kahane \cite{Kahane1985}, provides a framework for constructing random measures from log-correlated fields; see Rhodes and Vargas \cite{Rho:Var2014} for a comprehensive review. Since the logarithm of the characteristic polynomial in many random matrix models exhibits log-correlated behavior, GMC theory has become a powerful tool for studying these properties, particularly extreme value statistics; see, for example, \cite{Bre:Web:Wong2018, BLZ-GAFA2025, JLW-cpam-2024, NSW-CUE-2020, Webb-CUE-2015}. Given that the function $h_N(x)$ is closely related to the imaginary part of the log-characteristic polynomial, we adopt this approach to establish global eigenvalue fluctuations as in \eqref{eq:gue-optimal}. A crucial aspect of our argument involves verifying the sufficient conditions outlined in \cite[Assumptions 2.5]{CFL2021} to demonstrate that the measure $d\mu_N^\gamma(x)$ in \eqref{eq:dmu-N} converges to a GMC measure for $0<\gamma<\sqrt{2}$. This convergence provides a precise characterization of the extremal behavior of $h_N(x)$, from which the desired rigidity results follow.

It is worth mentioning that beyond random matrices, GMC is also extensively applied when considering extreme value statistics in other probability models; see, for instance, \cite{BH-sine-Gordon-2022, CZ-polymer-2023}.

The following subsection presents our main results.

\subsection{Main Results}\label{sec:results}

We begin by establishing an extremal property of the normalized eigenvalue counting function $h_N(x)$.

\begin{theorem} \label{main theorem for hN}
Let $h_N(x)$ be defined as in \eqref{hN}. For any $\varepsilon>0$, we have
\begin{equation}
\lim_{N\to\infty} \mathbb{P} \left[(1-\varepsilon)\sqrt{2}\log N \leq \max_{x\in [-1, 1]}\{\pm h_N(x)\} \leq (1+\varepsilon)\sqrt{2}\log N\right] = 1.
\end{equation}
\end{theorem}

\begin{remark}
The maximum of $h_N(x)$ in the JUE admits the same estimation as in the GUE \cite[Theorem 1.3]{CFL2021}. This is not surprising as it is reasonable to expect that the ensembles differ significantly only at their edges: $x = \pm 1$ for the hard edge in JUE and the soft edge in GUE. According to its definition in \eqref{hN}, we have $h_N(\pm 1) = 0$ in JUE; while we can expect that $\lim_{N\to\infty} \mathbb{P} (h_N( \pm 1) = 0) = 1$. Hence, in the sense of probability, the maximum of $h_N(x)$ does not differ between these two cases.
\end{remark}


To state our rigidity result, we define the $j$-th percentile $\kappa_j$, associated with \eqref{JUE-limiting-measure} by
\begin{equation} \label{eq: kappa-j-def}
\int_{-1}^{\kappa_j} \frac{1}{\pi}\frac{1}{\sqrt{1-x^2}}dx = \frac{j-1/2}{N}, \qquad j = 1, \ldots, N.
\end{equation}
In contrast to the definition \eqref{eq:gue-quantile} used for the semicircle law, the choice of $\frac{j-1/2}{N}$ ensures symmetry about the origin. It follows immediately that $\kappa_j = -\cos \frac{(j-1/2)\pi}{N}$. 

Our main result is the following optimal global rigidity estimate for the JUE.

\begin{theorem}\label{Main Theorem}
Let $\lambda_1\leq\cdots\leq \lambda_N$ be the ordered eigenvalues of the ensemble \eqref{JUEdensity}, where $t(x)$ is a real analytic function on $[-1,1]$. Then, for any $\varepsilon>0$, we have
\begin{equation} \label{eq:main-theorem}
\lim_{N\to\infty}\mathbb{P}\left((1-\varepsilon)\frac{\log N}{N} \leq \max_{j=1, \cdots, N}\left\{\frac{1}{\sqrt{1-\kappa_j^2}}|\lambda_j-\kappa_j|\right\} \leq (1+\varepsilon)\frac{\log N}{N}\right) = 1.
\end{equation}
\end{theorem}

\begin{remark} \label{remark for main theorem}
We employ the GMC theory to prove the lower bound in the above theorem, following a strategy similar to the GUE case in \cite{CFL2021}. However, the proof of the upper bound differs significantly due to the influence of the hard edge. More precisely, the singularities of $ F'(x)$ at $ x = \pm 1 $, as seen in \eqref{distribution of eq measure}. Our strategy proceeds in two stages: first, we use the bounds for the maximum of $ h_N(x) $ from Theorem \ref{main theorem for hN} to establish an initial bulk estimate $ |\lambda_j - \kappa_j| \leq \log N / N $ for all $ j $ (see \eqref{eq:event-Bn0} below). Then, we implement an iterative refinement procedure to improve this estimate as $ \kappa_j $ approaches the hard edges at $ \pm 1 $.
\end{remark}

\begin{remark}
The approach used for the GUE edge rigidity in \cite{CFL2021} does not extend to the JUE. In the GUE case, the semicircular density $ \frac{2}{\pi} \sqrt{1-\kappa_j^2} $ is monotonically increasing near $ -1 $, and the fluctuation $ |\lambda_j - \kappa_j| $ is bounded by $ |\lambda_1 - \kappa_1| $, where the smallest eigenvalue $\lambda_1$ follows the Tracy-Widom distribution \cite{Tracy-Widom}. In contrast, for the JUE, the function $ F'(\kappa_j) $ is unbounded near $ -1 $, and the fluctuations $ |\lambda_j - \kappa_j| $ are larger than $ |\lambda_1 - \kappa_1| $. Consequently, we need to rely on the iterative refinement scheme outlined in Remark \ref{remark for main theorem} to establish edge rigidity. An analogous argument applies near the right endpoint $ x = 1 $.
\end{remark}

To study the convergence of the measure $d\mu_N^\gamma(x)$ in \eqref{eq:dmu-N}, we need a detailed estimation of the exponential moments $h_N(x)$, which are related to Hankel determinants. By Heine's identity (see, for example, \cite[Prop 3.8]{Deift2000}), we have
\begin{equation}\label{Heine Identity}
\mathbb{E}e^{\sum_{j=1}^N w(\lambda_j)} = \frac{1}{Z_N} D_N(e^{w}),
\end{equation}
where the expectation is taken with respect to density function in \eqref{JUEdensity}. For our purposes,  we  need asymptotics of Hankel determinants as follows
\begin{multline}
D_N(x_1, x_2; \gamma_1, \gamma_2; w)  \\
:= \det\left(\int_{\mathbb{R}} \lambda^{i+j} e^{\sqrt{2}\pi\gamma_1 \mathbf{1}_{\lambda\leq x_1} + \sqrt{2}\pi\gamma_2 \mathbf{1}_{\lambda\leq x_2}} (1-\lambda)^\alpha (1+\lambda)^{\beta} e^{w(\lambda)} d\lambda\right)_{i,j=0}^{N-1}, \label{eq1}
\end{multline}
and
\begin{eqnarray}
D_N(x; \gamma; w) := \det\left(\int_{\mathbb{R}} \lambda^{i+j} e^{\sqrt{2}\pi\gamma \mathbf{1}_{\lambda\leq x}} (1-\lambda)^\alpha (1+\lambda)^{\beta} e^{w(\lambda)} d\lambda\right)_{i,j=0}^{N-1}, \label{eq2}
\end{eqnarray}
where $-1 < x_1 \leq x_2 < 1$, $\gamma_1, \gamma_2 \in \mathbb{R}$, and $w(\lambda)$ is real-analytic in $\mathbb{R}$. We will adopt Deift-Zhou nonlinear steepest method to derive asymptotics of these determinants.

\paragraph{Organization of the paper}
The rest of the paper is organized as follows. In Section \ref{RHP for Hankel Determinants}, we present the Riemann-Hilbert problem and the differential identities for the Hankel determinants. In Section \ref{Steepest Descent Analysis for the RH Problem}, we apply the Deift-Zhou nonlinear steepest descent method to transform the original Riemann-Hilbert problem into small-norm ones. Depending on the weight of the Hankel determinants, we will consider three distinct cases: (i) the weight has two jumps at $x_1$ and $x_2$, which are neither close to each other nor to the endpoints $\pm 1$; (ii) the weight has two jumps at $x_1$ and $x_2$, which are close to each other but not to the endpoints $\pm 1$; (iii) the weight has only one jump at $x$, which is close to $-1$ or $1$. In Section \ref{Sec:Asy-hankle}, we derive the asymptotics of the Hankel determinants under these three different regimes using steepest descent analysis. Finally, in Section \ref{Eigenvalue Rigidity}, we investigate the maximum of $h_N(x)$, establish eigenvalue rigidity, and prove our main theorems.

\section{The Riemann-Hilbert Problems for the Hankel Determinants} \label{RHP for Hankel Determinants}

Starting from this section, we conduct steepest descent analysis to derive asymptotics of the corresponding Hankel determinants defined in \eqref{eq1} and \eqref{eq2}. It is well-known that these determinants are related to a Riemann-Hilbert (RH) problem as follows.

{
}
\begin{rhp}\label{rhp for Y}
 \hfill
 \begin{itemize}
 \item[(a)] $Y: \mathbb{C}\setminus [-1, 1] \to \mathbb{C}^{2\times 2}$ is analytic.

 \item[(b)]  $Y(z)$ possesses continuous boundary values $Y_\pm(x)$ for $x \in (-1, 1)\setminus\{x_1, x_2\}$, with $-1<x_1\leq x_2<1$, when $z$ approaches $x$ from above and below, respectively. Moreover, they satisfy the following jump condition
\begin{equation}
Y_+(x) = Y_-(x)\left(\begin{array}{cc} 1 & \mathcal{W}(x)\\ 0 & 1 \end{array}\right) \qquad  \textrm{for }x\in (-1, 1)\setminus\{x_1, x_2\},
\end{equation}
 where
 \begin{equation} \label{RHP-Y-weight}
  \mathcal{W}(x) = e^{\sqrt{2}\pi\gamma_1 1_{(-1, x_1]}(x)+\sqrt{2}\pi\gamma_2 1_{(-1, x_2]}} w_J(x), \qquad x \in [-1,1].
 \end{equation}
 Here $\gamma_1, \gamma_2$ are two real constants and $w_J(x)$ is the following modified Jacobi weight
    \begin{equation}\label{Jacobi weight w}
    w_J(x) = (1-x)^\alpha (1+x)^\beta  e^{t(x)}, \qquad \alpha, \beta > -1,
    \end{equation}
and $t(x)$ is real analytic on $[-1,1]$.

 \item[(c)] As $z\to\infty$, $Y(z)$ has the following asymptotics
  \[
  Y(z) = \left( I+O\left(\frac{1}{z}\right)\right) \left(\begin{array}{cc} z^N & 0\\ 0 & z^{-N} \end{array}\right).
  \]

 \item[(d)] As $z\to x_j, j=1, 2$ and $z \in \mathbb{C} \setminus [-1,1]$, we have
  \[
  Y(z) = O(\log|z-x_j|).
  \]

 \item[(e)] As $z\to 1$ and $z \in \mathbb{C} \setminus [-1,1]$, we have
  \begin{equation}\label{asym1 for Y}
  Y(z) = \left\{
  \begin{array}{ll}
  O\left(\begin{array}{cc} 1 & |z-1|^\alpha \\ 1 & |z-1|^\alpha \end{array}\right), & \text{if } \alpha<0;\\
  O\left(\begin{array}{cc} 1 & \log|z-1| \\ 1 & \log|z-1| \end{array}\right), & \text{if } \alpha=0,\\
  O\left(\begin{array}{cc} 1 & 1 \\ 1 & 1 \end{array}\right), & \text{if } \alpha>0.
  \end{array}
  \right.
  \end{equation}

  \item[(f)] As $z\to -1$ and $z \in \mathbb{C} \setminus [-1,1]$, we have
  \begin{equation}\label{asym2 for Y}
  Y(z) = \left\{
  \begin{array}{ll}
  O\left(\begin{array}{cc} 1 & |z+1|^\beta \\ 1 & |z+1|^\beta \end{array}\right), & \text{if } \beta<0;\\
  O\left(\begin{array}{cc} 1 & \log|z+1| \\ 1 & \log|z+1| \end{array}\right), & \text{if } \beta=0,\\
  O\left(\begin{array}{cc} 1 & 1 \\ 1 & 1 \end{array}\right), & \text{if } \beta>0.
  \end{array}
  \right.
  \end{equation}
 \end{itemize}
\end{rhp}

By the well-known results of Fokas, Its and Kitaev \cite{Fokas-Its-Kitaev1992}, the solution of the above RH problem is given in terms of corresponding orthogonal polynomials and their Cauchy transforms. Let $p_N(x)$ be the orthonormal polynomials with respect to $\mathcal{W}(x)$ in \eqref{RHP-Y-weight}, that is
\begin{equation}
\int_{-1}^1 p_m(x) p_n(x) \mathcal{W}(x) dx = \delta_{m,n},
\end{equation}
where $p_N(x) = \varkappa_N x^N + \cdots$ with $\varkappa_N>0$. Then, the unique solution to the above RH problem is given by
\begin{eqnarray} \label{Y}
Y(z) = Y_N(z; x_1, x_2; \gamma_1, \gamma_2; \mathcal{W}) =
\left(
\begin{array}{cc}
\frac{1}{\varkappa_N} p_N(z) & \frac{1}{\varkappa_N} C[p_N \mathcal{W}](z) \\
-2\pi i \varkappa_{N-1} p_{N-1}(z) & -2\pi i \varkappa_{N-1} C[p_{N-1} \mathcal{W}](z)\\
\end{array}
\right),
\end{eqnarray}
where $C[f]$ is the Cauchy transform of $f$ on $[-1,1]$
\begin{eqnarray}
C[f](z) = \frac{1}{2\pi i}\int_{-1}^1  \frac{f(x)}{x-z} dx \qquad z \in \mathbb{C}\setminus[-1,1].
\end{eqnarray}
One may refer to \cite[Theorem 2.4]{ABJ2004} for the case where $\gamma_1=\gamma_2=0$.

Furthermore, the logarithmic derivatives of the Hankel determinants defined in \eqref{eq1} and \eqref{eq2} admit explicit representations in terms of the matrix-valued function $Y$ specified in \eqref{Y}. The following two differential identities play a crucial role, which provide the foundation for our subsequent asymptotic analysis.

\begin{proposition} \label{prop:diff-identity}
Let the Hankel determinants $D_N(x_1, x_2; \gamma_1, \gamma_2; w)$ and $D_N(x; \gamma; w)$ defined in \eqref{eq1} and \eqref{eq2}, respectively. For $\gamma_1, \gamma_2 \in \mathbb{R}$, $-1 < x_1 \leq x_2 < 1$, we have
\begin{equation}\label{di1}
\frac{d}{dx_2} \log D_N(x_1, x_2; \gamma_1, \gamma_2; t) = -(1-x_2)^\alpha (1+x_2)^\beta e^{t(x_2)} \frac{1-e^{\sqrt{2}\pi\gamma_2}}{2\pi i} \lim_{z\to x_2}\left(Y^{-1}(z) Y'(z)\right)_{2, 1},
\end{equation}
where the limit is taken as $z\to x_2$ with $z\in\mathbb{C}\setminus \{x_1, x_2\}$. If $x_1 = x_2 = x \in (-1, 1)$, we have
\begin{equation}\label{di2}
\frac{d}{dx} \log D_N(x; \gamma; t) = -(1-x)^\alpha (1+x)^\beta e^{t(x)} \frac{1 - e^{\sqrt{2}\pi\gamma}}{2\pi i} \lim_{z\to x}\left(Y^{-1}(z) Y'(z)\right)_{2, 1},
\end{equation}
where $\gamma = \gamma_1+\gamma_2$, and the limit is taken in a similar way as \eqref{di1}. 
\end{proposition}

\begin{proof}
The proof is similar to that in \cite[Sec. 3]{Charlier-IMRN2019}. Similar differential identities are also obtained in \cite{CG2021, CFL2021}.
\end{proof}

Starting from the next section, we will perform a series of transformations to RH problem \ref{rhp for Y}. The key idea is to transform the original RH problem to the so-called ``small norm'' RH problem, where the asymptotics can be established.

\section{Steepest Descent Analysis for the RH Problem} \label{Steepest Descent Analysis for the RH Problem}

\subsection{Normalization: $Y\mapsto T$}

The first transformation $Y\mapsto T$ is introduced to normalize the behavior of $Y$ when $z\to\infty$. Let
\begin{equation} \label{transform: Y-to-T}
T(z) = 2^{N\sigma_3}Y(z)\varphi(z)^{-N\sigma_3}
\end{equation}
for $z\in \mathbb{C}\setminus [-1, 1]$, where
\begin{equation*}
\sigma_3 = \left(\begin{array}{cc} 1 & 0 \\ 0 & -1 \end{array}\right)
\end{equation*}
is the third Pauli matrix and
\begin{equation} \label{varphi}
\varphi(z) = z+(z^2-1)^{1/2}, \qquad z\in \mathbb{C}\setminus [-1, 1].
\end{equation}
The principal branch is taken in the above definition such that $\varphi(z)$ behaves like $2z$ as $z\to\infty$. One can easily check that
\begin{equation}\label{phi+phi-}
\varphi_{-}(x)\varphi_{+}(x) = 1 \qquad \text{ for } x\in (-1, 1).
\end{equation}
Then we have the following RH problem for $T$.

\begin{rhp}
 \hfill
 \begin{itemize}
 \item[(a)] $T: \mathbb{C}\setminus [-1, 1] \to \mathbb{C}^{2\times 2}$ is analytic.

 \item[(b)] $T$ satisfies the following jump relations
 \begin{eqnarray}
 T_+(x) =\left\{
 \begin{array}{ll}
 T_-(x) \left(\begin{array}{cc} \varphi_+(x)^{-2N} & e^{\sqrt{2}\pi\gamma_1+\sqrt{2}\pi\gamma_2}w_J(x) \\ 0 & \varphi_-(x)^{-2N} \end{array} \right), \quad  & -1<x<x_1;\\
 T_-(x) \left(\begin{array}{cc} \varphi_+(x)^{-2N} & e^{\sqrt{2}\pi\gamma_2}w_J(x) \\ 0 & \varphi_-(x)^{-2N} \end{array} \right), \qquad & x_1<x<x_2; \label{T-jump-2} \\
 T_-(x) \left(\begin{array}{cc} \varphi_+(x)^{-2N} & w_J(x) \\ 0 & \varphi_-(x)^{-2N} \end{array} \right), \qquad & x_2<x<1;
 \end{array}
 \right.
 \end{eqnarray}

 \item[(c)] As $z\to\infty$, $T(z)$ has the following asymptotics
  \[
  T(z) = I+O\left(\frac{1}{z}\right)
  \]

 \item[(d)] As $z\to x_j, j=1, 2$, we have
  \[
  T(z) = O(\log|z-x_j|).
  \]

 \item[(e)] As $z\to \pm 1$, $T(z)$ satisfies the same behaviour as in \eqref{asym1 for Y} and \eqref{asym2 for Y}.
 \end{itemize}
\end{rhp}

\subsection{Contour Deformation: $T\to S$}

As $\Re \varphi_{\pm}(x) = 0$ when $x\in (-1, 1)$, the diagonal entries $\varphi_{\pm}(x)^{2N}$ are highly oscillatory as $N\to \infty$. To remove these oscillatory terms, we introduce the second transformation $T\to S$, which involves a contour transformation. This is based on the following matrix factorization
\begin{eqnarray}
& &\left(\begin{array}{cc} \varphi_+(x)^{-2N} & e^{\sqrt {2}\pi\gamma} w_J(x) \\ 0 & \varphi_-(x)^{-2N}\end{array}\right) \nonumber\\
&& = \left(\begin{array}{cc} 1 & 0 \\ e^{-\sqrt {2}\pi\gamma}w_J(x)^{-1}\varphi_-(x)^{-2N} & 1\end{array}\right)\left(\begin{array}{cc} 0 & e^{\sqrt {2}\pi\gamma}w_J(x) \\ -e^{-\sqrt {2}\pi\gamma}w_J(x)^{-1} & 0\end{array}\right) \nonumber\\
& & \times\left(\begin{array}{cc} 1 & 0 \\ e^{-\sqrt {2}\pi\gamma}w_J(x)^{-1}\varphi_{+}(x)^{-2N} & 1 \end{array}\right), \label{T-jump-factorize}
\end{eqnarray}
where the property of $\varphi(z)$ in \eqref{phi+phi-} is used. Depending on properties of the jump points in the original weight function $\mathcal{W}(x)$ in \eqref{RHP-Y-weight}, we will consider three cases in the subsequent analysis, namely,
\begin{itemize}
\item[(I)] the weight has two jumps at $x_1$ and $x_2$, which are neither close to each other nor to the endpoints $\pm 1$;

\item[(II)] the weight has two jumps at $x_1$ and $x_2$, which are close to each other but not to the endpoints $\pm 1$;

\item[(III)] the weight has only one jump at $x$, which is close to $-1$ or $1$.
\end{itemize}

We will discuss the second transformation $T\mapsto S$ for Case (I) in details, while the other two cases are similar. First, let us extend the definition of $w_J(x)$ in \eqref{Jacobi weight w} from $[-1,1]$ to $U \setminus ((-\infty, -1] \cup [1, \infty)$, where $U$ is a neighbourhood of $[-1, 1]$ in which $t(z)$ is analytic. Denote this analytic extension by $w_J(z)$, i.e.,
\begin{equation} \label{def:wJz}
w_J(z) = (1-z)^\alpha (1+z)^{\beta} e^{t(z)}, \qquad z \in U \setminus ((-\infty, -1] \cup [1, \infty).
\end{equation}
Based on the factorization in \eqref{T-jump-factorize}, we also define 
\begin{eqnarray}
J_1(z; \gamma) &=& \left(\begin{array}{cc} 1 & 0 \\ e^{-\sqrt {2}\pi\gamma}w_J(z)^{-1}\varphi(z)^{-2N} & 1\end{array}\right), \label{eq:JS-jump-1} \\
J_2(x; \gamma) &=& \left(\begin{array}{cc} 0 & e^{\sqrt {2}\pi\gamma}w_J(x) \\ -e^{-\sqrt {2}\pi\gamma}w_J(x)^{-1} & 0\end{array}\right), \label{eq:JS-jump-2} \\
J_3(z; \gamma) &=& \left(\begin{array}{cc} 1 & 0 \\ e^{-\sqrt {2}\pi\gamma}w_J(z)^{-1}\varphi(z)^{-2N} & 1 \end{array}\right), \label{eq:JS-jump-3}
\end{eqnarray}
where the value of $\gamma$ is taken equal to $\gamma_1+\gamma_2$ for $\Re z<x_1$, equal to $\gamma_2$ for $x_1<\Re z<x_2$ and equal to $0$ for $\Re z>x_2$.  Now, we introduce the second transformation $T \mapsto S$ by defining $S$ as follows:
\begin{equation} \label{eq:T-S-map}
S(z) = \left\{\begin{array}{ll}T(z)J_3(z; \gamma)^{-1}, & z\in\Omega,\\ T(z)J_1(z;\gamma), & z\in\overline{\Omega},\\ T(z), & \text{elsewhere};
\end{array}\right.
\end{equation}
see the description of the regions $\Omega$ and $\overline{\Omega }$ in Figure \ref{RHP for S}.

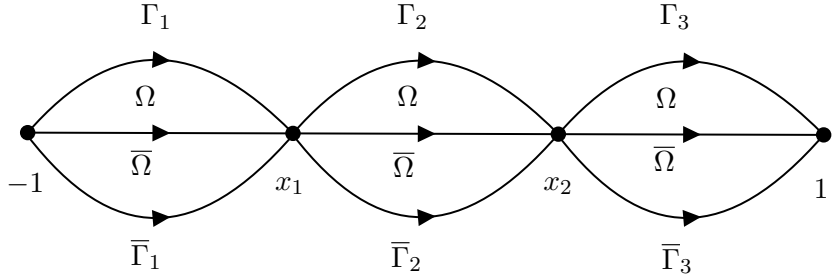
\begin{figure}[htbp]

\centering

\tikzset{every picture/.style={line width=0.75pt}} 

\begin{tikzpicture}[x=0.75pt,y=0.75pt,yscale=-1,xscale=1]

\draw    (100,110) -- (232.44,110.38) ;
\draw [shift={(171.22,110.2)}, rotate = 180.16] [fill={rgb, 255:red, 0; green, 0; blue, 0 }  ][line width=0.08]  [draw opacity=0] (8.93,-4.29) -- (0,0) -- (8.93,4.29) -- cycle    ;
\draw [shift={(100,110)}, rotate = 0.16] [color={rgb, 255:red, 0; green, 0; blue, 0 }  ][fill={rgb, 255:red, 0; green, 0; blue, 0 }  ][line width=0.75]      (0, 0) circle [x radius= 3.35, y radius= 3.35]   ;
\draw    (232.44,110.38) -- (364.89,110.76) ;
\draw [shift={(303.67,110.58)}, rotate = 180.16] [fill={rgb, 255:red, 0; green, 0; blue, 0 }  ][line width=0.08]  [draw opacity=0] (8.93,-4.29) -- (0,0) -- (8.93,4.29) -- cycle    ;
\draw [shift={(232.44,110.38)}, rotate = 0.16] [color={rgb, 255:red, 0; green, 0; blue, 0 }  ][fill={rgb, 255:red, 0; green, 0; blue, 0 }  ][line width=0.75]      (0, 0) circle [x radius= 3.35, y radius= 3.35]   ;
\draw    (364.89,110.76) -- (497.33,111.14) ;
\draw [shift={(497.33,111.14)}, rotate = 0.16] [color={rgb, 255:red, 0; green, 0; blue, 0 }  ][fill={rgb, 255:red, 0; green, 0; blue, 0 }  ][line width=0.75]      (0, 0) circle [x radius= 3.35, y radius= 3.35]   ;
\draw [shift={(436.11,110.96)}, rotate = 180.16] [fill={rgb, 255:red, 0; green, 0; blue, 0 }  ][line width=0.08]  [draw opacity=0] (8.93,-4.29) -- (0,0) -- (8.93,4.29) -- cycle    ;
\draw [shift={(364.89,110.76)}, rotate = 0.16] [color={rgb, 255:red, 0; green, 0; blue, 0 }  ][fill={rgb, 255:red, 0; green, 0; blue, 0 }  ][line width=0.75]      (0, 0) circle [x radius= 3.35, y radius= 3.35]   ;
\draw    (100,110) .. controls (142.33,61.14) and (184.33,61.14) .. (232.44,110.38) ;
\draw [shift={(171.28,73.86)}, rotate = 184.29] [fill={rgb, 255:red, 0; green, 0; blue, 0 }  ][line width=0.08]  [draw opacity=0] (8.93,-4.29) -- (0,0) -- (8.93,4.29) -- cycle    ;
\draw    (232.44,110.38) .. controls (274.78,61.51) and (316.78,61.51) .. (364.89,110.76) ;
\draw [shift={(303.72,74.24)}, rotate = 184.29] [fill={rgb, 255:red, 0; green, 0; blue, 0 }  ][line width=0.08]  [draw opacity=0] (8.93,-4.29) -- (0,0) -- (8.93,4.29) -- cycle    ;
\draw    (364.89,110.76) .. controls (407.22,61.89) and (449.22,61.89) .. (497.33,111.14) ;
\draw [shift={(436.17,74.62)}, rotate = 184.29] [fill={rgb, 255:red, 0; green, 0; blue, 0 }  ][line width=0.08]  [draw opacity=0] (8.93,-4.29) -- (0,0) -- (8.93,4.29) -- cycle    ;
\draw    (364.89,111.56) .. controls (407.22,166.74) and (449.22,166.74) .. (497.33,111.14) ;
\draw [shift={(436.17,152.37)}, rotate = 175.16] [fill={rgb, 255:red, 0; green, 0; blue, 0 }  ][line width=0.08]  [draw opacity=0] (8.93,-4.29) -- (0,0) -- (8.93,4.29) -- cycle    ;
\draw    (232.44,111.18) .. controls (274.78,166.36) and (316.78,166.36) .. (364.89,110.76) ;
\draw [shift={(303.72,151.99)}, rotate = 175.16] [fill={rgb, 255:red, 0; green, 0; blue, 0 }  ][line width=0.08]  [draw opacity=0] (8.93,-4.29) -- (0,0) -- (8.93,4.29) -- cycle    ;
\draw    (100,111.61) .. controls (142.33,166.79) and (184.33,166.79) .. (232.44,111.18) ;
\draw [shift={(171.28,152.42)}, rotate = 175.16] [fill={rgb, 255:red, 0; green, 0; blue, 0 }  ][line width=0.08]  [draw opacity=0] (8.93,-4.29) -- (0,0) -- (8.93,4.29) -- cycle    ;

\draw (88,129) node [anchor=north west][inner sep=0.75pt]   [align=left] {$\displaystyle -1$};
\draw (222,131) node [anchor=north west][inner sep=0.75pt]   [align=left] {$\displaystyle x_{1}$};
\draw (356,131) node [anchor=north west][inner sep=0.75pt]   [align=left] {$\displaystyle x_{2}$};
\draw (491,132) node [anchor=north west][inner sep=0.75pt]   [align=left] {$\displaystyle 1$};
\draw (152,85) node [anchor=north west][inner sep=0.75pt]   [align=left] {$\displaystyle \Omega $};
\draw (283,85) node [anchor=north west][inner sep=0.75pt]   [align=left] {$\displaystyle \Omega $};
\draw (412,86) node [anchor=north west][inner sep=0.75pt]   [align=left] {$\displaystyle \Omega $};
\draw (150,117) node [anchor=north west][inner sep=0.75pt]   [align=left] {$\displaystyle \overline{\Omega }$};
\draw (281,118) node [anchor=north west][inner sep=0.75pt]   [align=left] {$\displaystyle \overline{\Omega }$};
\draw (411,115) node [anchor=north west][inner sep=0.75pt]   [align=left] {$\displaystyle \overline{\Omega }$};
\draw (155,44) node [anchor=north west][inner sep=0.75pt]   [align=left] {$\displaystyle \Gamma _{1}$};
\draw (283,44) node [anchor=north west][inner sep=0.75pt]   [align=left] {$\displaystyle \Gamma _{2}$};
\draw (414,44) node [anchor=north west][inner sep=0.75pt]   [align=left] {$\displaystyle \Gamma _{3}$};
\draw (150,162) node [anchor=north west][inner sep=0.75pt]   [align=left] {$\displaystyle \overline{\Gamma }_{1}$};
\draw (280,163) node [anchor=north west][inner sep=0.75pt]   [align=left] {$\displaystyle \overline{\Gamma }_{2}$};
\draw (415,164) node [anchor=north west][inner sep=0.75pt]   [align=left] {$\displaystyle \overline{\Gamma }_{3}$};

\end{tikzpicture}

\caption{Regions $\Omega$, $\overline{\Omega }$ and jump contours of the RH problem for $S$ in Case (I)} \label{RHP for S}

\end{figure}

Then, we have the following RH problem for $S$.

\begin{rhp}
 \hfill
 \begin{itemize}
 \item[(a)] $S: \mathbb{C}\setminus \Sigma_S \to \mathbb{C}^{2\times 2}$ is analytic, where $\Sigma_S =\bigcup\limits_{j=1}^3 \Gamma_j  \bigcup\limits_{j=1}^3\overline{\Gamma}_j \cup  [-1, 1] $; see Figure \ref{RHP for S}. 

 \item[(b)] $S$ satisfies the jump condition
\begin{equation}
 S_+(z) = S_-(z) J_S(z), \qquad z \in \Sigma_S
\end{equation} 
with
 \begin{eqnarray} \label{jumps for S}
 J_S(z) = \begin{cases}
J_1(z; \gamma), & z\in \overline{\Gamma}_j, j=1,2,3,\\
J_2(x; \gamma), & -1<x<1,\\
J_3(z; \gamma), & z\in \Gamma_j, j=1,2,3,
 \end{cases}
 \end{eqnarray}
and the matrices $J_j$ are defined in \eqref{eq:JS-jump-1}-\eqref{eq:JS-jump-3}.

 \item[(c)] As $z\to\infty$, $S(z)$ has the following asymptotics
  \[
  S(z) = I+O\left(\frac{1}{z}\right).
  \]

 \item[(d)] As $z\to x_j, j=1, 2$, we have
  \[
  S(z) = O(\log|z-x_j|).
  \]

 \item[(e)]
 For $\alpha<0$, the matrix function $S(z)$ has the following behavior:
 \begin{equation*}
 S(z) =
  O\left(\begin{array}{cc} 1 & |z-1|^\alpha \\ 1 & |z-1|^\alpha \end{array}\right), \text{ as }z\to 1, z\in \mathbb{C}\setminus\Sigma.
 \end{equation*}
 For $\alpha=0$, $S(z)$ has the following behavior:
 \begin{equation*}
 S(z) =
  O\left(\begin{array}{cc} \log|z-1| & \log|z-1| \\ \log|z-1| & \log|z-1| \end{array}\right), \text{ as }z\to 1, z\in \mathbb{C}\setminus\Sigma.
 \end{equation*}
 For $\alpha<0$, $S(z)$ has the following behavior:
 \begin{equation*}
 S(z) = \left\{\begin{array}{ll} O\left(\begin{array}{cc} 1 & 1\\ 1 & 1 \end{array}\right), & \text{as $z\to 1$ outside the lens},\\
 O\left(\begin{array}{cc} |z-1|^{-\alpha} & 1\\ |z-1|^{-\alpha} & 1 \end{array}\right), & \text{as $z\to 1$ inside the lens},
 \end{array}
 \right.
 \end{equation*}

  \item[(f)] $S(z)$ has the same behavior near $-1$ if we replace in (e) $\alpha$ by $\beta$, $|z-1|$ by $|z+1|$ and take the limit $z\to -1$ instead of $z\to 1$.
 \end{itemize}
\end{rhp}


In Case (II), since the two jump points $x_1, x_2$ are close to each other, we do not open the lens along the interval $[x_1,x_2]$. Consequently, the transformation $T \mapsto S$ in \eqref{eq:T-S-map} is defined based on Figure \ref{RHP for S in the merging case}. Note that the jump matrix for $S(x)$ on $[x_1,x_2]$ remains the same as that for $T(x)$ in \eqref{T-jump-2}.

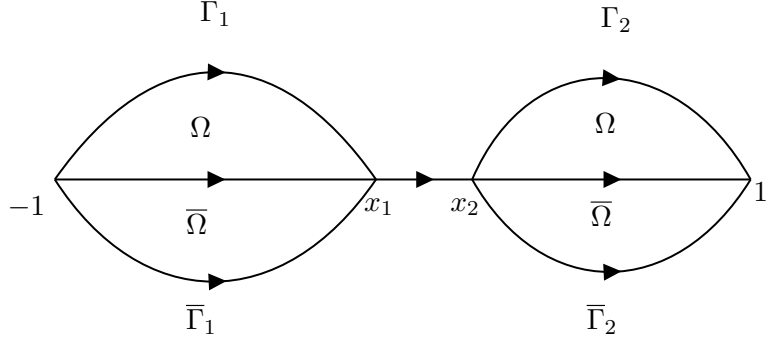
\begin{figure}[htbp]
\centering

\tikzset{every picture/.style={line width=0.75pt}} 

\begin{tikzpicture}[x=0.75pt,y=0.75pt,yscale=-1,xscale=1]

\draw    (100,114) -- (260.33,114) -- (308.33,114) -- (447.33,114) ;
\draw [shift={(185.17,114)}, rotate = 180] [fill={rgb, 255:red, 0; green, 0; blue, 0 }  ][line width=0.08]  [draw opacity=0] (8.93,-4.29) -- (0,0) -- (8.93,4.29) -- cycle    ;
\draw [shift={(289.33,114)}, rotate = 180] [fill={rgb, 255:red, 0; green, 0; blue, 0 }  ][line width=0.08]  [draw opacity=0] (8.93,-4.29) -- (0,0) -- (8.93,4.29) -- cycle    ;
\draw [shift={(382.83,114)}, rotate = 180] [fill={rgb, 255:red, 0; green, 0; blue, 0 }  ][line width=0.08]  [draw opacity=0] (8.93,-4.29) -- (0,0) -- (8.93,4.29) -- cycle    ;
\draw    (100,114) .. controls (150.33,42) and (210.33,42) .. (260.33,114) ;
\draw [shift={(185.25,60.19)}, rotate = 181.87] [fill={rgb, 255:red, 0; green, 0; blue, 0 }  ][line width=0.08]  [draw opacity=0] (8.93,-4.29) -- (0,0) -- (8.93,4.29) -- cycle    ;
\draw    (100,114) .. controls (143.33,182) and (213.33,183) .. (260.33,114) ;
\draw [shift={(185.71,165.08)}, rotate = 177.18] [fill={rgb, 255:red, 0; green, 0; blue, 0 }  ][line width=0.08]  [draw opacity=0] (8.93,-4.29) -- (0,0) -- (8.93,4.29) -- cycle    ;
\draw    (308.33,114) .. controls (337.33,46) and (409.33,46) .. (447.33,114) ;
\draw [shift={(382.39,63.51)}, rotate = 184.41] [fill={rgb, 255:red, 0; green, 0; blue, 0 }  ][line width=0.08]  [draw opacity=0] (8.93,-4.29) -- (0,0) -- (8.93,4.29) -- cycle    ;
\draw    (308.33,114) .. controls (343.33,176) and (409.33,176) .. (447.33,114) ;
\draw [shift={(382.86,160.2)}, rotate = 177.23] [fill={rgb, 255:red, 0; green, 0; blue, 0 }  ][line width=0.08]  [draw opacity=0] (8.93,-4.29) -- (0,0) -- (8.93,4.29) -- cycle    ;

\draw (164,175) node [anchor=north west][inner sep=0.75pt]   [align=left] {$\displaystyle \overline{\Gamma }_{1}$};
\draw (364,175) node [anchor=north west][inner sep=0.75pt]   [align=left] {$\displaystyle \overline{\Gamma }_{2}$};
\draw (171,23) node [anchor=north west][inner sep=0.75pt]   [align=left] {$\displaystyle \Gamma _{1}$};
\draw (371,26) node [anchor=north west][inner sep=0.75pt]   [align=left] {$\displaystyle \Gamma _{2}$};
\draw (166,81) node [anchor=north west][inner sep=0.75pt]   [align=left] {$\displaystyle \Omega $};
\draw (368,79) node [anchor=north west][inner sep=0.75pt]   [align=left] {$\displaystyle \Omega $};
\draw (164,127) node [anchor=north west][inner sep=0.75pt]   [align=left] {$\displaystyle \overline{\Omega }$};
\draw (366,123) node [anchor=north west][inner sep=0.75pt]   [align=left] {$\displaystyle \overline{\Omega }$};
\draw (75,121) node [anchor=north west][inner sep=0.75pt]   [align=left] {$\displaystyle -1$};
\draw (447.33,114) node [anchor=north west][inner sep=0.75pt]   [align=left] {$\displaystyle 1$};
\draw (253,122) node [anchor=north west][inner sep=0.75pt]   [align=left] {$\displaystyle x_{1}$};
\draw (296,122) node [anchor=north west][inner sep=0.75pt]   [align=left] {$\displaystyle x_{2}$};

\end{tikzpicture}

\caption{Regions $\Omega$, $\overline{\Omega }$ and jump contours of the RH problem for $S$ in Case (II)} \label{RHP for S in the merging case}

\end{figure}

In Case (III), there is a single jump point at $x$. Since $x$ is close to one of the endpoints, we open the lens along the interval $[-1,x]$ or $[x, 1]$. The transformation $T \mapsto S$ in \eqref{eq:T-S-map} is then defined according to Figure \ref{RHP for S in the edge}, which depicts the case where $x$ is close to 1. 

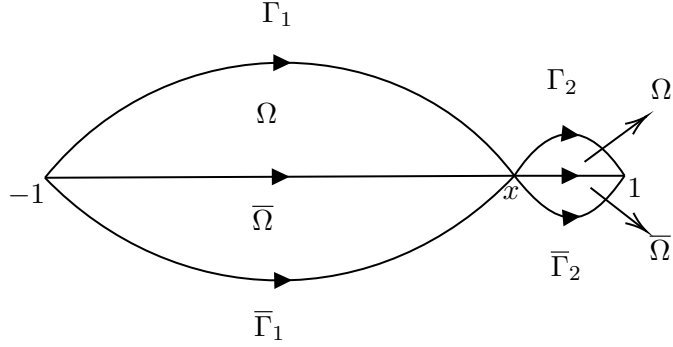
\begin{figure}[htbp]

\centering

\tikzset{every picture/.style={line width=0.75pt}} 

\begin{tikzpicture}[x=0.75pt,y=0.75pt,yscale=-1,xscale=1]

\draw    (100,111) -- (334.33,110) -- (389.33,110) ;
\draw [shift={(222.17,110.48)}, rotate = 179.76] [fill={rgb, 255:red, 0; green, 0; blue, 0 }  ][line width=0.08]  [draw opacity=0] (8.93,-4.29) -- (0,0) -- (8.93,4.29) -- cycle    ;
\draw [shift={(366.83,110)}, rotate = 180] [fill={rgb, 255:red, 0; green, 0; blue, 0 }  ][line width=0.08]  [draw opacity=0] (8.93,-4.29) -- (0,0) -- (8.93,4.29) -- cycle    ;
\draw    (100,111) .. controls (163.33,32) and (279.33,36) .. (334.33,110) ;
\draw [shift={(222.92,53.17)}, rotate = 180.24] [fill={rgb, 255:red, 0; green, 0; blue, 0 }  ][line width=0.08]  [draw opacity=0] (8.93,-4.29) -- (0,0) -- (8.93,4.29) -- cycle    ;
\draw    (100,111) .. controls (164.33,178) and (266.33,182) .. (334.33,110) ;
\draw [shift={(223.36,162.5)}, rotate = 178.85] [fill={rgb, 255:red, 0; green, 0; blue, 0 }  ][line width=0.08]  [draw opacity=0] (8.93,-4.29) -- (0,0) -- (8.93,4.29) -- cycle    ;
\draw    (334.33,110) .. controls (352.33,82) and (371.33,82) .. (389.33,110) ;
\draw [shift={(366.83,89.68)}, rotate = 186.48] [fill={rgb, 255:red, 0; green, 0; blue, 0 }  ][line width=0.08]  [draw opacity=0] (8.93,-4.29) -- (0,0) -- (8.93,4.29) -- cycle    ;
\draw    (334.33,110) .. controls (356.33,138) and (370.33,138) .. (389.33,110) ;
\draw [shift={(367.33,130.39)}, rotate = 174.88] [fill={rgb, 255:red, 0; green, 0; blue, 0 }  ][line width=0.08]  [draw opacity=0] (8.93,-4.29) -- (0,0) -- (8.93,4.29) -- cycle    ;
\draw    (369.33,103) -- (398.73,81.19) ;
\draw [shift={(400.33,80)}, rotate = 143.43] [color={rgb, 255:red, 0; green, 0; blue, 0 }  ][line width=0.75]    (10.93,-3.29) .. controls (6.95,-1.4) and (3.31,-0.3) .. (0,0) .. controls (3.31,0.3) and (6.95,1.4) .. (10.93,3.29)   ;
\draw    (372.33,116) -- (398.76,136.76) ;
\draw [shift={(400.33,138)}, rotate = 218.16] [color={rgb, 255:red, 0; green, 0; blue, 0 }  ][line width=0.75]    (10.93,-3.29) .. controls (6.95,-1.4) and (3.31,-0.3) .. (0,0) .. controls (3.31,0.3) and (6.95,1.4) .. (10.93,3.29)   ;

\draw (207,21) node [anchor=north west][inner sep=0.75pt]   [align=left] {$\displaystyle \Gamma _{1}$};
\draw (349,55) node [anchor=north west][inner sep=0.75pt]   [align=left] {$\displaystyle \Gamma _{2}$};
\draw (203,176) node [anchor=north west][inner sep=0.75pt]   [align=left] {$\displaystyle \overline{\Gamma }_{1}$};
\draw (351,146) node [anchor=north west][inner sep=0.75pt]   [align=left] {$\displaystyle \overline{\Gamma }_{2}$};
\draw (204,72) node [anchor=north west][inner sep=0.75pt]   [align=left] {$\displaystyle \Omega $};
\draw (401,60) node [anchor=north west][inner sep=0.75pt]   [align=left] {$\displaystyle \Omega $};
\draw (202,123) node [anchor=north west][inner sep=0.75pt]   [align=left] {$\displaystyle \overline{\Omega }$};
\draw (400.33,138) node [anchor=north west][inner sep=0.75pt]   [align=left] {$\displaystyle \overline{\Omega }$};
\draw (80,113) node [anchor=north west][inner sep=0.75pt]   [align=left] {$\displaystyle -1$};
\draw (327,114) node [anchor=north west][inner sep=0.75pt]   [align=left] {$\displaystyle x$};
\draw (389.33,110) node [anchor=north west][inner sep=0.75pt]   [align=left] {$\displaystyle 1$};

\end{tikzpicture}

\caption{Regions $\Omega$, $\overline{\Omega }$ and jump contours of the RH problem for $S$ in Case (III)} \label{RHP for S in the edge}

\end{figure}

\subsection{The Global Parametrix}

By the definition of $\varphi(z)$ in \eqref{varphi}, one can see $|\varphi(z)| > 1$ for $z$ bounded away from $[-1, 1]$. As a consequence, all the jumps for $S$ in \eqref{jumps for S} tend to the identity matrix exponentially as $N\to\infty$. This gives us the following global parametrix.

\begin{rhp}\label{rhp for Pinfty}
\hfill
\begin{itemize}
\item[(a)] $P^\infty: \mathbb{C}\setminus [-1, 1] \to \mathbb{C}^{2\times 2}$ is analytic.

\item[(b)] $P^\infty$ satisfies the following jump relations
 \begin{eqnarray*}
 \begin{array}{lr}
 P^\infty_+(x) = 
 \left\{
 \begin{array}{ll}
 P^\infty_-(x) J_2(x; \gamma_1+\gamma_2), & -1<x<x_1;\\
 P^\infty_-(x) J_2(x; \gamma_2), & x_1<x<x_2;\\
 P^\infty_-(x) J_2(x; 0), & x_2 < x <1,
 \end{array}
 \right.
 \end{array}
 \end{eqnarray*}
 where $J_2(x; \gamma)$ is defined in \eqref{eq:JS-jump-2}.

 \item[(c)] As $z\to\infty$, we have
 \begin{equation}
 P^{\infty}(z) = I + O(z^{-1}).
 \end{equation}

 \item[(d)] As $x\to x_j$ for $j=1, 2$, we have
 \begin{equation}
 P^{\infty}(z) = O(\log |z-x_j|).
 \end{equation}

 \item[(e)] As $z\to 1$, we have
 \begin{equation}
 P^{\infty}(z) = O\left(
                    \begin{array}{cc}
                      |z - 1|^{-1/4} & |z - 1|^{-1/4} \\
                      |z - 1|^{-1/4} & |z - 1|^{-1/4} \\
                    \end{array}
                  \right) (z-1)^{-\frac{\alpha}{2}\sigma_3}.
 \end{equation}

 \item[(e)] As $z\to -1$, we have
 \begin{equation}
 P^{\infty}(z) = O\left(
                    \begin{array}{cc}
                      |z + 1|^{-1/4} & |z + 1|^{-1/4} \\
                      |z + 1|^{-1/4} & |z + 1|^{-1/4} \\
                    \end{array}
                  \right) (z+1)^{-\frac{\beta}{2}\sigma_3}.
 \end{equation}
\end{itemize}
\end{rhp}

When $\gamma = 0$, the above RH problem is indeed the global parametrix studied in \cite[Sec. 5]{ABJ2004}, where the explicit solution is provided. The case $\gamma \neq 0$ is also considered in \cite[Sec. 5.4]{CG2021}. In our case, the solution is explicitly given by
\begin{equation}\label{Pinfty}
P^\infty (z) = D_{\infty}^{\sigma_3} Q(z) D(z)^{-\sigma_3},
\end{equation}
where
\begin{eqnarray}\label{Q}
Q(z) = \left(
         \begin{array}{cc}
           \frac{1}{2}(a(z) + a(z)^{-1}) & -\frac{1}{2i}(a(z) - a(z)^{-1}) \\
           \frac{1}{2i}(a(z) - a(z)^{-1}) & \frac{1}{2}(a(z) + a(z)^{-1}) \\
         \end{array}
       \right)
\end{eqnarray}
and
\begin{equation}
a(z) = \left(\frac{z+1}{z-1}\right)^{1/4}, \quad z\in \mathbb{C} \setminus [-1, 1]
\end{equation}
with $a(z) \sim 1$ as $z\to\infty$. The Szeg\H{o} function  $D(z)$ takes the form of
\begin{equation}\label{D}
D(z) = D_w(z) D_{t}(z) D_{\gamma}(z),
\end{equation}
where
\begin{eqnarray}
D_w(z) &=& \frac{(z-1)^{\alpha/2}(z+1)^{\beta/2}}{\varphi(z)^{(\alpha+\beta)/2}}, \\
D_{t}(z) &=&  \exp\left(\frac{(z^2-1)^{1/2}}{2\pi} \int_{-1}^1 \frac{t(x)}{\sqrt{1-x^2}} \frac{dx}{z-x}\right), \\
D_\gamma (z) &=& \exp\left(\frac{(z^2-1)^{1/2}}{\sqrt{2}}\sum_{j=1}^2 \gamma_j \int_{-1}^{x_j} \frac{1}{\sqrt{1-x^2}} \frac{dx}{z-x}\right) \label{Dgamma}
\end{eqnarray}
and $\displaystyle D_{\infty} = \lim_{z\to\infty} D(z)$ is a constant.

In Case (I), where the two jump points are separated and not close to the endpoints, the steepest descent analysis has been conducted in \cite{CG2021}. Therefore, we will apply the results in \cite{CG2021} to express the asymptotics of the Hankel determinant $D_N(x_1, x_2; \gamma_1, \gamma_2; w)$ in Proposition \ref{Asymptotics in the Separated Regime} below.

We will study Case (II) and (III) in more details in the following two subsections. In Case (II), where the two jump points $x_1$ and $x_2$ are close to each other, we do not open the lens along the interval $[x_1,x_2]$. Nevertheless, we keep the same global parametrix defined in \eqref{Pinfty} and construct a local parametrix near $x_1$ in terms of the Painlev\'{e} V functions. This local parametrix is defined in a neighbourhood that also encloses $x_2$. In Case (III), where only one jump point $x$ exists, the global parametrix in \eqref{Pinfty} is modified by setting $x_1=x_2=x$ and $\gamma_1+\gamma_2 = \gamma$. The local parametrix near $x$ will be constructed using confluent hypergeometric functions.

\subsection{Local Parametrices in the Merging Case}

In this subsection, we construct local parametrices for the merging case, i.e. the Case (II) illustrated in Figure \ref{RHP for S in the merging case}.
Let $\delta>0$ be a fixed positive number and $B(z_0, \delta):= \{ z \,| \, |z-z_0| < \delta\}$ be a neighbourhood of a given point $z_0$. We first consider the endpoints $\pm 1$ and look for a function $P(z)$ satisfying the following RH problem in $B(\pm 1, \delta)$.

\begin{rhp} \label{sec3.4-RHP1}
\hfill
\begin{itemize}
\item[(a)] $P: B(\pm 1, \delta) \setminus ([-1, 1] \cup \Gamma_1 \cup \overline{\Gamma_1} \cup \Gamma_2 \cup \overline{\Gamma_2})\to \mathbb{C}^{2\times 2}$ is analytic.
\item[(b)] $P$ satisfies the same jump relations as $S$ on $B(\pm 1, \delta)\cap ([-1, 1] \cup \Gamma_1 \cup \overline{\Gamma_1} \cup \Gamma_2 \cup \overline{\Gamma_2})$.
\item[(c)] For $z\in \partial B(\pm 1, \delta)$, we have the matching condition
\begin{equation} \label{sec3.3-matching1}
P(z)P^\infty(z)^{-1} = I+O(1/N) \quad \text{ as } N\to\infty.
\end{equation}
\item[(d)] As $z\to\pm 1$, $P(z)$ has the same asymptotic behaviours as $S(z)$.
\end{itemize}
\end{rhp}

It is well-known that the above RH problems can be solved in terms of the Bessel parametrix in \eqref{Phi-B-solution}; see the explicit construction in \cite[Sec. 6]{ABJ2004}.

We will focus on the local parametrix near $x_1$. Let us consider a neighbourhood of $B(x_1, \delta)$ with $\delta > 0$ , which encloses the other jump point $x_2$. We look for a function $P_{x_1}(z)$ satisfying the following RH problem in $B(x_1, \delta)$.

\begin{rhp} \label{rhp:local-merge}
\hfill
\begin{itemize}
\item[(a)] $P_{x_1}: B(x_1, \delta)\setminus ([-1, 1] \cup \Gamma_1 \cup \overline{\Gamma_1} \cup \Gamma_2 \cup \overline{\Gamma_2}) \to \mathbb{C}^{2\times 2}$ is analytic.
\item[(b)] $P_{x_1}$ satisfies the same jump condition as $S$ in $B(x_1, \delta) \cap (\mathbb{R}\cup \Gamma_1 \cup \overline{\Gamma_1} \cup \Gamma_2 \cup \overline{\Gamma_2})$.
\item[(c)] For $z\in\partial B(x_1, \delta)$, we have the matching condition
\begin{equation} \label{eq: p-p-infty-math}
P_{x_1}(z)P^\infty(z)^{-1} = I + O((\delta N)^{-1})
\end{equation}
as $N\to\infty$, uniformly for all relevant parameters and for $0<y\leq \delta / 2$.
\item[(d)] As $z\to x_j$, we have
\begin{equation}
P_{x_1}(z) = O(\log(z-x_j)).
\end{equation}
\end{itemize}
\end{rhp}

Using a similar idea as in \cite[Sec. 7]{CFL2021}, we solve the above RH problem explicitly in terms of the Painlev\'{e} V functions. First, recalling the definition of $\varphi(z)$ in \eqref{varphi}, let us introduce the following conformal mapping near $z = x_1$: 
\begin{equation} \label{lambdayz}
\lambda_y(z) = \pm \frac{2N}{s_{N, y}}(\log \varphi(z) - \log \varphi_{\pm}(x_1)), \quad \pm\Im z>0,
\end{equation}
where $s_{N, y}$ is the quantity measuring the distance between $x_2$ and $x_1$, defined as 
\begin{eqnarray}  \label{sNy} 
s_{N, y} =  2N(\log \varphi_+(x_2) - \log \varphi_+(x_1)), \qquad \textrm{with } y = x_2- x_1.
\end{eqnarray}
From the definition in \eqref{lambdayz}, it is easy to see that $\lambda_y(z)$ maps the two jump points $x_1$ and $x_2$ to $0$ and $1$, respectively. Moreover, for $x \in (-1, 1)\cap B(x_1, \delta)$,  we have
\begin{eqnarray*}
\lambda_{y, +}(x) = \frac{2N}{s_{N, y}}(\log \varphi_{+}(x) - \log \varphi_+(x_1)) = -\frac{2N}{s_{N, y}}(\log \varphi_{-}(x) - \log \varphi_{-}(x_1)) =\lambda_{y, -}(x).
\end{eqnarray*}
This implies that $\lambda_y(z)$ is analytic in $B(x_1, \delta)$. Moreover, it follows from \eqref{sNy} that there exists a constant $0<c<\delta$ such that
\begin{eqnarray}
s_{N, y} = -2Nyi\frac{1}{\sqrt{1-x_1^2}} + O(Ny^2), \quad \textrm{as } N\to\infty,
\end{eqnarray}
uniformly for $0< y \leq c$. This gives us
\begin{equation}
\lambda'_y(x_1) = \frac{2N}{s_{N, y}}(\log \varphi_+)'(x_1) = \frac{1}{y} + O(1).
\end{equation}
Then, the solution to the above RH problem is given below.

\begin{lemma}
Let $\varphi(z)$ and $w_J(z)$ be defined in \eqref{varphi} and \eqref{def:wJz},  and $\widehat{\Phi}_{PV}$ be the Painlev\'e V parametrix given in Appendix \ref{Appendix: Section: PV model RHP}. Then, the solution to RH problem \ref{rhp:local-merge} is given by
\begin{equation}\label{Pmerge}
P_{x_1}(z) = E_{N, y} (z) \widehat{\Phi}_{PV}(\lambda_y(z); s_{N, y})w_J(z)^{-\sigma_3 / 2}\varphi(z)^{-N\sigma_3},
\end{equation}
where $E_{N, y}(z)$ is an analytic function in $B(x_1, \delta)$ given as 
\begin{equation}\label{Emerge}
E_{N, y}(z) = P^\infty(z)w_J(z)^{\sigma_3 / 2} \widehat{\Phi}^\infty(\lambda_y(z))^{-1} \varphi_{+}(x_1)^{-N\sigma_3},
\end{equation}
with 
\begin{equation}\label{hatpsiinfty}
\widehat{\Phi}^{\infty}(\lambda) = \widehat{\Phi}^{\infty}(\lambda; s) = \left\{\begin{array}{ll}
e^{-\sqrt{2}\pi\gamma_2\sigma_3} |s|^{\frac{\gamma_1+\gamma_2}{\sqrt{2}i}\sigma_3} \lambda^{\frac{\gamma_1}{\sqrt{2}i}\sigma_3} (1-\lambda)^{\frac{\gamma_2}{\sqrt{2}i}\sigma_3}, & \Im \lambda<0, \\
|s|^{\frac{\gamma_1+\gamma_2}{\sqrt{2}i}\sigma_3} \lambda^{\frac{\gamma_1}{\sqrt{2}i}\sigma_3} (1-\lambda)^{\frac{\gamma_2}{\sqrt{2}i}\sigma_3}\sigma_3\sigma_1, & \Im \lambda>0.
\end{array}
\right.
\end{equation}
In the above formula, the principal branches are chosen such that $\arg \lambda, \arg(1-\lambda) \in (-\pi,\pi)$.
\end{lemma}
\begin{proof}
Let us first show the analyticity of $E_{N, y}$ in $B(x_1, \delta)$. From its definition, we only need to verify the analytic property of $E_{N, y}(z)$ across the interval $[-1,1] \cap B(x_1, \delta)$. It is straightforward to see that $E_{N, y}(x)_{-}^{-1} E_{N, y}(x)_+ = I$ for $ x \in (x_1-\delta, x_1+\delta)$. Furthermore, as $z \to x_1$, we have $\lambda_y(z) \to 0$. Since $\gamma_1$ is real, the factor $\lambda^{\frac{\gamma_1}{\sqrt{2}i}\sigma_3} $ in \eqref{hatpsiinfty} remains bounded as $z \to x_1$. Consequently, we find that $E_{N, y}(z) = O(1)$ as $z\to x_1$, indicating that $z = x_1$ is a removable singularity. Hence, $E_{N, y}(z)$ is analytic in $B(x_1, \delta)$.

Next, with \eqref{jump hatPsi} and the analyticity of $E_{N, y}$, it is easy to check that $P_{x_1}(z)$ constructed in \eqref{Pmerge} satisfies the same jump conditions as $S(z)$ in $B(x_1, \delta)$. 

Our final task is to verify the matching condition \eqref{eq: p-p-infty-math}. By combining \eqref{Pmerge} and \eqref{Emerge}, we obtain
\begin{equation*}
P_{x_1}(z) P^{\infty}(z)^{-1}  =  E_{N, y} (z) \widehat{\Phi}_{PV}(\lambda_y(z); s_{N, y})\varphi(z)^{-N\sigma_3} \widehat{\Phi}^\infty(\lambda_y(z))^{-1} \varphi_{+}(x_1)^{-N\sigma_3} E_{N, y}(z)^{-1}.
\end{equation*}
When $\Im z < 0$, note that $\widehat{\Phi}^{\infty}(\lambda(z))$ is a diagonal matrix (cf. \eqref{hatpsiinfty}), which implies that the products in $\varphi(z)^{-N\sigma_3} \widehat{\Phi}^\infty(\lambda_y(z))^{-1} \varphi_{+}(x_1)^{-N\sigma_3}$ commute. Moreover, it follows from \eqref{phi+phi-} and \eqref{lambdayz} that
\begin{eqnarray*}
\varphi(z)^{-N\sigma_3} \varphi_{+}(x_1)^{-N\sigma_3} = \varphi(z)^{-N\sigma_3} \varphi_{-}(x_1)^{N\sigma_3} = e^{\frac{s_{N, y}}{2} \lambda_y(z) \sigma_3}.
\end{eqnarray*}
Using the above two formulas, we get, for $\Im z < 0$,
\begin{eqnarray}
P_{x_1}(z) P^{\infty}(z)^{-1} 
= E_{N, y} (z) \widehat{\Phi}_{PV}(\lambda_y(z); s_{N, y}) e^{\frac{s_{N, y}}{2} \lambda_y(z) \sigma_3} \widehat{\Phi}^\infty(\lambda_y(z))^{-1} E_{N, y}(z)^{-1}.
\end{eqnarray}
For $\Im z > 0$, using the fact $\sigma_1 a^{\sigma_3} = a^{-\sigma_3} \sigma_1$, a similar computation gives us
\begin{eqnarray}
P_{x_1}(z) P^{\infty}(z)^{-1} = E_{N, y} (z) \widehat{\Phi}_{PV}(\lambda_y(z); s_{N, y}) e^{-\frac{s_{N, y}}{2} \lambda_y(z) \sigma_3} \widehat{\Phi}^\infty(\lambda_y(z))^{-1} E_{N, y}(z)^{-1}.
\end{eqnarray}
To approximate $\widehat{\Phi}_{PV}(\lambda_y(z); s_{N, y}) e^{\pm \frac{s_{N, y}}{2} \lambda_y(z) \sigma_3} \widehat{\Phi}^\infty(\lambda_y(z))^{-1}$ in the formulas above, we adopt an argument similar to that in \cite[Lemma 7.1]{CFL2021} to obtain
\begin{equation}
\widehat{\Phi}_{PV}(\lambda_y(z); s_{N, y}) e^{\mp \frac{s_{N, y}}{2} \lambda_y(z) \sigma_3} \widehat{\Phi}^\infty(\lambda_y(z))^{-1} = I+O((N\delta)^{-1}), \qquad N \to \infty,
\end{equation}
uniformly for $z\in \partial B(x_1, \delta)$, with $0<y< \delta/2$ and $x_1$ in a fixed compact subset of $(-1, 1)$. Regarding the behavior of $E_{N, y}(z)$ for $z\in \partial B(x_1, \delta)$, both $N$-depend terms, $\widehat{\Phi}^\infty(\lambda_y(z))$ and $\varphi_{+}(x_1)^{-N\sigma_3}$, are bounded as $N \to \infty$. Since $P^\infty (z)$ and $w_J(z)$ have no singularities on $\partial B(x_1, \delta)$, we conclude from the definition in \eqref{Emerge} that $E_{N, y}(z)^{\pm 1} = O(1)$ as $N \to \infty$ uniformly for $z \in \partial B(x_1, \delta)$. This, together with the formulas above, provides us with the desired matching condition \eqref{eq: p-p-infty-math}.

This finishes the proof of the lemma.
\end{proof}

With all the global and local parametrices constructed, we define the final transformation as follows:
\begin{eqnarray} \label{merging-R-def}
R(z) = \left\{\begin{array}{ll}
                S(z)P(z)^{-1} & z\in B(-1, \delta)\cup B(1, \delta), \\
                S(z)P_{x_1}(z)^{-1} & z\in B(x_1, \delta), \\
                S(z)P^\infty(z)^{-1} & z\in\mathbb{C}\setminus (B(x_1, \delta)\cup B(-1, \delta)\cup B(1, \delta)).
              \end{array}
\right.
\end{eqnarray}
It is straightforward to see that $R(z)$ satisfies the following RH problem.

\begin{rhp}
\hfill
\begin{itemize}
\item[(a)] $R: \mathbb{C}\setminus \Sigma_R \to \mathbb{C}^{2\times 2}$ is analytic, where 
\begin{equation}
\Sigma_R = \Sigma_S \cup \partial B(\pm 1, \delta) \cup \partial B(x_1, \delta) \setminus \{ [-1, 1] \cup B(\pm 1, \delta) \cup B(x_1, \delta)\}.
\end{equation}
The orientations on $\partial B(\pm 1, \delta)$ and $\partial B(x_1, \delta)$ are taken to be clockwise.

\item[(b)] $R_+(z) = R_-(z)J_R(z)$, where
\begin{eqnarray}
J_R(z) = \left\{\begin{array}{ll}
                P(z)P^\infty(z)^{-1} & z\in  \partial B(\pm 1, \delta) \\
                P_{x_1}(z)P^\infty(z)^{-1} & z\in \partial B(x_1, \delta) \\
                P^\infty(z)J_S(z)P^\infty(z)^{-1} & z\in\Sigma_R\setminus (B(x_1, \delta)\cup B(-1, \delta)\cup B(1, \delta))
              \end{array}
\right.
\end{eqnarray}
\item[(c)] As $z\to\infty$, we have
\begin{equation}
R(z) = I+O(z^{-1}).
\end{equation}
\end{itemize}
\end{rhp}

From \eqref{jumps for S}, one can see that, there exists a positive constant $c>0$ such that
\begin{equation*}
J_{R}(z) =  I+O\left(e^{-2c N}\right), \qquad z\in\Sigma_R\setminus (B(x_1, \delta)\cup B(-1, \delta)\cup B(1, \delta)).
\end{equation*}
For $z \in \partial B(x_1, \delta)\cup \partial B(\pm 1, \delta)$, it follows from \eqref{sec3.3-matching1} and \eqref{eq: p-p-infty-math} that
\begin{eqnarray}
J_{R}(z) =  I+O\left(\frac{1}{N \delta}\right), \quad z \in \partial B(x_1, \delta),  \qquad J_{R}(z) = I+O\left(\frac{1}{N}\right), \quad z \in \partial B(\pm 1, \delta).
\end{eqnarray}
Then, by the standard result for small-norm RH problems (see \cite[Section 7]{DeiftZhou}), we have
\begin{eqnarray} \label{Asym for R in merging case}
R(z) =  I + O\left(\frac{1}{N(|z|+1)}\right) \quad \textrm{and} \quad  R'(z) = O\left(\frac{1}{N(|z|+1)}\right), \quad \text{ as } N \to \infty,
\end{eqnarray}
uniformly for $z \in \mathbb{C} \setminus \Sigma_R$.

\subsection{Local Parametrices in the Edge Region} \label{sec:edge-analysis}

In this subsection, we construct local parametrix in the edge regime, i.e., the Case (III) illustrated in Figure \ref{RHP for S in the edge}. Without loss of generality, we only consider the case that the jump point $x$ is close to $1$. First, we look for a function $P_{-1}(z)$ satisfying the following RH problem in $B(-1, \delta)$.

\begin{rhp} \label{rhp for local parametrix near -1 in the edge case}
\hfill
\begin{itemize}
\item[(a)] $P_{-1}: B(-1, \delta) \setminus ([-1, 1] \cup \Gamma_1 \cup \overline{\Gamma}_1)\to \mathbb{C}^{2\times 2}$ is analytic.
\item[(b)] $P_{-1}$ satisfies the same jump relations as $S$ on $B(-1, \delta)\cap ([-1, 1] \cup \Gamma_1 \cup \overline{\Gamma}_1)$.
\item[(c)] For $z\in \partial B(-1, \delta)$, we have the matching condition
\begin{equation}\label{matching condition: P-1Pinfty}
P_{-1}(z)P^\infty(z)^{-1} = I+O(1/N), \qquad \textrm{as } N\to\infty.
\end{equation}

\item[(d)] As $z\to - 1$, $P_{-1}(z)$ has the same asymptotic behaviours as $S(z)$.
\end{itemize}
\end{rhp}

The construction of the local parametrix near $-1$ is the same as that in RH problem \ref{sec3.4-RHP1}.  As a consequence, the solution is also given in terms of the Bessel parametrix in Section \ref{section: RHP for Bessel}.

Next, we focus on the local parametrix near $z=1$, which encloses the jump point $x$. We look for a function $P_1(z)$ satisfying the following RH problem in $B(1, \delta)$.

\begin{rhp} \label{rhp of the local parametrix for the edge regime near 1}
\hfill
\begin{itemize}
\item[(a)] $P_{1}: B(1, \delta) \setminus ([-1, 1] \cup \Gamma_1 \cup \overline{\Gamma}_1 \cup \Gamma_2 \cup \overline{\Gamma}_2)\to \mathbb{C}^{2\times 2}$ is analytic.
\item[(b)] $P_{1}$ satisfies the same jump relations as $S$ on $B(1, \delta)\cap ([-1, 1] \cup \Gamma_1 \cup \overline{\Gamma}_1 \cup \Gamma_2 \cup \overline{\Gamma}_2)$.
\item[(c)] For $z\in \partial B(1, \delta)$, we have the matching condition
\begin{equation} \label{mathcing condition near 1}
P_{1}(z)P^\infty(z)^{-1} = I+O(1/N), \qquad \text{ as } N\to\infty.
\end{equation}

\item[(d)] As $z\to 1$, $P_{1}(z)$ has the same asymptotic behaviours as $S(z)$.
\end{itemize}
\end{rhp}

In order to solve the above RH problem, we need to construct a model RH problem.

\subsubsection{The Model RH Problem}\label{modelrhp}

\begin{figure}[htbp]

\centering

\tikzset{every picture/.style={line width=0.75pt}} 

\begin{tikzpicture}[x=0.75pt,y=0.75pt,yscale=-1,xscale=1]

\draw    (100,122) -- (300.17,122.57) ;
\draw [shift={(300.17,122.57)}, rotate = 0.16] [color={rgb, 255:red, 0; green, 0; blue, 0 }  ][fill={rgb, 255:red, 0; green, 0; blue, 0 }  ][line width=0.75]      (0, 0) circle [x radius= 3.35, y radius= 3.35]   ;
\draw [shift={(205.08,122.3)}, rotate = 180.16] [fill={rgb, 255:red, 0; green, 0; blue, 0 }  ][line width=0.08]  [draw opacity=0] (8.93,-4.29) -- (0,0) -- (8.93,4.29) -- cycle    ;
\draw    (300.17,122.57) -- (500.33,123.14) ;
\draw [shift={(500.33,123.14)}, rotate = 0.16] [color={rgb, 255:red, 0; green, 0; blue, 0 }  ][fill={rgb, 255:red, 0; green, 0; blue, 0 }  ][line width=0.75]      (0, 0) circle [x radius= 3.35, y radius= 3.35]   ;
\draw [shift={(405.25,122.87)}, rotate = 180.16] [fill={rgb, 255:red, 0; green, 0; blue, 0 }  ][line width=0.08]  [draw opacity=0] (8.93,-4.29) -- (0,0) -- (8.93,4.29) -- cycle    ;
\draw [shift={(300.17,122.57)}, rotate = 0.16] [color={rgb, 255:red, 0; green, 0; blue, 0 }  ][fill={rgb, 255:red, 0; green, 0; blue, 0 }  ][line width=0.75]      (0, 0) circle [x radius= 3.35, y radius= 3.35]   ;
\draw    (191.33,14.14) -- (300.17,122.57) ;
\draw [shift={(249.29,71.88)}, rotate = 224.89] [fill={rgb, 255:red, 0; green, 0; blue, 0 }  ][line width=0.08]  [draw opacity=0] (8.93,-4.29) -- (0,0) -- (8.93,4.29) -- cycle    ;
\draw    (185.33,227.14) -- (300.17,122.57) ;
\draw [shift={(246.45,171.49)}, rotate = 137.68] [fill={rgb, 255:red, 0; green, 0; blue, 0 }  ][line width=0.08]  [draw opacity=0] (8.93,-4.29) -- (0,0) -- (8.93,4.29) -- cycle    ;
\draw    (300.17,122.57) .. controls (351.33,61.14) and (437.33,59.14) .. (500.33,123.14) ;
\draw [shift={(405.5,76.15)}, rotate = 182.72] [fill={rgb, 255:red, 0; green, 0; blue, 0 }  ][line width=0.08]  [draw opacity=0] (8.93,-4.29) -- (0,0) -- (8.93,4.29) -- cycle    ;
\draw    (300.17,122.57) .. controls (347.33,180.14) and (448.33,181.14) .. (500.33,123.14) ;
\draw [shift={(405.22,166.07)}, rotate = 178.76] [fill={rgb, 255:red, 0; green, 0; blue, 0 }  ][line width=0.08]  [draw opacity=0] (8.93,-4.29) -- (0,0) -- (8.93,4.29) -- cycle    ;

\draw (288,143) node [anchor=north west][inner sep=0.75pt]   [align=left] {$\displaystyle -1$};
\draw (494,143) node [anchor=north west][inner sep=0.75pt]   [align=left] {$\displaystyle 0$};
\draw (228,29) node [anchor=north west][inner sep=0.75pt]   [align=left] {$\displaystyle \Gamma _{\Phi ,1}$};
\draw (222,197) node [anchor=north west][inner sep=0.75pt]   [align=left] {$\displaystyle \Gamma _{\Phi ,1}$};
\draw (382,44) node [anchor=north west][inner sep=0.75pt]   [align=left] {$\displaystyle \Gamma _{\Phi ,2}$};
\draw (376,176) node [anchor=north west][inner sep=0.75pt]   [align=left] {$\displaystyle \Gamma _{\Phi ,2}$};

\end{tikzpicture}

\caption{The jump contours of the RH problem for $\Phi$} \label{Figure: Model RHP}

\end{figure}
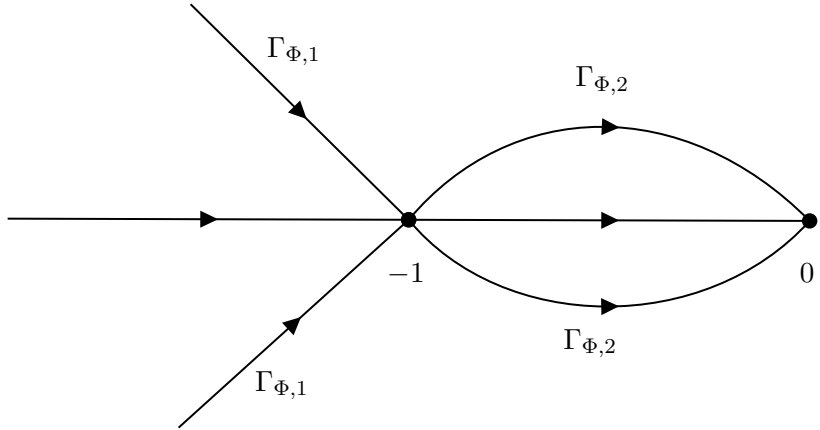

This model RH problem is similar to that in \cite[Sec. 6.2]{CFL2021}, but with a slightly different contour. This difference arises because the endpoint $z=1$ is a hard edge in the present work, in contrast to the soft edge in \cite{CFL2021}.

\begin{rhp} \label{model rhp for Phi}
\hfill
\begin{itemize}
\item[(a)] $\Phi = \Phi(\lambda; u)$ is analytic on $\mathbb{C}\setminus ((-\infty, 0]\cup \Gamma_{\Phi, 1}\cup \Gamma_{\Phi, 2})$; see Figure \ref{Figure: Model RHP}.

\item[(b)] On $(-\infty, 0]\cup \Gamma_{\Phi, 1}\cup \Gamma_{\Phi, 2}\setminus \{-1, 0\}$, $\Phi$ satisfies the following jump conditions:
    \begin{eqnarray} \label{Model-jump-Phi}
    \begin{array}{ll}
    \Phi_+(\lambda) = 
    \left\{
    \begin{array}{ll}
    \Phi_-(\lambda)\left(\begin{array}{ll} 1 & 0\\ e^{-\sqrt{2}\pi\gamma}e^{-4(\lambda u)^{1/2}}e^{\pm \alpha\pi i} & 1
     \end{array}\right) & \text{ for } \lambda\in\Gamma_{\Phi, 1}, \text{ and } \lambda \in \mathbb{C}^{\pm}, \\
     \Phi_-(\lambda)\left(\begin{array}{ll} 1 & 0\\ e^{-4(\lambda u)^{1/2}}e^{\pm \alpha\pi i} & 1
     \end{array}\right) & \text{ for } \lambda\in\Gamma_{\Phi, 2}, \text{ and } \lambda \in \mathbb{C}^{\pm}, \\
     \Phi_-(\lambda)\left(\begin{array}{ll} 0 & e^{\sqrt{2}\pi\gamma}\\ -e^{-\sqrt{2}\pi\gamma} & 0
     \end{array}\right) & \text{ for } \lambda\in (-\infty, -1),\\
     \Phi_-(\lambda)\left(\begin{array}{ll} 0 & 1\\ -1 & 0
     \end{array}\right) & \text{ for } \lambda \in (-1, 0),
     \end{array}
     \right.
    \end{array}
    \end{eqnarray}
   where principal branches are chosen in the square roots, and $u>0$ and $\gamma\in\mathbb{R}$ are parameters.

\item[(c)] As $\lambda\to\infty$,
\begin{equation}\label{Phiasym}
\Phi(\lambda) = (I+O(\lambda^{-1}))\left(\begin{array}{cc} 1 & 0\\ i\sqrt{2}\gamma & 1\end{array}\right)\lambda^{-\frac{1}{4}\sigma_3} B e^{-\sqrt{2}\pi \frac{\gamma}{2}\sigma_3},
\end{equation}
where the principal branches are chosen, and
\begin{equation}\label{A}
B = \frac{1}{\sqrt 2}\left(\begin{array}{cc} 1 & i\\ i & 1\end{array}\right).
\end{equation}

\item[(d)] As $\lambda\to -1$, $\Phi(\lambda) = O(\log |\lambda+1|)$.

\item[(e)] For $\alpha<0$, $\Phi(\lambda)$ has the following behaviour:
\begin{eqnarray}
\Phi(\lambda) = O\left(
                   \begin{array}{cc}
                     |\lambda|^{\alpha / 2} & |\lambda|^{\alpha / 2} \\
                     |\lambda|^{\alpha / 2} & |\lambda|^{\alpha / 2} \\
                   \end{array}
                 \right), \text{ as } \lambda\to 0.
\end{eqnarray}
For $\alpha=0$, $\Phi(\lambda)$ has the following behaviour:
\begin{eqnarray}
\Phi(\lambda) = O\left(
                   \begin{array}{cc}
                     \log |\lambda| & \log |\lambda| \\
                     \log |\lambda| & \log |\lambda| \\
                   \end{array}
                 \right), \text{ as } \lambda\to 0.
\end{eqnarray}
For $\alpha>0$, $\Phi(\lambda)$ has the following behaviour:
\begin{eqnarray}
\Phi(\lambda) = \left\{\begin{array}{ll}
O\left(
   \begin{array}{cc}
     |\lambda|^{\alpha/2} & |\lambda|^{-\alpha/2} \\
     |\lambda|^{\alpha/2} & |\lambda|^{-\alpha/2} \\
   \end{array}
 \right), & \text{as } \lambda\to 0 \text{ outside the lens,}\\
O\left(
   \begin{array}{cc}
     |\lambda|^{-\alpha/2} & |\lambda|^{-\alpha/2} \\
     |\lambda|^{-\alpha/2} & |\lambda|^{-\alpha/2} \\
   \end{array}
 \right), & \text{as } \lambda\to 0 \text{ inside the lens,}
\end{array}
\right.
\end{eqnarray}
\end{itemize}
\end{rhp}

Compared to classical parametrices, the above RH problem appears more complicated. Fortunately, to obtain asymptotics of the related Hankel determinants, an explicit expression for $\Phi(\lambda; u)$ is unnecessary; only its existence for sufficiently large $u$ is required. This existence can be established by analyzing the asymptotic behavior of the RH problem as $u \to +\infty$, which will be the focus of the rest of this subsection.

\smallskip
\noindent{\bf Global parametrix}
\smallskip

As $u \to +\infty$, the jump matrices in \eqref{Model-jump-Phi} tend to the identity matrix exponentially except for those on $(-\infty, 0]$. This gives us the following global parametrix.

\begin{rhp} 
\hfill
\begin{itemize}
\item[(a)] $M(\lambda)$ is analytic on $\mathbb{C}\setminus (-\infty, 0]$.

\item[(b)] $M$ satisfies the following jump conditions:
    \begin{eqnarray} 
     M_+(\lambda) = 
     \left\{
     \begin{array}{ll}
     M_-(\lambda)\left(\begin{array}{ll} 0 & e^{\sqrt{2}\pi\gamma}\\ -e^{-\sqrt{2}\pi\gamma} & 0
     \end{array}\right) \qquad   &\text{ for } \lambda\in (-\infty, -1),\\
     M_-(\lambda)\left(\begin{array}{ll} 0 & 1\\ -1 & 0
     \end{array}\right) \qquad   &\text{ for } \lambda\in(-1, 0). 
     \end{array}
     \right.
     \label{model-global-M-jump2}
    \end{eqnarray}

\item[(c)] As $\lambda\to\infty$,
\begin{equation} 
M(\lambda) = (I+O(\lambda^{-1}))\left(\begin{array}{cc} 1 & 0\\ i\sqrt{2}\gamma & 1\end{array}\right)\lambda^{-\frac{1}{4}\sigma_3} B e^{-\sqrt{2}\pi \frac{\gamma}{2}\sigma_3}.
\end{equation}

\end{itemize}
\end{rhp}

The solution to the above RH problem is given explicitly as follows
\begin{equation}\label{M}
M(\lambda) = \lambda^{-\frac{1}{4}\sigma_3} B\left(\frac{1+e^{-\pi i/2}\sqrt{\lambda}}{\sqrt{\lambda+1}}\right)^{-i\sqrt{2}\gamma \sigma_3}, \quad \lambda \in \mathbb{C}\setminus (-\infty, 0],
\end{equation}
where $B$ is defined in \eqref{A} and all root functions take their principal branches.

\smallskip
\noindent{\bf Local parametrix near $\lambda = 0$}
\smallskip

We look for a function $\hat{P}_0$ satisfying the following RH problem in $B(0, \delta)$.

\begin{rhp} \label{model rhp local parametrix near 0}
\hfill
\begin{itemize}
\item[(a)] $\hat{P}_{0}: B(0, \delta) \setminus ([-1, 0] \cup \Gamma_{\Phi,2})\to \mathbb{C}^{2\times 2}$ is analytic.
\item[(b)] $\hat{P}_{0}$ satisfies the same jump relations as $\Phi$ on $B(0, \delta)\cap ([-1, 0] \cup \Gamma_{\Phi, 2})$.
\item[(c)] For $z\in \partial B(0, \delta)$, we have the matching condition
\begin{equation} \label{sec3.5.1-matching1}
\hat{P}_{0}(\lambda)M(\lambda)^{-1} = I+O(u^{-1/2}),  \qquad \textrm{as } u \to +\infty.
\end{equation}

\item[(d)] As $\lambda\to 0$, $\hat{P}_{0}(\lambda)$ has the same asymptotic behaviours as $\Phi(z)$.
\end{itemize}
\end{rhp}

The solution to the above RH problem is given in terms of the Bessel parametrix in Section \ref{Appendix: Model RHP for Bessel}. More precisely, we have the following result.

\begin{lemma}
With $\Phi_{\text{Bes}}^{(\alpha)}$ defined in \eqref{Phi-B-solution}, the solution to RH problem \ref{model rhp local parametrix near 0} is given by
\begin{equation}\label{P0}
\hat{P}_0(\lambda) = \hat{E}_0(\lambda)\Phi_{\text{Bes}}^{(\alpha)}(\lambda u)e^{-2(\lambda u)^{1/2}\sigma_3},
\end{equation}
where $\hat{E}_0(\lambda)$ is an analytic function in $B(0, \delta)$ given as 
\begin{equation}\label{hatE0}
\hat{E}_0(\lambda) = M(\lambda)B^{-1}(4\pi^2\lambda u)^{\frac{1}{4} \sigma_3},
\end{equation}
the matrices $M(\lambda)$ and $B$ are defined in \eqref{A} and \eqref{M}, respectively, and all root functions take their principal branches.
\end{lemma}

\begin{proof}
From \eqref{model-global-M-jump2} and \eqref{hatE0}, it follows that $\hat{E}_{0,+}(\lambda) = \hat{E}_{0,-}(\lambda)$ for $\lambda \in (-\delta, 0)$. Moreover, since $\hat{E}_0(\lambda) = O(\lambda^{-1/2})$ as $\lambda\to 0$, we conclude that $\hat{E}_0(\lambda) $ is analytic in $B(0, \delta)$. Next, recalling \eqref{eq:Besl-infty}, it is easy to see that $\hat{P}_0(\lambda)$ satisfies the same jump conditions as $\Phi$ on $B(0, \delta)\cap ([-1, 0] \cup \Gamma_{\Phi, 2})$.

Finally, using \eqref{eq:Besl-infty}, we have, as $u \to +\infty$,
\begin{eqnarray}
\hat{P}_0(\lambda) M(\lambda)^{-1} =  M(\lambda) B^{-1} \frac{1}{\sqrt{2}}
\left(
\begin{array}{cc}
    1+O((\lambda u)^{-1/2}) & i+O((\lambda u)^{-1/2}) \\
    i+O((\lambda u)^{-1/2}) & 1+O((\lambda u)^{-1/2}) \\
\end{array}
\right) M(\lambda)^{-1}
\end{eqnarray}
for $\lambda \in \partial B(0, \delta)$. With the definition of $M(\lambda)$ in \eqref{M}, we obtain \eqref{sec3.5.1-matching1}, which completes the proof of the lemma.
\end{proof}

\smallskip
\noindent{\bf Local parametrix near $\lambda = -1$}
\smallskip

We look for a function $\hat{P}_{-1}$ satisfying the following RH problem in $B(-1, \delta)$.

\begin{rhp} \label{model rhp local parametrix near -1}
\hfill
\begin{itemize}
\item[(a)] $\hat{P}_{-1}: B(-1, \delta) \setminus ((-\infty, 0] \cup \Gamma_{\Phi,1} \cup \Gamma_{\Phi, 2})\to \mathbb{C}^{2\times 2}$ is analytic.
\item[(b)] $\hat{P}_{-1}$ satisfies the same jump relations as $\Phi$ on $B(-1, \delta)\cap ((-\infty, 0] \cup \Gamma_{\Phi, 2} \cup \Gamma_{\Phi, 1})$.
\item[(c)] For $z\in \partial B(-1, \delta)$, we have the matching condition
\begin{equation} \label{sec3.5.1-matching2}
\hat{P}_{-1}(\lambda)M(\lambda)^{-1} = I+O(u^{-1/2}),  \qquad \textrm{as } u \to +\infty.
\end{equation}

\item[(d)] As $\lambda\to -1$, $\hat{P}_{-1}(\lambda)$ has the same asymptotic behaviours as $\Phi(z)$.
\end{itemize}
\end{rhp}

The solution to the above RH problem is given in terms of the confluent hypergeometric parametrix in Appendix \ref{Appendix: Section: HG model RHP}. To state the solution explicitly, we first define
\begin{equation}\label{hu}
h_u(\lambda) = 4e^{-\pi i/2} \left((\lambda u)^{1/2}-e^{-\frac{\pi i}{2}} u^{1/2}\right), \qquad \lambda \in \mathbb{C} \setminus (0, i \infty),
\end{equation}
with $\lambda^{1/2}$ is taken to be positive for $\lambda >0$. With this branch choice, we have $(-1)^{1/2} = e^{-\frac{\pi i}{2}}$ and $h_u(-1) = 0$. Furthermore, we get  
\begin{equation}\label{hasymp}
h_u(\lambda) = 2u^{1/2}(\lambda+1)\left(1+\frac{1}{4}(\lambda+1) + O((\lambda+1)^2)\right), \quad \textrm{as }  \lambda \to -1.
\end{equation}
This shows that $h_u(\lambda) $ defines a conformal mapping on $B(-1, \delta)$, which maps the region $\Im (\lambda +1) > 0$ to $\Im h_u > 0$.

Then, the solution to the above RH problem is given below.

\begin{lemma} \label{the solution to the model rhp}
Let $h_u$ be defined in \eqref{hu}, and define $\chi(\lambda)$ as
\begin{eqnarray}\label{chi}
\chi(\lambda) = \left\{\begin{array}{ll}
e^{\alpha \pi i /2}, & \Im \lambda > 0;\\
e^{-\alpha \pi i /2}, & \Im \lambda < 0.
\end{array}
\right.
\end{eqnarray}
Let $\Phi_{\text{HG}}$ be the solution to the model RH problem \ref{Appendix: HG model RHP}. Then, the solution to RH problem \ref{model rhp local parametrix near -1} is given by
\begin{equation}\label{P-1}
\hat{P}_{-1}(\lambda) = \hat{E}_{-1}(\lambda) \Phi_{\text{HG}}\left(h_u(\lambda); \beta=\frac{\gamma}{\sqrt{2}i}\right) e^{-2(\lambda u)^{1/2}\sigma_3}e^{-\sqrt{2}\pi\frac{\gamma}{4}\sigma_3} \chi(\lambda)^{\sigma_3},
\end{equation}
where $\hat{E}_{-1}$ is an analytic function in $B(-1, \delta)$ given as 
\begin{eqnarray}\label{E-1}
\hat{E}_{-1}(\lambda) = \left\{\begin{array}{ll} M(\lambda)e^{-\frac{\pi \alpha i}{2} \sigma_3} h_u(\lambda)^{\frac{\gamma}{\sqrt{2}i}\sigma_3} e^{2iu^{1/2}\sigma_3}, & \Im \lambda>0; \\ M(\lambda)\left(\begin{array}{cc} 0 & 1\\ -1 & 0 \end{array}\right)e^{-\frac{\pi \alpha i}{2} \sigma_3} h_u(\lambda)^{\frac{\gamma}{\sqrt{2}i}\sigma_3} e^{2iu^{1/2}\sigma_3}, & \Im\lambda <0,\end{array}\right. 
\end{eqnarray}
the matrix $M(\lambda)$ is defined in \eqref{M}, and all root functions take their principal branches.
\end{lemma}

\begin{proof}
Let us first show the analyticity of $\hat{E}_{-1}(\lambda)$ in $B(-1, \delta)$. From its definition, we only need to verify that $\hat{E}_{-1}(\lambda)$ is analytic across the interval $(-\infty, -1] \cap B(-1, \delta)$. With the jump conditions of $M(\lambda)$ in \eqref{model-global-M-jump2}, it is easy to see that $\hat{E}_{-1}(x)_{-}^{-1} \hat{E}_{-1}(x)_+ = I$ across $(-1-\delta, -1+\delta)$. Furthermore, as $\lambda \to -1$, we have $h_u(\lambda) \to 0$. Since $\gamma$ is real, the factor $h_u(\lambda)^{\frac{\gamma}{\sqrt{2}i}\sigma_3}$ in \eqref{E-1} remains bounded as $\lambda \to -1$. Similar arguments can be used to show that $M(\lambda)$ is also bounded as $\lambda\to -1$. As a result, $\hat{E}_{-1}(\lambda)$ is of order $O(1)$ as $\lambda \to -1$, indicating that $\lambda = -1$ is a removable singularity. Hence $\hat{E}_{-1}(\lambda)$ is analytic in $B(-1, \delta)$. 

Next, recalling the jump conditions for $\Phi_{HG}$ in \eqref{HJumps}, a direct computation shows that $\hat{P}^{-1}(\lambda)$ satisfies the same jump conditions as $\Phi(\lambda)$ in $B(-1, \delta)$. 

Last, we will verify the matching condition \eqref{sec3.5.1-matching2}. Let us first consider the case where $\Im \lambda>0$. Recalling the asymptotics of $\Phi_{\text{HG}}$ in \eqref{Appendix: model RHP HG asymptotics}, we have from \eqref{P-1} that
\begin{eqnarray}
& & \hat{P}_{-1}(\lambda) M(\lambda)^{-1} =  M(\lambda) e^{-\frac{\pi \alpha i}{2}\sigma_3} h_u(\lambda)^{\frac{\gamma}{\sqrt{2}i}\sigma_3} e^{2iu^{1/2}\sigma_3}e^{-\sqrt{2}\pi\frac{\gamma}{4}\sigma_3} (I + O(h_u(\lambda)^{-1})) \nonumber\\
& & \qquad \qquad  \times h_u(\lambda)^{-\frac{\gamma}{\sqrt{2}i}\sigma_3} e^{-\frac{i h_u(\lambda)}{2}\sigma_3} e^{\frac{\sqrt{2}\pi\gamma}{2}\sigma_3} e^{-2(\lambda u)^{1/2}\sigma_3} e^{-\sqrt{2}\pi\frac{\gamma}{4}\sigma_3} e^{\frac{\pi \alpha i}{2}\sigma_3} M(\lambda)^{-1}. \label{P-1M-1}
\end{eqnarray}
Note that in the definition of $h_u(\lambda)$ in \eqref{hu}, we choose the branch cut of $\lambda^{1/2}$ along $(0, i \infty)$, while the principal branch is used in the formula above. This gives us
\begin{equation}
-\frac{i h_u(\lambda)}{2} = 2(\lambda u)^{1/2} - 2iu^{1/2}, \quad \lambda \in B(-1, \delta) \cap \Im \lambda > 0.
\end{equation}
Combining \eqref{hasymp} with the above two formulas, we get, as $u \to +\infty$,
\begin{multline}
\hat{P}_{-1}(\lambda) M(\lambda)^{-1} =  M(\lambda) e^{-\frac{\pi \alpha i}{2}\sigma_3} h_u(\lambda)^{\frac{\gamma}{\sqrt{2}i}\sigma_3} e^{2iu^{1/2}\sigma_3} e^{-\sqrt{2}\pi\frac{\gamma}{4}\sigma_3} (I + O(u^{-1/2})) \\
 \times h_u(\lambda)^{-\frac{\gamma}{\sqrt{2}i}\sigma_3} e^{-2iu^{1/2}\sigma_3}  e^{\sqrt{2}\pi\frac{\gamma}{4}\sigma_3} e^{\frac{\pi \alpha i}{2}\sigma_3} M(\lambda)^{-1}, 
\end{multline}
for $\lambda \in B(-1, \delta) \cap \Im \lambda > 0 $. Since $h_u(\lambda)^{\frac{\gamma}{\sqrt{2}i}\sigma_3} e^{2iu^{1/2}\sigma_3}$ is bounded as $u \to +\infty$, the above formula yields \eqref{sec3.5.1-matching2} for $\Im \lambda > 0$. The derivation for the case $\Im \lambda < 0 $ is similar, and we omit the details here.

This finishes the proof of the lemma.
\end{proof}

\smallskip
\noindent{\bf Small-norm RH problem}
\smallskip

Now we define the final transformation as
\begin{eqnarray}\label{Xi}
\Xi(\lambda) = \left\{\begin{array}{ll} \Phi(\lambda)M^{-1}(\lambda) & \text{for } \lambda\in \mathbb{C}\setminus (B(0, \delta)\cup B(-1, \delta)),\\
\Phi(\lambda)\hat{P}_0^{-1}(\lambda) & \text{for } \lambda\in B(0, \delta),\\
\Phi(\lambda)\hat{P}_{-1}^{-1}(\lambda) & \text{for } \lambda\in B(-1, \delta). \end{array}\right.
\end{eqnarray}
It is straightforward to see that $\Xi(\lambda)$ satisfies the following RH problem.

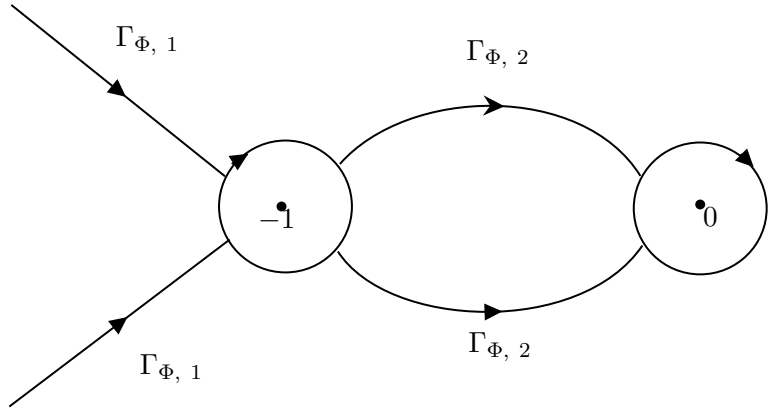
\begin{figure}[htbp]

\centering

\tikzset{every picture/.style={line width=0.75pt}} 

\begin{tikzpicture}[x=0.75pt,y=0.75pt,yscale=-1,xscale=1]

\draw   (236,140.17) .. controls (236,121.85) and (250.85,107) .. (269.17,107) .. controls (287.48,107) and (302.33,121.85) .. (302.33,140.17) .. controls (302.33,158.48) and (287.48,173.33) .. (269.17,173.33) .. controls (250.85,173.33) and (236,158.48) .. (236,140.17) -- cycle ;
\draw    (242.33,118.8) -- (250.33,113.8) ;
\draw [shift={(250.57,113.65)}, rotate = 147.99] [fill={rgb, 255:red, 0; green, 0; blue, 0 }  ][line width=0.08]  [draw opacity=0] (8.93,-4.29) -- (0,0) -- (8.93,4.29) -- cycle    ;
\draw  [color={rgb, 255:red, 0; green, 0; blue, 0 }  ,draw opacity=1 ][fill={rgb, 255:red, 0; green, 0; blue, 0 }  ,fill opacity=1 ] (265.21,140.17) .. controls (265.21,139.07) and (266.1,138.19) .. (267.19,138.19) .. controls (268.28,138.19) and (269.17,139.07) .. (269.17,140.17) .. controls (269.17,141.26) and (268.28,142.14) .. (267.19,142.14) .. controls (266.1,142.14) and (265.21,141.26) .. (265.21,140.17) -- cycle ;
\draw   (443.02,141.14) .. controls (443.02,122.83) and (457.87,107.98) .. (476.19,107.98) .. controls (494.51,107.98) and (509.36,122.83) .. (509.36,141.14) .. controls (509.36,159.46) and (494.51,174.31) .. (476.19,174.31) .. controls (457.87,174.31) and (443.02,159.46) .. (443.02,141.14) -- cycle ;
\draw  [color={rgb, 255:red, 0; green, 0; blue, 0 }  ,draw opacity=1 ][fill={rgb, 255:red, 0; green, 0; blue, 0 }  ,fill opacity=1 ] (474.21,139.17) .. controls (474.21,138.07) and (475.1,137.19) .. (476.19,137.19) .. controls (477.28,137.19) and (478.17,138.07) .. (478.17,139.17) .. controls (478.17,140.26) and (477.28,141.14) .. (476.19,141.14) .. controls (475.1,141.14) and (474.21,140.26) .. (474.21,139.17) -- cycle ;
\draw    (296.33,118.8) .. controls (330.33,79.8) and (417.33,77.8) .. (446.33,124.8) ;
\draw [shift={(378.53,89.76)}, rotate = 182.01] [fill={rgb, 255:red, 0; green, 0; blue, 0 }  ][line width=0.08]  [draw opacity=0] (10.72,-5.15) -- (0,0) -- (10.72,5.15) -- (7.12,0) -- cycle    ;
\draw    (295.33,162.8) .. controls (320.33,201.8) and (416.33,205.8) .. (447.33,159.8) ;
\draw [shift={(377.46,192.95)}, rotate = 178.06] [fill={rgb, 255:red, 0; green, 0; blue, 0 }  ][line width=0.08]  [draw opacity=0] (8.93,-4.29) -- (0,0) -- (8.93,4.29) -- cycle    ;
\draw    (496.33,113.8) -- (502.33,119.8) ;
\draw [shift={(502.87,120.34)}, rotate = 225] [fill={rgb, 255:red, 0; green, 0; blue, 0 }  ][line width=0.08]  [draw opacity=0] (8.93,-4.29) -- (0,0) -- (8.93,4.29) -- cycle    ;
\draw    (132.33,38.8) -- (239,125) ;
\draw [shift={(189.56,85.04)}, rotate = 218.94] [fill={rgb, 255:red, 0; green, 0; blue, 0 }  ][line width=0.08]  [draw opacity=0] (8.93,-4.29) -- (0,0) -- (8.93,4.29) -- cycle    ;
\draw    (131.33,240.8) -- (241.33,156.8) ;
\draw [shift={(190.31,195.77)}, rotate = 142.63] [fill={rgb, 255:red, 0; green, 0; blue, 0 }  ][line width=0.08]  [draw opacity=0] (8.93,-4.29) -- (0,0) -- (8.93,4.29) -- cycle    ;

\draw (254,140) node [anchor=north west][inner sep=0.75pt]   [align=left] {$\displaystyle -1$};
\draw (476.19,139.17) node [anchor=north west][inner sep=0.75pt]   [align=left] {$\displaystyle 0$};
\draw (358,55) node [anchor=north west][inner sep=0.75pt]   [align=left] {$\displaystyle \Gamma _{\Phi ,\ 2}$};
\draw (359,202) node [anchor=north west][inner sep=0.75pt]   [align=left] {$\displaystyle \Gamma _{\Phi ,\ 2}$};
\draw (183,48) node [anchor=north west][inner sep=0.75pt]   [align=left] {$\displaystyle \Gamma _{\Phi ,\ 1}$};
\draw (195,213) node [anchor=north west][inner sep=0.75pt]   [align=left] {$\displaystyle \Gamma _{\Phi ,\ 1}$};

\end{tikzpicture}

\caption{The jump contours of the RH problem for $\Xi$} \label{Figure: RHP for Xi}

\end{figure}

\begin{rhp} \label{rhp:xi}
\hfill
\begin{itemize}
\item[(a)] $\Xi: \mathbb{C}\setminus \Sigma_{\Xi}\to\mathbb{C}$ is analytic, where $\Sigma_{\Xi}$ and its orientations are depicted in Figure \ref{Figure: RHP for Xi}.
\item[(b)] $
\Xi_+(\lambda) = \Xi_-(\lambda) J_{\Xi}(\lambda),
$
with
\begin{eqnarray} \label{eq:j-xi-jump}
{\small
J_{\Xi}(\lambda) = \left\{\begin{array}{ll}
\hat{P}_0(\lambda)M(\lambda)^{-1}, & \lambda\in\partial B(0, \delta),\\
\hat{P}_{-1}(\lambda)M(\lambda)^{-1}, & \lambda\in\partial B(-1, \delta),\\
M(\lambda)\left(\begin{array}{cc} 1 & 0\\ e^{-\sqrt{2}\pi\gamma} e^{-4(\lambda u)^{1/2}} e^{\pm \alpha \pi i} & 1\end{array}\right)M(\lambda)^{-1}, & \lambda\in \Gamma_{\Phi, 1} \setminus \overline{B(-1, \delta)},\\
M(\lambda)\left(\begin{array}{cc} 1 & 0\\
e^{-4(\lambda u)^{1/2}} e^{\pm \alpha \pi i} & 1\end{array}\right)M(\lambda)^{-1}, &  \lambda\in \Gamma_{\Phi, 2} \setminus (\overline{B(-1, \delta)} \cup \overline{B(0, \delta)}).
\end{array}\right.}
\end{eqnarray}
\item[(c)] As $\lambda\to\infty$,
\begin{equation}
\Xi(\lambda) = I+O(\lambda^{-1}).
\end{equation}
\end{itemize}
\end{rhp}

It is easy to see from \eqref{eq:j-xi-jump} that, there exists a positive constant $c>0$ such that
\begin{equation}
J_{\Xi}(\lambda)= I+O(e^{-cu^{1/2}}), \quad u \to +\infty
\end{equation}
for $\lambda\in\Sigma_{\Xi}\setminus (\partial B(0, \delta) \cup\partial B(-1, \delta))$. Moreover, the matching conditions in \eqref{sec3.5.1-matching1} and \eqref{sec3.5.1-matching2} give us
\begin{equation}
J_{\Xi}(\lambda) = I+O(u^{-1/2}), \quad u \to +\infty
\end{equation}
for $\lambda\in\partial B(0, \delta) \cup\partial B(-1, \delta)$. Thus, the RH problem \ref{rhp:xi} is a small-norm RH problem. Consequently, its solution $\Xi(\lambda)$ exists and is unique for sufficiently large $u$. Furthermore, we have
\begin{eqnarray}
\Xi(\lambda) = I+O((1+|\lambda|)^{-1}u^{-1/2})  \quad \textrm{and} \quad 
\Xi'(\lambda) = O((1+|\lambda|)^{-1}u^{-1/2}),\qquad u\to+\infty \label{Xiasymp1}
\end{eqnarray}
uniformly for $\lambda\in\mathbb{C}\setminus \Sigma_{\Xi}$. With the transformation in \eqref{Xi}, we also conclude that the solution to the model RH problem \ref{model rhp for Phi} exists for sufficiently large $u$.

\subsubsection{Construction of the Local Parametrix near $1$}

Using the model RH problem constructed in the previous subsection, we now solve the RH problem \ref{rhp of the local parametrix for the edge regime near 1}.

We first define
\begin{eqnarray}\label{f}
f(z) = \frac{1}{4}(\log \varphi(z))^2,
\end{eqnarray}
where $\varphi(z)$ is given in \eqref{varphi}. A direct computation yields
\begin{eqnarray} \label{eq:f-expan}
f(z) = \frac{1}{2}(z-1) - \frac{1}{12}(z-1)^2 + O((z-1)^3), \quad \text{ as } z\to 1,
\end{eqnarray}
which shows that $f(z)$ is a conformal map in a neighborhood of $z=1$. Denote
\begin{equation} \label{def:un-x}
u_{N, x} = -N^2 f(x) = -\frac{N^2}{4} (\log \varphi(z))^2,
\end{equation}
where $x$ is the jump point in Case (III); see Figure \ref{RHP for S in the edge} for an illustration. Obviously, we have 
\begin{equation}\label{u}
u_{N, x} = O(N^2(1-x)), \quad \text{ as } N\to\infty,
\end{equation}
uniformly for $x\in (1-\varepsilon, 1-N^{-2}\log\log N)$, so that $u_{N, x} \to \infty$ as $N\to \infty$ for $x$ in this regime.

We further define
\begin{equation}\label{lambdadef}
\lambda_{x}(z) = -\frac{f(z)}{f(x)},
\end{equation}
which is also a conformal map near $z = 1$, mapping $z=1$ to $0$ and $z =x$ to $-1$, respectively. Moreover, we have the following approximation
\begin{equation}\label{lambda}
\lambda_{x}(z) = O((1-x)^{-1}), \quad \text{ as } x\to 1,
\end{equation}
uniformly for $z\in\partial B(1, \delta)$.  We also define 
\begin{equation} \label{W(z)}
W(z) = ((z-1)^\alpha (z+1)^\beta e^{t(z)})^{1/2}, \quad \text{ for } z\in \mathbb{C} \setminus (-\infty, 1],
\end{equation}
where the branch of the square root is chosen such that $W(z)>0$ for $z>1$. With this definition, we have
\begin{eqnarray} \label{eq:W-wj-relation}
W^2(z) = \left\{\begin{array}{ll}
                  e^{\alpha\pi i} w_J(z), & \Im z>0; \\
                  e^{-\alpha\pi i} w_J(z), & \Im z<0,
                \end{array}
\right.
\end{eqnarray}
where $w_J(z)$ is given in \eqref{def:wJz}.

Now we choose the jump contour $\Sigma_S$ of $S$ near $1$ so that  $\lambda_x$ maps $\Sigma_S\cap B(1, \delta)$ to the jump contour $\Gamma_{\Phi,1}$ and $\Gamma_{\Phi, 2}$ from the RH problem \ref{model rhp for Phi}. The solution to the RH problem \ref{rhp of the local parametrix for the edge regime near 1} is given by the following lemma.

\begin{lemma}
Let $\lambda_x(z)$ and $W(z)$ be defined in \eqref{lambdadef} and \eqref{W(z)},  and $\Phi$ be the solution to the model RH problem \ref{model rhp for Phi}. Then, the solution to RH problem \ref{rhp of the local parametrix for the edge regime near 1} is given by
\begin{equation}\label{P}
P_{1}(z) = E_{1}(z)\Phi(\lambda_x(z), u_{N, x}) W(z)^{-\sigma_3},
\end{equation}
where  $E_{1}(z)$ is an analytic function in $B(1, \delta)$ defined as
\begin{equation}\label{E}
E_{1}(z) = P^{\infty}(z) W(z)^{\sigma_3} M(\lambda_x(z))^{-1},
\end{equation}
with $P^{\infty}(z)$ and $M(\lambda)$ given in \eqref{Pinfty} and \eqref{M}, respectively. 
\end{lemma}

\begin{proof}
We first show that $E_1(z)$ is analytic in $B(1, \delta)$. From its definition, we only need to verify the analytic property of $E_1(z)$ across $(-1, 1)\cap B(1, \delta)$. Noting that $P^\infty(z)$ and $M(\lambda_x(z))$ satisfy the same jumping conditions in $B(1, \delta)$, we find that $E_{1}(z)_{-}^{-1}E_1(z)_{+} = I$ for $z\in (1-\delta, 1+\delta)$. Furthermore, as $z\to 1$, $\lambda_x(z) \to 0$, which implies that $E_1(z)$ is of order $O((z-1)^{-1/4+\alpha /2})$; as $z\to x$, $E_1(z)$ is of order $O(1)$. As a result, both $x$ and $1$ are removable singularities. So we conclude that $E_{1}(z)$ is analytic in $B(1, \delta)$. 

Next, with the jumping conditions of $\Phi(\lambda, u)$ in \eqref{Model-jump-Phi}, it is straightforward to check that $P(z)$ satisfies the same jump conditions as $S(z)$ in $B(1, \delta)$.

So, it remains to show that $P_1(z)$ satisfies the matching condition \eqref{mathcing condition near 1}. From the definition of $\Xi(\lambda)$ in \eqref{Xi} and the asymptotics \eqref{Xiasymp1}, one has
\begin{equation}
\Phi(\lambda)M(\lambda)^{-1} = I+O(u^{-1/2}\lambda^{-1}), \qquad u \to + \infty
\end{equation}
for $\lambda$ sufficiently large. With the property of $u_{N, x}$ in \eqref{u}, we have
\begin{equation}
\Phi(\lambda_x(z), u_{N, x})M(\lambda_x(z))^{-1} = I+O(u_{N, x}^{-1/2}\lambda_x(z)^{-1}), \qquad N \to  \infty,
\end{equation}
uniformly for $z\in \partial B(1, \delta)$ and $x \in (1-\varepsilon, 1-N^{-2}\log\log N)$. Then, by the definition of $P_1(z)$ in \eqref{P}, we have
{\small
\begin{eqnarray*}
P_1(z)P^\infty(z)^{-1} = P^\infty(z)W(z)^{\sigma_3}M^{-1}(\lambda_x(z)) \Phi(\lambda_x(z), u_{N, x}) M(\lambda_x(z))^{-1} M(\lambda_x(z))W(z)^{-\sigma_3}P^\infty(z)^{-1} \nonumber \\
= P^\infty(z)W(z)^{\sigma_3}M^{-1}(\lambda_x(z))\left(I+O\left(u_{N, x}^{-1/2}\lambda_x(z)^{-1}\right)\right)M(\lambda_x(z)) W(z)^{-\sigma_3} P^\infty(z)^{-1}.
\end{eqnarray*}}
Using \eqref{u} and \eqref{lambda}, we get from the above formula, as $N\to\infty$,
\begin{equation*}
P_1(z)P^\infty(z)^{-1} = P^\infty(z)W(z)^{\sigma_3}M^{-1}(\lambda_x(z))\left(I+O\left(N^{-1} (1-x)^{1/2} \right)\right)M(\lambda_x(z)) W(z)^{-\sigma_3} P^\infty(z)^{-1}.
\end{equation*}
uniformly for $z\in \partial B(1, \delta)$ and $x \in (1-\varepsilon, 1-N^{-2}\log\log N)$. By \eqref{lambda} and the definition of $M$ in \eqref{M}, one can see
\begin{equation}
M(\lambda_x(z))^{\pm 1} = O(|1-x|^{-1/4}), \qquad \textrm{as } x \to 1,
\end{equation}
uniformly for $z\in\partial B(1, \delta)$. Combining the above two formulas, we have 
\begin{equation}\label{PPinfty-1}
P_1(z)P^\infty(z)^{-1} = I + O(N^{-1}), \qquad \text{ as } N\to \infty,
\end{equation}
uniformly for $x\in (1-\delta, 1-N^{-2}\log\log N)$ and $z\in \partial B(1, \delta)$. This concludes the proof of the lemma.
\end{proof}

\subsubsection{The Small-norm RH Problem}
With all the global and local parametrices constructed, we define the final transformation as follows:
\begin{eqnarray}\label{transform map: edge S to R}
R(z) = \left\{\begin{array}{ll} S(z)P^\infty(z)^{-1} & \text{ for } z\in \mathbb{C}\setminus (\Gamma_S\cup B(1, \delta)\cup B(-1, \delta)),\\
S(z)P_{1}(z)^{-1} & \text{ for } z\in B(1, \delta), \\
S(z)P_{-1}(z)^{-1} & \text{ for } z\in B(-1, \delta).
\end{array}\right.
\end{eqnarray}
Then, $R(z)$ satisfies the following RH problem.

\begin{rhp}
\hfill
\begin{itemize}
\item[(a)] $R: \mathbb{C}\setminus \Sigma_R \to \mathbb{C}$ is analytic, where
\begin{equation}
\Sigma_R = \Sigma_S \cup \partial B(\pm 1, \delta) \setminus \{[-1, 1] \cup B(\pm 1, \delta)\}
\end{equation}
and the orientations on $\partial B(\pm 1, \delta)$ are taken to be clockwise.
\item[(b)] On $\Sigma_R$, $R$ satisfies the following jump conditions
\begin{eqnarray}
R_+(z) = 
\left\{
\begin{array}{ll}
R_{-}(z) 
    \left(
      \begin{array}{cc}
        1 & 0 \\
        e^{-\sqrt{2}\pi\gamma} w(z)^{-1} \varphi(z)^{-2N} & 1 \\
      \end{array}
    \right), & z\in \Gamma_1 \cup \overline{\Gamma}_1, \\
R_{-}(z) P_{-1}(z)P^{\infty}(z)^{-1}, & z\in B(-1, \delta),\\
R_{-}(z) P_1(z)P^{\infty}(z)^{-1}, & z\in B(1, \delta).
\end{array}
\right.
\end{eqnarray}
\item[(c)] As $z\to\infty$,
\begin{equation}
R(z) = I+O(z^{-1}).
\end{equation}
\end{itemize}
\end{rhp}

Recalling the definition of $\varphi(z)$ in \eqref{varphi}, and the matching property in \eqref{matching condition: P-1Pinfty} and \eqref{PPinfty-1},  it is direct to see that the jumps on $\Gamma_1$ and $\overline{\Gamma}_1$ are $I + O(e^{-N})$ as $N\to\infty$; and the jumps on $\partial B(-1, \delta) \cup \partial B(1, \delta)$ are $I + O(1/N)$ uniformly for $1-\varepsilon < x < 1 - N^{-2}\log\log N$. By the standard result small-norm RH problems, it follows that
\begin{equation}\label{Rasymp}
R(z) = I+O(N^{-1}) \quad \textrm{and} \quad R'(z) = O(N^{-1}), \qquad \textrm{as } N\to\infty, 
\end{equation}
uniformly for $z\in \mathbb{C}\setminus \Gamma_R$ and $1-\varepsilon < x < 1-N^{-2}\log\log N$. Here $\varepsilon > 0$ is a constant that can be sufficiently small.

\section{Asymptotics of Hankel Determinants} \label{Sec:Asy-hankle}

With the steepest descent analysis conducted in the previous section, we are now ready to derive the asymptotics for the corresponding Hankel determinants using the differential identities provided in Proposition \ref{prop:diff-identity}. One may compare the asymptotic results in Propositions \ref{Asymptotics in the Separated Regime}, \ref{Asymptotics in the merging regime}, and \ref{Asymptotics in the edge} with Theorems 1.8, 1.9, and 1.10 in \cite{CFL2021}, respectively.

\subsection{Asymptotics of Hankel Determinants in the Separated Regime}

We begin with Case (I), in which the weight function has jump discontinuities at $x_1$  and $x_2$, neither close to each other nor to the endpoints $\pm 1$. The asymptotics of the corresponding Hankel determinants follow from \cite[Theorem 1.3]{CG2021}, which applies to a family of Jacobi-type weight functions with finitely many Fisher-Hartwig singularities in $(-1,1)$. The weight function in \cite[Theorem 1.3]{CG2021} is required to be analytic in a fixed neighborhood of $[-1, 1]$ except for these Fisher-Hartwig singularities. We need a slight generalization of their result to allow the weight to possess finitely many additional singularities in the complex plane that may be close to $[-1,1]$. More precisely, we consider weight functions that are analytic and uniformly bounded on the following domain:
\begin{eqnarray}
\mathcal{S}_N &=& \left\{ z\in \mathbb{C}: |\Re z| \leq 1-3\delta_N, |\Im z| < \varepsilon_N / 2\right\} \nonumber\\
& & \cup \left\{ z\in\mathbb{C}: 1-3\delta_N \leq |\Re z| \leq 1+3\delta_N, |\Im z| < 3\delta_N \right\}.
\end{eqnarray}
where $\varepsilon_N = N^{-1+\alpha}$ and $\delta_N = N^{-\alpha / 2}$ with $0<\alpha<2/3$.

\begin{proposition}\label{Asymptotics in the Separated Regime}
Let $\gamma_1, \gamma_2 \in \mathbb{R}$ and  $-1<x_1<x_2<1$. Assume that $w=w_N$ is a sequence of functions which are real analytic on $[-1, 1]$ and uniformly bounded on $\mathcal{S}_N$. Then, we have
\begin{eqnarray} \label{equation: Asymptotics in seperated regime}
& & \log\frac{D_N(x_1, x_2; \gamma_1, \gamma_2; t+w_N)}{D_N(x_1, x_2; \gamma_1, \gamma_2; t)} \nonumber\\
&=& N\int \frac{w_N(x)dx}{\pi \sqrt{1-x^2}} + \frac{\alpha+\beta}{2\pi} \int_{-1}^1 \frac{w_N(x)}{\sqrt{1-x^2}}dx - \frac{\beta}{2}w_N(-1) - \frac{\alpha}{2}w_N(1) \nonumber\\
& & + \frac{1}{2}\sigma(w_N)^2 +\sigma^2(w_N; t) + \sum_{j=1}^2 \frac{\gamma_j}{\sqrt{2}} \sqrt{1-x_j^2} \mathcal{U} w_N(x_j) + o(1)
\end{eqnarray}
as $N\to\infty$. Here $\mathcal{U}w$ is the finite Hilbert transform defined by
\begin{equation} \label{def-Hilbert-transform}
    \mathcal{U}w(x) = \frac{1}{\pi} P. V. \int_{-1}^1 \frac{w(t)}{x-t}\frac{dt}{\sqrt{1-t^2}},
\end{equation}
and
\begin{equation}
\sigma^2(w; t) = \frac{1}{2\pi^2} \int\int_{[-1,1]^2} w'(x)t'(y) \Sigma(x, y)dxdy = -\frac{1}{2\pi} \int_{-1}^1 w'(t) \mathcal{U} t(s) \sqrt{1-s^2}ds, \label{eq:sigma-def}
\end{equation}
with $\sigma^2(w) := \sigma^2(w; w)$ and $\Sigma(x, y)$ given in \eqref{correlation kernel of X}.
\end{proposition}

\begin{proof}
The asymptotics is obtained by setting the parameters $m=2, \alpha_0 = \alpha, \alpha_3 = \beta, \alpha_1 = \alpha_2 = 0; V(x)=1; \psi(x) = 1/ \pi$ and $t_k = x_k, \beta_k = -\frac{i\sqrt{2}\gamma_k}{2}$ with $k=1, 2$ in \cite[Theorem 1.3]{CG2021}.  The result in \eqref{equation: Asymptotics in seperated regime} is derived by comparing the asymptotics for the two weight functions $W(x)=t(x) + w_N(x)$ and $W(x) = t(x)$, where $W(x)$ denotes the weight function appearing in \cite[Theorem 1.3]{CG2021}.

When $w_N$ is analytic in a fixed neighborhood of $[-1, 1]$, the error term in \eqref{equation: Asymptotics in seperated regime} is $O(\frac{\log N}{N})$, as established in \cite[Theorem 1.3]{CG2021}. In our setting, as we only assume that $w_N$ is analytic and uniformly bounded in a shrinking neighborhood $\mathcal{S}_N$ of $[-1, 1]$, the error term deteriorates slightly and becomes $o(1)$ as $N\to\infty$.
\end{proof}



\subsection{Asymptotics of the Hankel Determinants in the Merging Regime}

Next, we move to Case (II), in which the two jumps in the weight function are close to each other but not near the endpoints $\pm 1$. To derive the asymptotics of the corresponding Hankel determinants, we first establish the following approximation for $E_{N, y}(z)$.

\begin{lemma}
Let $E_{N, y}(z)$ be the analytic function in $B(x_1, \delta)$ defined in \eqref{Emerge} and $x_2 = x_1+y$. Then, there exists a constant $0<c< \delta$ such that
\begin{equation} \label{eq:sec4.3-E-approx}
    E_{N, y}^{-1}(x_2) E_{N, y}'(x_2) = O(1), \quad \textrm{as } N \to \infty,
\end{equation}
uniformly for $0<y\leq c$.
\end{lemma}
\begin{proof}
Recalling the definition of $E_{N, y}(z)$ in \eqref{Emerge}, we have 
\begin{eqnarray}
& & \hspace{-0.5cm} E_{N, y}^{-1}(z)E_{N, y}'(z) =  \varphi_{+}(x_1)^{N\sigma_3} \left(\widehat{\Phi}^\infty(\lambda_y(z)) \frac{d}{dz} \left(\widehat{\Phi}^\infty(\lambda_y(z))^{-1}\right) \right. \nonumber\\
& & \  + \widehat{\Phi}^{\infty}(\lambda_y(z)) w_J(z)^{-\sigma_3 / 2} \frac{d}{dz} (w_J(z)^{\sigma_3 / 2}) \widehat{\Phi}^{\infty}(\lambda_y(z))^{-1} \nonumber\\
& & \  + \left. \widehat{\Phi}^\infty(\lambda_y(z)) w_J(z)^{-\sigma_3 / 2} P^\infty(z)^{-1} \frac{d}{dz}P^\infty(z) w_J(z)^{\sigma_3 / 2} \widehat{\Phi}^\infty(\lambda_y(z))^{-1}\right) \varphi_{+}(x_1)^{-N\sigma_3}. \label{eq:lemma4.3-1}
\end{eqnarray}
Since $E_{N, y}(z)$ is analytic in $B(x_1, \delta)$, we will study the limit $z\to x_2$ of the above quantity with $\Im z < 0$. For the first term in the above formula, by \eqref{hatpsiinfty}, we have 
\begin{eqnarray} \label{merging: Hankel: fisrt term}
\widehat{\Phi}^\infty(\lambda_y(z)) \frac{d}{dz} \widehat{\Phi}^\infty(\lambda_y(z))^{-1} = -\frac{\gamma_1}{\sqrt{2}i\lambda_y(z)}\lambda'_y(z) \sigma_3 - \frac{\gamma_2}{\sqrt{2}i(\lambda_y(z) - 1)} \lambda'_y(z)\sigma_3, \quad \Im z < 0.
\end{eqnarray}
For the second term, as $\widehat{\Phi}^{\infty}(z)$ is a diagonal matrix for $\Im z < 0$, we get
\begin{eqnarray}
\widehat{\Phi}^{\infty}(\lambda_y(z)) w_J(z)^{-\sigma_3 / 2} \frac{d}{dz} (w_J(z)^{\sigma_3 / 2}) \widehat{\Phi}^{\infty}(\lambda_y(z))^{-1} 
= w_J(z)^{-\sigma_3 / 2} w_J'(z)^{\sigma_3 / 2}.
\end{eqnarray}
For the third term, from the definition of $P^\infty$ in \eqref{Pinfty}, we have
\begin{multline}
P^\infty(z)^{-1}\frac{d}{dz}P^\infty(z) = \\ -\left[(\log D)'_\gamma(z)) +(\log D_{w})'(z) + (\log D_{t})'(z)\right]\sigma_3 + D(z)^{\sigma_3}Q^{-1}(z)Q'(z)D(z)^{-\sigma_3}.
\end{multline}
It is direct to see that $D(z)^{\sigma_3}Q^{-1}(z)Q'(z)D(z)^{-\sigma_3}$ is bounded near $z = x_2$,  and  $\varphi_+(x_1)^N = O(1)$ as $N\to\infty$. We also note that all other terms in the preceding three formulas are diagonal matrices, which commute with $\varphi_{+}(x_1)^{\pm N \sigma_3}$ in \eqref{eq:lemma4.3-1}. Recalling the definition of $\lambda_y(z)$ in \eqref{lambdayz}, we know $\lambda_y(x_2) = 1$. Therefore, we have
\begin{equation} \label{Eestimate}
E_{N, y}^{-1}(x_2) E'_{N, y}(x_2) = \left[-\frac{\gamma_1}{\sqrt{2}i}\lambda'_y(x_2) + \lim_{z\to x_2}\left( -\frac{\gamma_2}{\sqrt{2}i(\lambda_y(z) - 1)} \lambda'_y(z) - (\log D_\gamma)'(z)\right)\right]\sigma_3 + O(1)
\end{equation}
as $N\to\infty$, uniformly for $0< y \leq c$.

To study $(\log D_\gamma)'(z)$, from its definition in \eqref{Dgamma}, one has
\begin{eqnarray}
D_{\gamma}(z)
&=& \prod_{j=1}^2 \left(\frac{-zx_j+1 + e^{\frac{\pi i}{2}} \sqrt{(z^2-1)(1-x_j^2)}}{z-x_j}\right)^{\frac{\gamma_j}{\sqrt{2}i}}.
\end{eqnarray}
This gives us
\begin{eqnarray}
(\log D_\gamma)'(z) &=& \sum_{j=1}^2 \frac{\gamma_j}{\sqrt{2}i} \frac{d}{dz}\log \left(\frac{-zx_j+1 + e^{\frac{\pi i}{2}} \sqrt{(z^2-1)(1-x_j^2)}}{z-x_j}\right) \nonumber\\
&=& -\frac{1}{\sqrt{2}i} \left(  \frac{\gamma_2}{z-x_2} + \frac{\gamma_1}{x_2-x_1}  \right)+ O(1), \qquad z \to x_2, 
\end{eqnarray}
uniformly for $0< y \leq  c$. Using the definition of $\lambda_y(z)$ in \eqref{lambdayz}, one can see that for $\Im z<0$,
\begin{eqnarray}
\frac{\lambda'_y(z)}{\lambda_y(z)-1} = \frac{(\log\varphi)'(z)}{\log\varphi(z) + \log\varphi_{+}(x_2)}. 
\end{eqnarray}
Hence, we have
\begin{eqnarray}
& & \lim_{z\to x_2, \, \Im z <0}\left( -\frac{\gamma_2}{\sqrt{2}i(\lambda_y(z) - 1)} \lambda'_y(z) - (\log D_\gamma)'(z)\right) \nonumber\\
&=& \lim_{z\to x_2, \, \Im z <0}\left( -\frac{\gamma_2}{\sqrt{2}i} \frac{(\log\varphi)'(z)}{\log\varphi(z) + \log\varphi_{+}(x_2)} -\frac{1}{\sqrt{2}i} \left(  \frac{\gamma_2}{z-x_2} + \frac{\gamma_1}{x_2-x_1}  \right) \right) + O(1) \nonumber\\
&=& \frac{\gamma_1}{\sqrt{2}i}\frac{1}{x_2-x_1} + O(1),
\end{eqnarray}
uniformly for $0<y\leq c$. Using \eqref{Eestimate}, we obtain the desired approximation in \eqref{eq:sec4.3-E-approx} from the formula above.
\end{proof}

Now we have the following asymptotics for the Hankel determinants in the merging regime.
\begin{proposition}\label{Asymptotics in the merging regime}
Let $\gamma_1, \gamma_2\in\mathbb{R}$, $-1<x_1<x_2<1$, $t(x)$ be real analytic on $[-1,1]$, and $D_N(x_1, x_2; \gamma_1, \gamma_2; t)$ be defined in \eqref{eq1}. Then, as $N\to\infty$, we have
\begin{eqnarray}
 \log \frac{D_N(x_1, x_2; \gamma_1, \gamma_2; t)}{D_N(x_1; \gamma_1+\gamma_2; t)} = 
\sqrt{2}N\gamma_2 \int_{x_1}^{x_2} \frac{du}{\sqrt{1-u^2}} - \gamma_1\gamma_2\max\{0, \log |x_1-x_2|N\} + O(1),
\end{eqnarray}
where the error term is uniform for $-1+\delta<x_1<x_2<1-\delta, 0<x_2-x_1<\delta$ for $\delta$ sufficiently small.
\end{proposition}

\begin{proof}
By tracing back the transformations $Y \mapsto T \mapsto S \mapsto R$ in \eqref{transform: Y-to-T}, \eqref{eq:T-S-map} and \eqref{merging-R-def}, we have $Y(z) = 2^{-N\sigma_3}R(z) P_{x_1}(z)\varphi(z)^{N\sigma_3}$  for $z\in B(x_1, \delta)$. Recalling the definition of the local parametrix $P_{x_1}$ in \eqref{Pmerge}, we obtain
\begin{equation}
Y(z) = 2^{-N\sigma_3}R(z)E_{N, y}(z)\widehat{\Phi}_{\text{PV}}(\lambda_y(z); s_{N, y}) w_J(z)^{-\sigma_3 / 2}, \qquad z\in B(x_1, \delta).
\end{equation}
As $x_2 \in B(x_1, \delta) $ and $\lambda_y(x_2) = 1$, we apply the differential identity \eqref{di1} to get
\begin{eqnarray}\label{Asymptotics of DN in the bulk}
& & \frac{d}{dy} \log D_N(x_1, x_2 = x_1+y; \gamma_1, \gamma_2; t) \nonumber\\
&=& -\frac{1-e^{\sqrt{2}\pi\gamma_2}}{2\pi i}
\left[ \lim_{z\to x_2} \left( (\widehat{\Phi}_{\text{PV}}(\lambda_y(z); s_{N, y}))^{-1} E_{N, y}(z)^{-1} R(z)^{-1} R'(z) E_{N, y}(z) \widehat{\Phi}_{\text{PV}}(\lambda_y(z); s_{N, y})\right)_{2, 1} \right. \nonumber\\
& & + \lim_{z\to x_2}\left( (\widehat{\Phi}_{\text{PV}}(\lambda_y(z); s_{N, y}))^{-1} E_{N, y}(z)^{-1} E'_{N, y}(z) \widehat{\Phi}_{\text{PV}}(\lambda_y(z); s_{N, y})\right)_{2, 1} \nonumber\\
& & \left.+ \lambda'(x_2)\lim_{\lambda\to 1}\left( (\widehat{\Phi}_{\text{PV}}(\lambda_y(z); s_{N, y}))^{-1} \widehat{\Phi}_{\text{PV}}'(\lambda_y(z); s_{N, y})\right)_{2, 1}\right].
\end{eqnarray}
For the first term in the bracket, recalling the asymptotics of $R(z)$ in \eqref{Asym for R in merging case}, we have
\begin{equation*}
R(x_2)^{-1}R'(x_2) = O(1), \qquad \text{ as } N\to\infty.
\end{equation*}
Obviously, $E_{N, y}(x_2)^{\pm 1}$ is of order $O(1)$ as $N\to\infty$. Moreover, with the approximation in \eqref{eq:sec4.3-E-approx}, we can rewrite \eqref{Asymptotics of DN in the bulk} as
\begin{eqnarray}
& & \frac{d}{dy} \log D_N(x_1, x_2 = x_1+y; \gamma_1, \gamma_2; t) \nonumber\\
&=& -\frac{1-e^{\sqrt{2}\pi\gamma_2}}{2\pi i}  \left[ \lim_{\lambda\to 1}\ \left( (\widehat{\Phi}_{\text{PV}}(\lambda; s_{N, y}))^{-1} O(1) \widehat{\Phi}_{\text{PV}}(\lambda; s_{N, y})\right)_{2, 1} \right. \nonumber\\
& & \left.+ \lambda'(x_2) \lim_{\lambda\to 1}\left( (\widehat{\Phi}_{\text{PV}}(\lambda; s_{N, y}))^{-1} \widehat{\Phi}_{\text{PV}}'(\lambda; s_{N, y})\right)_{2, 1} \right], \qquad \text{ as } N\to\infty.
\end{eqnarray}
Using the properties of $\widehat{\Phi}_{\text{PV}}$ listed in Lemma 6.2 and Proposition 6.3 in \cite{CFL2021}, the above formula yields
\begin{multline}
 \frac{d}{dy} \log D_N(x_1, x_2 = x_1+y; \gamma_1, \gamma_2; t) \\
= \lambda'_y(x_2)\left(\sigma(s_{N, y}) - \frac{\gamma_1+\gamma_2}{2\sqrt{2}i}s_{N, y} - \frac{(\gamma_1+\gamma_2)^2}{4}\right) + O(1),  \qquad \text{ as } N\to\infty,
\end{multline}
where $\sigma(s)$ is a smooth solution to the $\sigma$-form of Painlev\'e V equation.

The remaining  derivation is very similar to that in \cite[Sec. 7.5]{CFL2021} and we omit the details here. Hence, it concludes the proof of the proposition.
\end{proof}

\subsection{Asymptotics of the Hankel Determinants in the Edge Regime}

Finally, we consider Case (III), where the weight function has a single jump point located near one of the endpoints $\pm 1$. Following the steepest analysis in Sec. \ref{sec:edge-analysis}, 
we have from \eqref{P-1}, \eqref{E-1}, \eqref{Xi}, \eqref{P} and \eqref{transform map: edge S to R} that
\begin{equation} \label{eq:Sec4.3-S-formula}
S(z) = R(z)H(z)\Phi_{\text{HG}}(h_u(\lambda_x(z))) e^{-2(\lambda_x(z) u_{N, x})^{1/2} \sigma_3} e^{-\sqrt{2}\pi\frac{\gamma}{4}\sigma_3} e^{\frac{\alpha\pi i}{2}\sigma_3} W(z)^{-\sigma_3}
\end{equation}
for $z$ outside the lens with $\Im z>0$, where  $H$ is given as
\begin{equation}\label{H}
H(z) = E_1(z)\Xi(\lambda_x(z))M(\lambda_x(z))e^{-\frac{\pi \alpha i}{2}\sigma_3} h_u(\lambda_x(z))^{\frac{\gamma}{\sqrt{2}i}\sigma_3} e^{2iu^{1/2}\sigma_3}
\end{equation}
To derive the asymptotics of the corresponding Hankel determinants, we first establish the following two approximations for  $H^{\pm 1}(x)$ and $H^{-1}(x) H'(x)$. 

\begin{lemma}\label{lemma: estimation of H}
Let $H(z)$ be defined in \eqref{H}. For any $\varepsilon > 0$, we have
\begin{equation}
    H(x)^{\pm 1} = O((1-x)^{-1/4}), \qquad \textrm{as } N\to\infty,
\end{equation}
 uniformly for $1- \varepsilon < x < 1- N^{-2}\log\log N$.
\end{lemma}

\begin{proof}

We first estimate $E_{1}^{\pm 1}(x)$. From its definition in \eqref{E}, we have 
\begin{eqnarray} \label{eq: E1(z)-Lemma4.4}
E_1(z) =  D_{\infty}^{\sigma_3}Q(z)D(z)^{-\sigma_3} W(z)^{\sigma_3} \left(\frac{1+e^{-\pi i/2}\sqrt{\lambda_x(z)}}{\sqrt{\lambda_x(z)+1}}\right)^{i\sqrt{2}\gamma \sigma_3} B^{-1} (\lambda_x(z))^{\frac{\sigma_3}{4}}.
\end{eqnarray}
For simplicity, denote $K(z)$ as
\begin{equation} \label{F}
K(z) = D(z)^{-\sigma_3} W(z)^{\sigma_3} \left(\frac{1+e^{-\pi i/2}\sqrt{\lambda_x(z)}}{\sqrt{\lambda_x(z)+1}}\right)^{i\sqrt{2}\gamma \sigma_3}.
\end{equation}
Hence, $E_1(z)$ can be written as $
E_1(z) = D_\infty^{\sigma_3} Q(z)K(z)B^{-1}(\lambda_x(z))^{\frac{\sigma_3}{4}}.$ Since there is only one jump located at $z=x$, it then follows from \eqref{Dgamma} that
\begin{equation}\label{523}
D_{\gamma}(z) = e^{\frac{\gamma\pi}{\sqrt{2}}} \left( \frac{zx-1 + e^{-\frac{\pi i}{2}} \sqrt{(z^2-1)(1-x^2)}}{z-x}\right)^{\frac{\gamma}{\sqrt{2}i}},
\end{equation}
where the principal branch is take in the above formula. Using \eqref{D} and \eqref{W(z)}, we have
\begin{multline}
K(z) = \varphi(z)^{\frac{\alpha+\beta}{2}\sigma_3} e^{\frac{t(z)}{2}\sigma_3} \exp\left(-\frac{(z^2-1)^{1/2}}{2\pi} \int_{-1}^1 \frac{t(x)}{\sqrt{1-x^2}} \frac{dx}{z-x} \sigma_3\right) \\
  \left(\frac{-zx+1+e^{\frac{\pi i}{2}} \sqrt{(z^2-1)(1-x^2)}}{z-x}\right)^{-\frac{\gamma}{\sqrt{2}i} \sigma_3} \left(\frac{1+e^{-\pi i/2}\sqrt{\lambda_x(z)}}{\sqrt{\lambda_x(z)+1}}\right)^{i\sqrt{2}\gamma \sigma_3}.
\end{multline}
Let us take $z\to x$ in the upper half-plane and outside the lens. Note that the first three terms in the above formula is bounded as $z \to x$. For the remaining two terms, we recall the definition of $\lambda_{x}(z) $ in \eqref{lambdadef} and obtain
\begin{equation}\label{lambdaasymp}
\lambda_x(z) = -1+c_1(z-x)+c_2(z-x)^2+O((z-x)^3), \quad \text{ as }z \to x
\end{equation}
with 
\begin{eqnarray}
c_1 = \lambda'_x(x) = \frac{1}{1-x}\left(1 + O(1-x)\right), \qquad 
c_2= O\left(\frac{1}{1-x}\right).
\end{eqnarray}
Then, we have
\begin{eqnarray} \label{asymptotics that help us to estimate F}
& &\left(\frac{-zx+1+e^{\frac{\pi i}{2}} \sqrt{(z^2-1)(1-x^2)}}{z-x}\right)\left(\frac{1+e^{-\pi i/2}\sqrt{\lambda}}{\sqrt{\lambda+1}}\right)^2 \nonumber \\
&=&\frac{1}{c_1(1-x^2)} \left(2+(z-x)\left[\frac{2x}{1-x^2} -c_1 - 2\frac{c_2}{c_1}\right]\right) + O((z-x)^2),
\end{eqnarray}
which implies $K_+(x)$ is of order $O(1)$ as $x\to 1$. Similarly, a straightforward computation gives us $K_+'(x) = O(1)$ as $x\to 1$. Using the definition of $Q$ in \eqref{Q}, we know $Q_+^{\pm 1} (x)= O((1-x)^{-1/4})$ as $x\to 1$. Therefore, we have from \eqref{eq: E1(z)-Lemma4.4} that
\begin{equation}\label{estimation of E1}
E_1^{\pm 1}(x) = O((1-x)^{-1/4}), \qquad \textrm{as } x\to 1.
\end{equation}

We now claim that
\begin{equation}\label{Mehu}
\lim_{z\to x} M(\lambda_x(z)) e^{-\frac{\pi\alpha i}{2} \sigma_3} h_u(\lambda_x(z))^{\frac{\gamma}{\sqrt{2}i}\sigma_3} = O(1), \qquad \textrm{as } x\to 1,
\end{equation}
where the limit $z\to x$ is taken in the upper half-plane and outside the lens. Using the definition of $\lambda_x(z)$ in \eqref{lambdadef}, one can see  $\lambda_x(z)\to -1$ as $z\to x$. In addition, as $\lambda \to -1$, we have 
\begin{equation}
\left(\frac{1+e^{-\pi i/2}\sqrt{\lambda}}{\sqrt{\lambda+1}}\right)^{-i\sqrt{2}\gamma} = 2^{-i\sqrt{2}\gamma}(\lambda+1)^{\frac{i\sqrt{2}\gamma}{2}} \left(1 + \frac{i\sqrt{2}\gamma}{4}(\lambda+1)+O((\lambda+1)^2)\right).
\end{equation}
Then, from the definition of $M$ in \eqref{M} and the asymptotics of $h_u$ in \eqref{hasymp}, we obtain
\begin{eqnarray}
&  M(\lambda) e^{-\frac{\pi\alpha i}{2} \sigma_3} h_u(\lambda)^{\frac{\gamma}{\sqrt{2}i}\sigma_3} = \lambda^{-\frac{1}{4}\sigma_3}B 2^{-i\sqrt{2}\gamma \sigma_3}(\lambda+1)^{\frac{i\sqrt{2}\gamma}{2}\sigma_3} \left(I + \frac{i\sqrt{2}\gamma}{4}(\lambda+1)\sigma_3+O((\lambda+1)^2)\right) \nonumber\\
&  \qquad \qquad \times e^{-\frac{\pi\alpha i}{2} \sigma_3} 2^{\frac{\gamma}{\sqrt{2}i}\sigma_3} u^{\frac{\gamma}{2\sqrt{2}i}\sigma_3}(\lambda+1)^{\frac{\gamma}{\sqrt{2}i}\sigma_3} \left(1+\frac{1}{4}(\lambda+1) + O((\lambda+1)^2)\right)^{\frac{\gamma}{\sqrt{2}i}\sigma_3} \nonumber\\
&\qquad \qquad = \lambda^{-\frac{1}{4}\sigma_3}B 2^{-\frac{3\sqrt{2}\gamma i}{2}\sigma_3} e^{-\frac{\pi\alpha i}{2} \sigma_3} u^{-\frac{\sqrt{2}}{4}i\gamma\sigma_3} \left(I+\frac{i\sqrt{2}\gamma}{8}\sigma_3(\lambda+1) + O((\lambda+1)^2)\right) \label{estimation regarding M and h}
\end{eqnarray}
as $\lambda\to -1$. Since $\gamma$ is real and $u^{-\frac{\sqrt{2}}{4}i\gamma\sigma_3} $ is bounded, the above formula implies \eqref{Mehu}.

Finally, recalling the definition of $H$ in \eqref{H}, estimation of $E_1$ in \eqref{estimation of E1} and the asymptotics for $\Xi$ in \eqref{Xiasymp1}, we conclude that $H(x) = O((1-x)^{-1/4})$. The same argument also applies to $H^{-1}(x)$. This concludes the proof of the lemma.
\end{proof}

\begin{lemma}\label{lemma: estimation of H-1H'}
Let $H(z)$ be defined in \eqref{H}. For any small enough $\varepsilon>0$, we have
\begin{equation} \label{eq:lemma-estimation of H-1H'}
H^{-1}(x)H'(x) = \frac{i\sqrt{2}\gamma}{8}\frac{1}{1-x}\sigma_3 + A_{N, x} + O(1),  \qquad \textrm{as } N\to\infty,
\end{equation}
where $A_{N, x}$ is a matrix such that for sufficiently small $\varepsilon>0$,
\begin{equation}
\int_{1-\varepsilon}^x A_{N, \eta}d\eta = O(1), \qquad \textrm{as } N\to\infty.
\end{equation}
uniformly for $1-\varepsilon < x < 1- N^{-2}\log\log N$.
\end{lemma}

\begin{proof}
Using the definitions of $E_1(z)$ in \eqref{E} and $F(z)$ in \eqref{F}, we introduce
\begin{equation}
G(z) = E_1(z)\Xi(\lambda_x(z)) \lambda_x(z)^{-\sigma_3/4} B = D_\infty^{\sigma_3}Q(z)K(z)\hat{\Xi}(z)
\end{equation}
with
\begin{equation}
\hat{\Xi}(z) = B^{-1}\lambda_x(z)^{\frac{1}{4}\sigma_3}\Xi(\lambda_x(z)) \lambda_x(z)^{-\frac{1}{4}\sigma_3}B.
\end{equation}
From \eqref{H} and \eqref{estimation regarding M and h}, we have
\begin{equation}
H(z) = G(z)\cdot 2^{-\frac{3\sqrt{2}}{2}i\gamma\sigma_3} e^{-\frac{\pi\alpha i}{2} \sigma_3} u^{-\frac{\sqrt{2}}{4}i\gamma\sigma_3} \left(I+\frac{i\sqrt{2}\gamma}{8}\sigma_3(\lambda+1) + O((\lambda+1)^2)\right)e^{2iu^{1/2}\sigma_3},
\end{equation}
with $\lambda = \lambda_x(z)$ defined in \eqref{lambdadef}. Taking $z\to x$ as described above, i.e., in the upper half-plane and outside the lens, then for some fixed small $\varepsilon > 0$, we have
\begin{eqnarray}
& & H^{-1}(x)H'(x) \nonumber\\
&=& e^{-2iu^{1/2}\sigma_3} u^{\frac{\sqrt{2}}{4}i\gamma\sigma_3 } e^{\frac{\pi\alpha i}{2} \sigma_3} 2^{\frac{3\sqrt{2}}{2}i\gamma\sigma_3} G^{-1}(x)G'(x) 2^{-\frac{3\sqrt{2}}{2}i\gamma\sigma_3} e^{-\frac{\pi\alpha i}{2} \sigma_3} u^{-\frac{\sqrt{2}}{4}i\gamma\sigma_3} e^{2iu^{1/2}\sigma_3} \nonumber\\
& & +\frac{i\sqrt{2}\gamma}{8}\cdot \frac{1}{1-x}\sigma_3 + O(1), \qquad \textrm{as } N\to\infty, \label{HH}
\end{eqnarray}
uniformly for $x\in (1-\varepsilon, 1-N^{-2}\log \log N)$. In the first term on the right-hand side, all factors except $G^{-1}(x)G'(x)$ are bounded as $N \to \infty$.  From the asymptotics of $\Xi$ in \eqref{Xiasymp1} and $\lambda$ in \eqref{lambdaasymp}, we have
\begin{eqnarray}
\hat{\Xi}(x) = I + O(u_{N, x}^{-1/2}).
\end{eqnarray}
To estimate $G^{-1}(x)G'(x)$, we first consider $\hat{\Xi}^{-1}(x) \hat{\Xi}'(x)$. It is direct to see that $\frac{d}{d\lambda} \left(\lambda^{\frac{1}{4}\sigma_3} \Xi(\lambda) \lambda^{-\frac{1}{4}\sigma_3}\right)$ is of order $O(u_{N, x}^{-1/2})$
as $N\to \infty$, where $u_{N, x} = O(N^2 (1-x))$. Then, we have
\begin{eqnarray}
\hat{\Xi}'(x) = B^{-1} \frac{d}{d\lambda} \left(\lambda^{\frac{1}{4}\sigma_3} \Xi(\lambda) \lambda^{-\frac{1}{4}\sigma_3}\right) \lambda'_x(x) B
= O(u^{-1/2} (1-x)^{-1}),
\end{eqnarray}
and hence
\begin{equation}
\hat{\Xi}(x)^{-1} \hat{\Xi}'(x) = O(u^{-1/2} (1-x)^{-1}) = O(N^{-1} (1-x)^{-3/2})
\end{equation}
as $N\to\infty$. Since both $K(x)$ and $K'(x)$ are bounded as $x\to 1$, we obtain, for any $\varepsilon > 0$
\begin{eqnarray}
& & (G_+^{-1}G'_+)(x) \nonumber\\
&=& \hat{\Xi}_+^{-1}(x)K_+^{-1}(x)Q_+^{-1}(x) \left(Q'_+(x)K_+(x)\hat{\Xi}_+(x) + Q_+(x)K'_+(x)\hat{\Xi}_+(x) + Q_+(x)K_+(x)\hat{\Xi}'_+(x)\right)\nonumber\\
&=& \hat{\Xi}_+^{-1}(x)K_+^{-1}(x)Q_+^{-1}(x)Q'_+(x)K_+(x)\hat{\Xi}_+(x) + \hat{\Xi}_+^{-1}(x)K_+^{-1}(x)K'_+(x)\hat{\Xi}_+(x) + \hat{\Xi}_+^{-1}(x)\hat{\Xi}'_+(x)\nonumber \\
&=& Q_+^{-1}(x)Q'_+(x)\left(I+ O(u_{N, x}^{-1/2}) + O((1-x)^{1/2})\right) + O(1) + O(N^{-1} (1-x)^{-3/2}), \label{eq:G-estimate}
\end{eqnarray}
as $N\to\infty$, uniformly for $x\in (1-\varepsilon, 1- N^{-2}\log\log N)$. With the explicit expression of $Q$ in \eqref{Q}, we have 
\begin{eqnarray}
Q_+^{-1}(x)Q'_+(x)= \frac{i}{2(1-x^2)}\left(
                                    \begin{array}{cc}
                                      0 & 1 \\
                                      -1 & 0 \\
                                    \end{array}
                                  \right).
\end{eqnarray}
Combining this with \eqref{HH} and \eqref{eq:G-estimate}, we conclude that, as $N\to\infty$, 
\begin{eqnarray}
& & H^{-1}(x)H'(x) \nonumber\\
&=& {\small \frac{i\sqrt{2}\gamma}{8}\frac{1}{1-x}\sigma_3 + \frac{i}{2(1-x^2)}
\left(
\begin{array}{cc}
0 & e^{-4iu_{N, x}^{1/2}}u_{N, x}^{\frac{\sqrt{2}}{2}i\gamma} e^{\pi \alpha i} 2^{3\sqrt{2}i\gamma} \\
-e^{4i u_{N, x}^{1/2}}u_{N, x}^{-\frac{\sqrt{2}}{2}i\gamma} e^{-\pi\alpha i} 2^{-3\sqrt{2}i\gamma} & 0
\end{array}
\right) } \nonumber\\
& & + O(N^{-1}(1-x)^{-3/2}) + O(1),
\end{eqnarray}
uniformly for $x\in (1-\varepsilon, 1- N^{-2}\log\log N)$.

To derive the asymptotics of the Hankel determinants, we need an approximation for the integration of the above expression with respect to  $x$. Regarding the term $O(N^{-1}(1-x)^{-3/2}) $, integrating from $1-\varepsilon$ to $x$ yields a term of order $O(N^{-1} (1-x)^{-1/2})$, which is bounded for $x\in (1-\varepsilon, 1- N^{-2}\log\log N)$. Our final task is to show that the integration of the second term in the above approximation is bounded. More precisely, we need to prove that
\begin{equation} \label{eq:prop-4.5-estimate}
\int_{1-\varepsilon}^{x} e^{2\pi i r_N(\eta)}\frac{d\eta}{1-\eta^2} = O(1), \qquad \textrm{as } N \to \infty,
\end{equation}
where 
\begin{equation}
r_N(\eta) = -\frac{2}{\pi}u_{N,\eta}^{1/2} - \frac{\sqrt{2}\gamma}{4\pi}\log u_{N,\eta}.
\end{equation}
Recalling \eqref{eq:f-expan} and \eqref{def:un-x},  we have
\begin{equation*}
u_{N,\eta} = -N^2 f(\eta) = -N^2\left(\frac{1}{2}(\eta - 1) - \frac{1}{12}(\eta - 1)^2 + O((\eta-1)^3)\right).
\end{equation*}
Thus, we find
\begin{equation}
r_N'(\eta) = - \left( \frac{1}{\pi}u_{N,\eta}^{-1/2} - \frac{\sqrt{2}\gamma}{4\pi u_{N, \eta}} \right) \frac{d}{d\eta} u_{N,\eta} = \frac{\sqrt{2}N}{2\pi (1-\eta)^{1/2}}\left( 1+O\left(\frac{1}{N(1-\eta)^{1/2}}\right)\right),
\end{equation}
as $N\to\infty$, uniformly for $\eta\in (1-\varepsilon, 1-N^{-2}\log \log N)$. Based on this approximation, we have
\begin{equation}
\int_{1-\varepsilon}^x e^{2\pi i r_N(\eta)}\frac{d\eta}{1-\eta^2} = \int_{1-\varepsilon}^x e^{2\pi i r_N(\eta)} \frac{r'_N(\eta) \cdot 2\pi(1-\eta)^{1/2}}{\sqrt{2}N(1-\eta^2)} d\eta + O\left(\frac{1}{N}\int_{1-\varepsilon}^x \frac{1}{(1-\eta)^{3/2}}d\eta\right).
\end{equation}
The last term is $O(N^{-1}(1-x)^{-1/2})$, which is of order $O(1)$ (or more precisely, of order $O((\log\log N)^{-1/2})$.  For the first term, an integration by parts gives us
\begin{eqnarray}
& &\int_{1-\varepsilon}^x e^{2\pi i r_N(\eta)}\frac{\sqrt{2}\pi r'_N(\eta) \cdot (1-\eta)^{1/2}}{N(1-\eta^2)} d\eta\nonumber \\
&=& \left.\sqrt{2}ie^{2\pi i r_N(\eta) }\frac{\sqrt{2}\pi  \cdot (1-\eta)^{1/2}}{N(1-\eta^2)}\right|_{1-\varepsilon}^x - \frac{\sqrt{2}\pi}{N} \int_{1-\varepsilon}^x e^{2\pi i r_N(\eta)} \frac{d}{d\eta} \left(\frac{(1-\eta)^{1/2}}{1-\eta^2}\right)d\eta.
\end{eqnarray}
Both terms in the above expression are of order $O(N^{-1}(1-x)^{-1/2})$, which is bounded as $N \to \infty $ uniformly for $x\in (1-\varepsilon, 1-N^{-2}\log\log N)$. Thus, we have established the approximation \eqref{eq:prop-4.5-estimate} and completed the proof of the lemma.
\end{proof}

Now we have the following asymptotics for the Hankel determinants in the edge regime.

\begin{proposition}\label{Asymptotics in the edge}
Let $\gamma\in [-M, M]$ with $M > 0$ be a positive constant, $t(x)$ be real analytic on $[-1,1]$, and $D_N(x; \gamma; t)$ be defined in \eqref{eq2}. Then, as $N\to\infty$, we have
\begin{equation}
\log \frac{D_N(x; \gamma; t)}{D_N(x; 0; t)} = \sqrt{2}\pi\gamma N\int_{-1}^x \pi^{-1} (1-s^2)^{-\frac{1}{2}}ds+ \frac{\gamma^2}{2}\log N+\frac{\gamma^2}{4}\log (1-x^2) + O(1), \label{eq3}
\end{equation}
where the error term is uniform for $|x| \leq 1- N^{-2}\log\log N$.
\end{proposition}

\begin{proof}
We first consider $ x \in [1-\varepsilon , 1- N^{-2} \log\log N]$, , which lies in the edge regime analyzed in Sec. \ref{sec:edge-analysis}. Tracing back the transformations $Y\mapsto T \mapsto S$ in \eqref{transform: Y-to-T} and \eqref{eq:T-S-map}, we have $Y(z) = 2^{-N\sigma_3} S(z) \varphi(z)^{N\sigma_3}$  for $z\in B(1, \delta)$, so that
\begin{equation*}
\left(Y^{-1}(z) Y'(z)\right)_{2, 1} = \Big( \varphi(z)^{-N \sigma_3} S^{-1}(z) S'(z) \varphi(z)^{N \sigma_3} \Big)_{2, 1}.
\end{equation*}
Using the differential identity \eqref{di2} and the above formula, we obtain
\begin{eqnarray}\label{edge: derivative of log DN}
\frac{d}{dx}\log D_N(x; \gamma; t) = \frac{i}{2\pi}(1-e^{\sqrt{2}\pi\gamma})\varphi_+(x)^{2N}(1-x)^{\alpha}(1+x)^{\beta} e^{t(x)} \left(S_+(x)^{-1}S'_+(x)\right)_{2, 1},
\end{eqnarray}
where $S_+(x)$ denotes the limit from outside the lens in the upper half-plane. We recall the expression of $S(z)$ in \eqref{eq:Sec4.3-S-formula}. From \eqref{eq:W-wj-relation}, for any matrix, we know
\begin{equation*}
\Big(W_{+}(x)^{\sigma_3} e^{-\frac{\alpha}{2}\pi i \sigma_3} M W_{+}(x)^{-\sigma_3} e^{\frac{\alpha}{2}\pi i \sigma_3} \Big)_{2,1} = M_{2,1}(1-x)^{-\alpha}(1+x)^{-\beta} e^{-t(x)}.
\end{equation*}
Similarly, using \eqref{def:un-x}, we have 
\begin{equation*}
\Big(e^{\sqrt{2}\pi\frac{\gamma}{4}\sigma_3} e^{2(\lambda_x(z) u_{N, x})^{1/2} \sigma_3} M  e^{-\sqrt{2}\pi\frac{\gamma}{4}\sigma_3} e^{-2(\lambda_x(z) u_{N, x})^{1/2} \sigma_3}  \Big)_{2,1} = M_{2,1} e^{-\frac{\pi\gamma}{\sqrt{2}}} \varphi_+(x)^{-2N}.
\end{equation*}
Since $\lambda_x(x) = -1$ and $h_u(-1) = 0$, then we have from the above three formulas that
\begin{eqnarray}
& &\frac{d}{dx}\log D_N(x; \gamma; t) = \frac{i}{2\pi}\left(e^{-\pi\gamma/\sqrt{2}}-e^{\pi\gamma/\sqrt{2}}\right)  \left[\left.\frac{d}{dz}h_{u_N,x}(\lambda_x(z))\right|_{z=x} \left(\Phi_{HG,+}^{-1}(0)\Phi'_{HG,+}(0)\right)_{2, 1}\right. \nonumber\\
& &\qquad \qquad \qquad +\left(\Phi_{HG,+}^{-1}(0) H^{-1}(x) H'(x) \Phi_{HG,+}(0)\right)_{2, 1} \nonumber\\
& &\qquad \qquad \qquad  +\left.\left(\Phi_{HG,+}^{-1}(0) H^{-1}(x) R_+^{-1}(x)R'_+(x) H(x) \Phi_{HG,+}(0)\right)_{2, 1}\right].\label{logDN1}
\end{eqnarray}

To compute the first term in the bracket in \eqref{logDN1}, we apply \eqref{Appendix: HG Lemma} with $\tilde{\beta}=\frac{\gamma}{\sqrt{2}i}$ to get 
\begin{equation}
\lim_{z\to 0}\left(\Phi_{HG,+}^{-1}(z)\frac{d}{dz}\Phi_{HG,+}(z)\right)_{2, 1} = \frac{\sqrt{2}\pi \gamma i}{e^{\pi \gamma/ \sqrt{2}}-e^{-\pi \gamma/ \sqrt{2}}}.
\end{equation}
Using the asymptotics of $h_u(x)$ in \eqref{hasymp} and the definitions of $u_{N, x}$ and $\lambda_{x}(z)$ from \eqref{def:un-x} and \eqref{lambdadef}, we obtain
\begin{equation}
\left.\frac{d}{dz}h_{u_{N, x}} \left(\lambda_x(z)\right)\right|_{z=x} = 2u_{N, x}^{1/2}\left.\frac{d}{dz} \left(\lambda_x(z)\right)\right|_{z=x} = 2N(1-x^2)^{-1/2}. 
\end{equation}
Combining these two formula, we have
\begin{equation} \label{logDN1-appro1}
\left. \frac{d}{dz}h_{u_N,x}(\lambda_x(z))\right|_{z=x} \left( \Phi_{HG,+}^{-1}(0)\Phi'_{HG,+}(0) \right)_{2, 1} = \sqrt{2}\gamma N(1-x)^{-1/2}.
\end{equation}
For the second term in the bracket in \eqref{logDN1}, we recall the estimate of $H^{-1}(x)H'(x) $ in Lemma \ref{lemma: estimation of H-1H'}. From \eqref{Appendix: HG specific value}, we have
\begin{eqnarray}
\left(\Phi_{HG, +}^{-1}(0) \sigma_3 \Phi_{HG, +}(0)\right)_{2, 1} & =&  
-2 \lim_{z\to 0, \Im z>0} \Phi_{HG, 11}(z)\Phi_{HG, 21}(z) \nonumber\\
&=& -2 \Gamma(1-\beta) \Gamma(1+\beta) =  \frac{-2\sqrt{2}\pi \gamma}{e^{\pi \gamma / \sqrt{2}} - e^{-\pi \gamma / \sqrt{2}}}.
\end{eqnarray}
It then follows from \eqref{eq:lemma-estimation of H-1H'} and the above formula that
\begin{equation} \label{logDN1-appro2}
\left(\Phi_{HG,+}^{-1}(0) H^{-1}(x) H'(x) \Phi_{HG,+}(0)\right)_{2, 1} = \frac{\gamma^2}{4}\frac{1}{1-x} + B_{N, x} + O(1), \quad \textrm{as } N\to\infty, 
\end{equation}
uniformly for $1-\varepsilon < x < 1- N^{-2} \log\log N$, where $B_{N, x}$ satisfies the approximation 
\begin{equation} \label{eq:BNx-integral}
\int_{1-\varepsilon}^x B_{N, \eta} d\eta = O(1).
\end{equation}
Regarding the last term in the bracket in \eqref{logDN1}, the asymptotics of $R(z)$ in \eqref{Rasymp} and $H(x)$ in Lemma~\ref{lemma: estimation of H} imply
\begin{equation}\label{the third term to evaluate in the edge regime}
\left(\Phi_{HG,+}^{-1}(0) H^{-1}(x) R_+^{-1}(x)R'_+(x) H(x) \Phi'_{HG,+}(0)\right)_{2, 1} = O(N^{-1} (1-x)^{-1/2}),
\end{equation}
which is of order $O(1)$ for $ x \in [1-\varepsilon,  1- N^{-2} \log\log N]$.

Thus, combining \eqref{logDN1-appro1}, \eqref{logDN1-appro2} and \eqref{the third term to evaluate in the edge regime}, we obtain
\begin{eqnarray}
\frac{d}{dx}\log D_N(x;\gamma;0) =\sqrt{2}\gamma N(1-x^2)^{-1/2} - \frac{\gamma^2}{4(1-x)} + B_{N, x} + O(1), \quad \textrm{as } N\to\infty, 
\end{eqnarray}
uniformly for $ x \in [1-\varepsilon,  1- N^{-2} \log\log N]$. Integrating the above formula from $1-\varepsilon$ to $x$ and using \eqref{eq:BNx-integral}, we have
\begin{equation}
\log \frac{D_N(x; \gamma; t) }{D_N(1-\varepsilon; \gamma; t)} =  \sqrt{2}\pi \gamma N \int_{1-\varepsilon}^x \pi^{-1} (1-\eta^2)^{-1/2}d\eta + \frac{\gamma^2}{4}\log (1-x) + O(1)
\end{equation}
as $N\to\infty$, uniformly for $ x \in [1-\varepsilon,  1- N^{-2} \log\log N]$.

Finally, to obtain the result in \eqref{eq3}, we apply \cite[Theorem 1.3]{CG2021} to derive the asymptotics of $D_N(1-\varepsilon; \gamma; t)$ and $D_N(x; 0; t)$. For $D_N(1-\varepsilon; \gamma; t)$, which corresponds to a weight function with a jump at $1-\varepsilon$, we set $m=1, \alpha_0 = \alpha,\alpha_1 = 0, \alpha_2 = \beta, V(x)=1, \psi(x) = 1/ \pi$ and $t_1 = 1-\varepsilon, \beta_1 = -\frac{i\sqrt{2}\gamma}{2}$ in the theorem. While for $D_N(x; 0; t)$, where the weight function is continuous on $[-1,1]$, we simply set $m=0$ and keep the remaining parameters. Then, we have
\begin{eqnarray}
\log \frac{D_N (1-\varepsilon; \gamma; t)}{D_N(x; 0; t)} = \sqrt{2}\pi \gamma N \int_{-1}^{1-\varepsilon} \pi^{-1} (1-\eta^2)^{-1/2}d\eta + \frac{\gamma^2}{2}\log N + O(1)
\end{eqnarray}
as $N\to\infty$. Combining the above two equations, we get
\begin{equation}
\log \frac{D_N(x; \gamma; t)}{D_N(x; 0; t)} = \sqrt{2}\pi\gamma N\int_{-1}^x \pi^{-1} (1-s^2)^{-\frac{1}{2}}ds+ \frac{\gamma^2}{2}\log N+\frac{\gamma^2}{4}\log (1-x) + O(1)
\end{equation}
as $N\to\infty$, uniformly for $ x \in [1-\varepsilon,  1- N^{-2} \log\log N]$. A similar result holds near the left edge $x=-1$; by replacing $\log (1-x)$ with $\log (1-x^2)$, we unify both results to yield \eqref{eq3} for $ x \in [1-\varepsilon,  1- N^{-2} \log\log N]$. Finally, it is easy to see from \cite[Theorem 1.3]{CG2021} that \eqref{eq3} also holds for $x$ in compact subsets of $[-1,1]$. This completes the proof of the proposition.
\end{proof}

\section{Eigenvalue Rigidity} \label{Eigenvalue Rigidity}

Recall that the exponential moments are related to Hankel determinants through \eqref{Heine Identity}. Using the asymptotics of Hankel determinants established in the previous section, we will demonstrate that the random measure $d\mu_N^\gamma(x) $ in \eqref{eq:dmu-N} converges to a GMC measure as $N \to \infty$.

\subsection{Estimation on the Exponential Moments of $h_N(x)$}

We begin by verifying that $h_N$, defined in \eqref{hN}, satisfies \cite[Assumption 3.1]{CFL2021}. Specifically, we prove the following proposition.

\begin{proposition}\label{Assumption 3.1}
Let $A$ be any compact set in  $(-1, 1)$ with positive Lebesgue measure. For any $\gamma>0$ and sufficiently large $N$, there exists a constant $C_{\gamma, A} > 0$ such that
\begin{equation} \label{prop5-1-eq1}
\mathbb{E} [ e^{\gamma h_N(x)}] \geq C_{\gamma, A} N^{\gamma^2 / 2} \qquad \textrm{ for all } x\in A.
\end{equation}
Moreover, there exists a constant $R_{\gamma} > 0$ such that 
\begin{equation} \label{prop5-1-eq2}
\mathbb{E} [ e^{\gamma h_N(x)}] \leq R_{\gamma} N^{\gamma^2 / 2} \qquad \textrm{for } x \in (-1, 1).
\end{equation}
The same bounds hold for $-h_N(x)$.
\end{proposition}

\begin{proof}
It suffices to prove the estimates for $h_N(x)$, as the case of $-h_N(x)$ follows similarly

We begin by computing the exponential moment of $h_N(x)$. From its definition in \eqref{hN}, we have
\begin{eqnarray}
\mathbb{E} [e^{\gamma h_N(x)}] = \mathbb{E} \big[e^{\sqrt{2}\pi\gamma \sum_{j=1}^N 1_{\lambda_j\leq x} - \sqrt{2}\pi\gamma N F(x)}\big] 
= e^{-\sqrt{2}\pi \gamma N F(x)} \mathbb{E} \big[e^{\sqrt{2}\pi\gamma \sum_{j=1}^N 1_{\lambda_j\leq x}}\big],
\end{eqnarray}
where $F(x)$ is the distribution function defined in \eqref{distribution of eq measure}. By \eqref{Heine Identity}, we get
\begin{eqnarray}
\mathbb{E} [e^{\gamma h_N(x)}] = e^{-\sqrt{2}\pi \gamma N F(x)} \frac{1}{Z_N} D_N (x; \gamma; t) = e^{-\sqrt{2}\pi \gamma N F(x)} \frac{D_N (x; \gamma; t)}{D_N (0; 0; t)}.  \label{EV-Heine Identity1}
\end{eqnarray}
Then, from Proposition \ref{Asymptotics in the edge}, we can see that
\begin{equation}\label{EehN}
\log \mathbb{E} [e^{\gamma h_N(x)}] = \frac{\gamma^2}{2}\log N+\frac{\gamma^2}{4}\log (1-x^2) + O(1),
\end{equation}
uniformly for $|x|\leq 1- N^{-2}\log\log N$. This implies that, for any compact set $A\subset (-1, 1)$,  the lower bound \eqref{prop5-1-eq1} holds.

We now turn to the upper bound \eqref{prop5-1-eq2}. From the definition of $F(x)$ in \eqref{distribution of eq measure},  there exists a constant $C>0$ such that $(1-x^2)^{1/2} \leq CF(x)$ for all $|x|\leq 1$. Then, by \eqref{EehN}, there exists $R_{\gamma}>0$ such that for all $|x|\leq 1- N^{-2}\log\log N$, 
\begin{equation} \label{hN moment bulk}
\mathbb{E} [ e^{\gamma h_N(x)}] \leq R_{\gamma} (NF(x))^{\gamma^2 / 2}.
\end{equation}
Since $F(x) \in [0,1]$, the right-hand side is bounded by $R_{\gamma} N^{\gamma^2 / 2}$. Next, we extend this estimate to $-1 \leq x\leq -1 + N^{-2}\log\log N$. From the definition of $h_N(x)$ in \eqref{hN}, we  get
\begin{multline}
 h_N(-1+N^{-2}\log\log N) + \sqrt{2}\pi N F(-1+N^{-2}\log \log N) \nonumber\\
= \sqrt{2}\pi \sum_{1\leq j \leq N} 1_{\lambda_j \leq -1 + N^{-2}\log \log N}  \geq \sqrt{2}\pi \sum_{1\leq j \leq N} 1_{\lambda_j \leq x} 
\geq  h_N(x)
\end{multline}
for all $ x\in [-1 , -1 + N^{-2}\log\log N]$. This gives us 
\begin{equation}
\mathbb{E} [ e^{\gamma h_N(x)}] \leq e^{\sqrt{2}\pi \gamma N  F(-1+N^{-2}\log \log N) } \, \mathbb{E} [e^{\gamma h_N(-1+N^{-2}\log \log N)}].
\end{equation}
Note that $F(-1+N^{-2} \log \log N) \leq C N^{-1}(\log\log N)^{1/2}$ for some $C>0$. Applying \eqref{hN moment bulk} at $x = -1+N^{-2} \log \log N$, we obtain
\begin{eqnarray} \label{hN moment edge}
\mathbb{E} [e^{\gamma h_N(x)}] &\leq& e^{\sqrt{2}\pi\gamma C (\log\log N)^{1/2}} R_{\gamma} \Big(N F(-1+N^{-2} \log \log N) \Big)^{\gamma^2 / 2} \nonumber\\
&\leq& e^{\sqrt{2}\pi\gamma C (\log\log N)^{1/2}} R_{\gamma} \Big( C(\log\log N)^{1/2} \Big)^{\gamma^2 / 2},
\end{eqnarray}
which is of order $o(N^{\gamma^2/2})$.  The case for  $1- N^{-2}\log\log N\leq x\leq 1 $ is analogous. This completes the proof of \eqref{prop5-1-eq2} and the proposition.
\end{proof}


Now, we verify the main exponential moment estimates required by \cite[Assumption 2.5]{CFL2021}. Analogous to \eqref{EV-Heine Identity1}, we have the Hankel determinant representation
\begin{eqnarray} \label{EV-Heine Identity2}
\mathbb{E} [e^{\gamma_1 h_N(x) + \gamma_2 h_N(y)}] = e^{-\sqrt{2}\pi (\gamma_1+\gamma_2) N F(x)} \frac{D_N (x, y; \gamma_1, \gamma_2; t)}{D_N (0; 0; t)}.
\end{eqnarray}
Since $h_N(x)$ defined in \eqref{hN} is not centered, we introduce the following modified eigenvalue counting function
\begin{equation} \label{tilde hN}
\tilde{h}_N(x) = h_N(x) - \sqrt{2}\pi \left(\frac{\alpha+\beta}{2\pi} \int_{-1}^x \frac{ds}{\sqrt{1-s^2}} - \frac{\alpha}{2}1_{[1, \infty)}(x) - \frac{\beta}{2}1_{[-1, \infty)}(x) - \frac{1}{2\pi}(\mathcal{U}t(x))\sqrt{1-x^2}\right),
\end{equation}
where $\mathcal{U}t$ is the Hilbert transform defined in \eqref{def-Hilbert-transform}. Note that $\tilde{h}_N(x)$ differs from $h_N(x)$ by a deterministic quantity. 

We also need the following regularized function
\begin{eqnarray} \label{tilde hN-epsilon}
\tilde{h}_{N, \varepsilon}(x) = (\tilde{h}_N \star \varphi_\varepsilon)(x),
\end{eqnarray}
where $\star$ is the convolution operator defined by
\begin{equation}
(f\star g)(x) = \int_{-\infty}^\infty f(x')g(x-x')dx',
\end{equation}
and the mollifier is given by
\begin{equation}\label{mollifier}
\varphi_\epsilon (x) = \varepsilon^{-1} \varphi(x / \varepsilon) \quad \textrm{with} \quad  \varphi(x) = \frac{1}{\pi}\frac{1}{1+x^2}.
\end{equation}
With the above definition, the exponential moment of $\sum_{j=1}^n r_j \tilde{h}_{N, \varepsilon_j}(x)$ becomes
\begin{eqnarray}
& & \mathbb{E} \big[e^{\sum_{j=1}^n r_j \tilde{h}_{N, \varepsilon_j}(x)} \big] =  \mathbb{E} \exp\left( \sum_{j=1}^n r_j \Big(\sum_{k=1}^N w_{\varepsilon_j, x}(\lambda_k) - N \int_{-1}^1 w_{\varepsilon_j, x} d\mu_V \right. \label{eq: exp-moment-1}  \\
& & \hspace{4cm} \left. - \frac{\alpha+\beta}{2\pi} \int_{-1}^1 \frac{w_{\varepsilon_j, x}(s)ds}{\sqrt{1-s^2}} + \frac{\alpha}{2}w_{\varepsilon_j, x}(1) + \frac{\beta}{2} w_{\varepsilon_j, x}(-1) + \sigma^2(w_{\varepsilon_j,x}; t) \Big) \right) \nonumber\\
&& = \mathbb{E} \big[e^{\sum_{k=1}^N w(\lambda_k)}\big] \exp\left[-N \int_{-1}^1 w d\mu_V- \frac{\alpha+\beta}{2\pi} \int_{-1}^1 \frac{w(s)ds}{\sqrt{1-s^2}} + \frac{\alpha}{2}w(1) + \frac{\beta}{2} w(-1) + \sigma^2(w; t)\right] ,  \nonumber
\end{eqnarray}
where
\begin{equation}\label{eq: w-def}
w(z) = \sum_{j=1}^n r_j w_{\varepsilon_j, x} (z),
\end{equation}
and
\begin{equation} \label{eq: w-epsilon-def}
w_{\varepsilon_j, x} (z) = \sqrt{2}\pi (1_{(-\infty, x]} \star \varphi_{\varepsilon_j}) (z) = \frac{\sqrt{2}\pi}{2} + \frac{\sqrt{2}}{2i} \Big( \log (\varepsilon - i(x-z)) - \log (\varepsilon + i(x-z)) \Big).
\end{equation}
Note that $w_{\varepsilon_j, x} (z) $ has a branch cut along $x \pm i[\varepsilon, \infty)$ and is analytic in a neighborhood of $[-1,1]$. By \eqref{Heine Identity}, $\mathbb{E} \big[e^{\sum_{k=1}^N w(\lambda_k)}\big] =\frac{D_N(0; 0; w+t)}{D_N(0; 0; t)}$,  so \eqref{eq: exp-moment-1} becomes
\begin{multline} \label{eq: exp-moment-2}
\mathbb{E} \big[e^{\sum_{j=1}^n r_j \tilde{h}_{N, \varepsilon_j}(x)} \big] = \frac{D_N(0; 0; w+t)}{D_N(0; 0; t)} \exp\left[-N \int_{-1}^1 w d\mu_V- \frac{\alpha+\beta}{2\pi} \int_{-1}^1 \frac{w(s)ds}{\sqrt{1-s^2}} \right. \\ \left.  + \frac{\alpha}{2}w(1) + \frac{\beta}{2} w(-1)+\sigma^2(w; t)\right].
\end{multline}
Here, we also make use of \cite[(A.6)\&(A.10)]{CFL2021} to rewrite the convolution term as
\begin{eqnarray}
 \frac{\sqrt{2}}{2}\int_{-1}^1\varphi_\varepsilon(s-x) (\mathcal{U}t)(s) \sqrt{1-s^2}ds = -\frac{1}{2\pi}\int_{-1}^1 w_{\varepsilon, x}'(s) (\mathcal{U}t)(s) \sqrt{1-s^2}ds = \sigma^2(w_{\epsilon, x}; t),
\end{eqnarray}
where $\sigma^2(w; t)$ is defined in \eqref{eq:sigma-def}.

Before proving the next proposition, we would like to clarify the notational differences between this paper and \cite{CFL2021}. In \cite{CFL2021}, the authors adopt the notation of a canonical process by considering the same random field, denoted as $X(x)$, under different probability measures. Their aim is to provide a general framework for the theory of GMC using simplified notations. In contrast, our work focuses on a concrete JUE model. Therefore, we use a more traditional notation, considering different random fields under a fixed probability measure with a density given by \eqref{JUEdensity}. More precisely, all expressions of the form $\mathbb{E}_N e^{\gamma X(x)}$ in \cite{CFL2021} correspond to $\mathbb{E} e^{\gamma h_N(x)}$ in this paper, where the expectation is taken with respect to \eqref{JUEdensity}. Similarly, we write $\mathbb{E} e^{\gamma h_{N, \varepsilon}(x)}$ instead of $\mathbb{E}_N e^{\gamma X_{\varepsilon}(x)}$. Note that in this paper, $X(x)$ is reserved exclusively for the log-correlated field with correlation kernel \eqref{correlation kernel of X}.

\begin{proposition} \label{all the asymptotics of hN}
The modified eigenvalue counting function $\tilde{h}_N(x)$ defined in \eqref{tilde hN} satisfies all the requirements of the exponential moments  in \cite[Assumption 2.5]{CFL2021}.
\end{proposition}

\begin{proof}
We first note that, since $\tilde{h}_N(x) $ differs from $h_N(x) $ only by a deterministic quantity (cf. \eqref{tilde hN}), we have
\begin{equation} \label{eq: prop-5.2-first-relation}
 \frac{\mathbb{E} [e^{\gamma \tilde{h}_{N, \varepsilon}(x) + \gamma \tilde{h}_{N, \varepsilon'}(y)}]}{\mathbb{E} [e^{\gamma \tilde{h}_{N, \varepsilon}(x)}] \, \mathbb{E} [e^{\gamma \tilde{h}_{N, \varepsilon'}(y)}]} = \frac{\mathbb{E} [e^{\gamma h_{N, \varepsilon}(x) + \gamma h_{N, \varepsilon'}(y)}]}{\mathbb{E} [e^{\gamma h_{N, \varepsilon}(x)}] \, \mathbb{E} [e^{\gamma h_{N, \varepsilon'}(y)}]} \qquad \textrm{for }  \varepsilon, \varepsilon' \geq 0.
\end{equation}
Recall from \eqref{EV-Heine Identity1}, \eqref{EV-Heine Identity2} and \eqref{eq: exp-moment-2} that all required exponential moments can be expressed in terms of Hankel determinants. Moreover, we have the following key observation: the exponential moment $\mathbb{E} e^{\gamma h_N(x)}$ is related to a Hankel determinant whose weight function possesses a jump singularity, whereas the exponential moment of the regularized function $\mathbb{E} e^{\gamma h_{N, \varepsilon}(x)}$ corresponds to a Hankel determinant with a smooth weight function on $[-1,1]$. Then, using the asymptotic results for Hankel determinants from Propositions \ref{Asymptotics in the Separated Regime}, \ref{Asymptotics in the merging regime}, and \ref{Asymptotics in the edge}, the verification of the assumptions follows a similar approach to \cite[Sec. 2.7.2]{CFL2021}. For illustration, we present detailed proofs for three representative conditions, namely (2.9), (2.11) and (2.12)  in \cite[Assumption 2.5]{CFL2021}. The remaining conditions can be verified analogously.

\underline{Verification of \cite[(2.9)]{CFL2021}.} In our notation, we need to show that, for any fixed $\varepsilon, \varepsilon' \geq 0$,
\begin{eqnarray} \label{eq: our version of (2.9)}
\lim_{N\to\infty} \frac{\mathbb{E} [e^{\gamma \tilde{h}_{N, \varepsilon}(x) + \gamma \tilde{h}_{N, \varepsilon'}(y)}]}{\mathbb{E} [e^{\gamma \tilde{h}_{N, \varepsilon}(x)}] \, \mathbb{E} [e^{\gamma \tilde{h}_{N, \varepsilon'}(y)}]}  = e^{\gamma^2 \mathbb{E} [X_{\varepsilon}(x)X_{\varepsilon'}(y)]},
\end{eqnarray}
uniformly for $(x,y)$ in any compact subset of $[-1,1]^2$. Here, $X_\varepsilon (x)$ is defined as 
\begin{equation}
X_{\varepsilon}(x) = \int_{-1}^1 \varphi_\varepsilon (x-u) X(u)du,
\end{equation}
where the mollifier $\varphi_{\varepsilon}(x)$ is given in \eqref{mollifier}. Note that $X_{0}(x)$ = $X(x)$. Depending on the values of $\varepsilon $ and $\varepsilon'$, we will prove this result in three different cases.

Case 1: $\varepsilon = \varepsilon' = 0$. From \eqref{EV-Heine Identity1}, \eqref{EV-Heine Identity2} and \eqref{eq: prop-5.2-first-relation}, we have
\begin{eqnarray}\label{eq 2.9: subcase 1: Hankel determinant}
\frac{\mathbb{E} [e^{\gamma \tilde{h}_N(x) + \gamma \tilde{h}_N(y)}]}{\mathbb{E} [e^{\gamma \tilde{h}_N(x)}] \,\mathbb{E} [e^{\gamma \tilde{h}_N(y)}]} = \frac{D_N(x, y; \gamma, \gamma; t) D_N(0; 0; t)}{D_N(x; \gamma; t) D_N(y; \gamma; t)}.
\end{eqnarray}
Using the asymptotics for Hankel determinants from \cite[Theorem 1.3]{CG2021}, we obtain
\begin{equation} \label{eq 2.9: subcase 1}
\lim_{N\to\infty} \frac{D_N(x, y; \gamma, \gamma; t) D_N(0; 0; t)}{D_N(x; \gamma; t) D_N(y; \gamma; t)} =\left(\frac{1- xy + \sqrt{1-x^2}\sqrt{1-y^2}}{|x-y|}\right)^{\gamma^2}.
\end{equation}
In view of \eqref{correlation kernel of X}, the right-hand side coincides with $e^{\gamma^2 \mathbb{E} X(x)X(y)}$.

Case 2: $\varepsilon = 0, \varepsilon' > 0$. From \eqref{EV-Heine Identity1}, \eqref{eq: exp-moment-2} and \eqref{eq: prop-5.2-first-relation}, we have
\begin{eqnarray}
\frac{\mathbb{E} [e^{\gamma \tilde{h}_N(x) + \gamma \tilde{h}_{N, \varepsilon'}(y)}]}{\mathbb{E} [e^{\gamma \tilde{h}_N(x)}]\, \mathbb{E} [e^{\gamma \tilde{h}_{N, \varepsilon'}(y)}]}
= \frac{D_N(x; \gamma; \gamma w_{\varepsilon', y}+t) D_N(0; 0; t)}{D_N(x; \gamma; t) D_N(0; 0; \gamma w_{\varepsilon', y}+t)},
\end{eqnarray}
where $w_{\varepsilon', y}(z)$ is defined in \eqref{eq: w-epsilon-def}. We apply Proposition \ref{Asymptotics in the Separated Regime} to derive the above limit as $N \to \infty$. Specifically, to approximate $\frac{D_N(x; \gamma; \gamma w_{\varepsilon', y}+t) }{D_N(x; \gamma; t) }$, we set $\gamma_1 = \gamma, \gamma_2 = 0$ and $w_N = \gamma w_{\varepsilon', y}$. For the term $\frac{D_N(0; 0; \gamma w_{\varepsilon', y}+t)}{D_N(0; 0; t)}$, we only need to change $\gamma_1 = 0$. Therefore,  from \eqref{equation: Asymptotics in seperated regime}, we obtain
\begin{eqnarray}
\lim_{N\to \infty} \frac{D_N(x; \gamma; w+t) D_N (0; 0; t)}{D_N(x; \gamma; t) D_N(0; 0; w+t)} =  e^{\frac{\gamma}{\sqrt{2}} \mathcal{U}w(x) \sqrt{1-x^2}} =e^{\gamma^2 \mathbb{E} [X(x)X_{\varepsilon'}(y)]},
\end{eqnarray}
where we have used \cite[(A.14)]{CFL2021} in the last identity.

Case 3: $\varepsilon > 0, \varepsilon' > 0$. Similarly,  from \eqref{eq: exp-moment-2} and \eqref{eq: prop-5.2-first-relation}, we have
\begin{eqnarray}
\frac{\mathbb{E} [e^{\gamma \tilde{h}_{N, \varepsilon}(x) + \gamma \tilde{h}_{N, \varepsilon'}(y)}]}{\mathbb{E} [e^{\gamma \tilde{h}_{N, \varepsilon}(x)}]\, \mathbb{E} [e^{\gamma \tilde{h}_{N, \varepsilon'}(y)}]} = \frac{D_N(0; 0; w_1 + w_2+t) D_N(0; 0; t)}{D_N(0; 0; w_1+t) D_N(0; 0; w_2+t)},
\end{eqnarray}
with
\begin{eqnarray}
w_1 = \gamma w_{\varepsilon, x}, \quad 
w_2 = \gamma w_{\varepsilon', y}.
\end{eqnarray}
Using Proposition \ref{Asymptotics in the Separated Regime} again, we set $\gamma_1=\gamma_2 = 0$ and $w_N = w_1, w_2, w_1+w_2$, respectively. Then, we obtain 
\begin{eqnarray}
&& \lim_{N\to\infty} \frac{D_N(0; 0; w_3+t) / D_N(0; 0; t)}{\left[D_N(0; 0; w_1+t) / D_N(0; 0; t)\right]\left[D_N(0; 0; w_2+t) / D_N(0; 0; t)\right]} \nonumber\\
&&\qquad = e^{\frac{1}{2} \sigma^2(w_1 + w_2) - \sigma^2(w_1) - \sigma^2(w_2)}  =  e^{\gamma^2 \mathbb{E} [X_{\varepsilon}(x)X_{\varepsilon'}(y)]},
\end{eqnarray}
where we have used \cite[(A.9)]{CFL2021} in the last identity. This completes the verification of \eqref{eq: our version of (2.9)}.

\underline{Verification of \cite[(2.11)]{CFL2021}.} In our notation, we need to show that, for $\varepsilon_N = N^{-1+\alpha}$ with $0<\alpha<2/3$,
\begin{equation} \label{eq 2.11 in our version}
\frac{\mathbb{E} [e^{\gamma \tilde{h}_N(x) + \gamma \tilde{h}_N(y)}]}{\mathbb{E} [e^{\gamma \tilde{h}_N(x)}] \,\mathbb{E} [e^{\gamma \tilde{h}_N(y)}]} << |x-y|^{-\gamma^2},
\end{equation}
uniformly for all $(x, y)$ in a fixed compact subset of $\{(x, y)\in (-1, 1)^2: |x-y| \geq \varepsilon_N\}$.

Note that the right-hand side of the above equation coincides with that of \eqref{eq 2.9: subcase 1: Hankel determinant}. Then, when $|x-y| \geq \delta$ for some small fixed $\delta>0$, the above bound follows directly from \eqref{eq 2.9: subcase 1}.

Now consider the case $\varepsilon_N \leq |x-y| \leq \delta$. We first rewrite \eqref{eq 2.9: subcase 1: Hankel determinant} as 
\begin{eqnarray}
\frac{\mathbb{E} [e^{\gamma \tilde{h}_N(x) + \gamma \tilde{h}_N(y)}]}{\mathbb{E} [e^{\gamma \tilde{h}_N(x)}] \, \mathbb{E} [e^{\gamma \tilde{h}_N(y)} ] } = \frac{D_N(x, y; \gamma, \gamma; t)}{D_N(x; 2\gamma; t)} \frac{D_N(x; 2\gamma; t)}{D_N(x; \gamma; t)} \frac{D_N(x; 0; t)}{D_N(y; \gamma; t)}.
\end{eqnarray}
Applying Proposition \ref{Asymptotics in the merging regime} with $\gamma_1 = \gamma_2 = \gamma$, $x_1 = x$ and $x_2 = y$, we have
\begin{eqnarray}
\frac{D_N(x, y; \gamma, \gamma; t)}{D_N(x; 2\gamma; 0)} &=& \exp\left( \sqrt{2}N \gamma \int_{x}^{y} \frac{du}{\sqrt{1-u^2}} - \gamma^2 \log |x-y| N + O(1)\right),
\end{eqnarray}
as $N\to\infty$, uniformly for $\varepsilon_N \leq|x-y| \leq \delta$. Using the asymptotics of the Hankel determinants in \cite[Theorem 1.3]{CG2021}, we also have
\begin{eqnarray}
& & \frac{D_N(x; 2\gamma; t)}{D_N(x; \gamma; t)} \frac{D_N(x; 0; t)}{D_N(y; \gamma; t)} =  \exp\left( -\sqrt{2}\gamma N \int_{x}^{y} \frac{du}{\sqrt{1-u^2}} + \gamma^2\log N + O(1)\right),
\end{eqnarray}
as $N\to\infty$. Combining the above two equations yields \eqref{eq 2.11 in our version}, 
uniformly for $(x, y) \in (-1, 1)^2$ with $\varepsilon_N \leq |x-y| \leq \delta$. This completes the verification of \eqref{eq 2.11 in our version}.

\underline{Verification of \cite[(2.12)]{CFL2021}.} In our notation, we need to show that
\begin{eqnarray}
& &\frac{\mathbb{E} \left[ e^{\gamma \tilde{h}_{N, \varepsilon}(x) + \gamma \tilde{h}_{N, \varepsilon'}(y)} e^{\sum_{j=1}^n r_j \tilde{h}_{N, \eta_j}(x_j)} \right] }{\mathbb{E} [e^{\gamma \tilde{h}_{N, \varepsilon}(x) + \gamma \tilde{h}_{N, \varepsilon'}(y)} ]} \label{eq 2.12 in our version}\\
&=& (1+o(1)) \exp \left(\frac{1}{2}\sum_{i, j=1}^n r_i r_j \mathbb{E} [X_{\eta_i}(x_i) X_{\eta_j}(x_j)] + \sum_{i=1}^n \gamma r_i \bigg(\mathbb{E} [X(x) X_{\eta_i}(x_i)] + \mathbb{E} [X(y) X_{\eta_i}(x_i)]\bigg)\right), \nonumber
\end{eqnarray}
as $N\to\infty$, uniformly for all $\eta_j \in (\varepsilon_N, 1]$, $x_j$ in any compact subset of $(-1, 1)$ for $j=1, \cdots, n$, and $(x, y)$ in a compact subset of $(-1, 1)^2$ with $|x-y| \geq \rho$.

We first consider the case $\varepsilon = \varepsilon' = 0$. Recalling the definition of $\tilde{h}_N$ in \eqref{tilde hN}, we have
\begin{eqnarray}
& & \frac{\mathbb{E} \big[ e^{\gamma \tilde{h}_N(x) + \gamma \tilde{h}_N(y)} e^{\sum_{j=1}^n r_j \tilde{h}_{N, \eta_j}(x_j)}\big]}{\mathbb{E}[e^{\gamma \tilde{h}_N(x) + \gamma \tilde{h}_N(y)}]} \nonumber\\
&=& \exp\left(-N \int_{-1}^1 w d\mu_V-\frac{\alpha+\beta}{2\pi} \int_{-1}^1 \frac{w(s)ds}{\sqrt{1-s^2}} + \frac{\alpha}{2}w(1) + \frac{\beta}{2}w(-1) - \sigma^2(w; t)\right) \nonumber\\
& & \cdot \frac{\mathbb{E}\big[e^{\gamma h_N(x) + \gamma h_N(y)} e^{\sum_{j=1}^n r_j h_{N, \eta_j}(x_j)}\big]}{\mathbb{E} [e^{\gamma h_N(x) + \gamma h_N(y)}]}, \label{eq 2.12 in our version: first simplification}
\end{eqnarray}
where $w$ is given in \eqref{eq: w-def}. By \eqref{EV-Heine Identity2}, we get
\begin{equation}\label{eq 2.12: Henkel determinants}
\frac{\mathbb{E}\big[e^{\gamma h_N(x) + \gamma h_N(y)} e^{\sum_{j=1}^n r_j h_{N, \eta_j}(x_j)}\big]}{\mathbb{E}[e^{\gamma h_N(x) + \gamma h_N(y)}]} = \frac{D_N(x, y; \gamma, \gamma; w+t)}{D_N(x, y; \gamma, \gamma; t)}.
\end{equation}
Applying Proposition \ref{Asymptotics in the Separated Regime} with $\gamma_1 = \gamma_2 = \gamma$, $x_1 = x$ and $x_2 = y$, we have from the above two formulas
\begin{eqnarray}
& & \frac{\mathbb{E}\big[e^{\gamma \tilde{h}_N(x) + \gamma \tilde{h}_N(y)} e^{\sum_{j=1}^n r_j \tilde{h}_{N, \eta_j}(x_j)}\big]}{\mathbb{E}[e^{\gamma \tilde{h}_N(x) + \gamma \tilde{h}_N(y)}]} \nonumber\\
&=& (1+o(1)) \exp\left(\frac{1}{2}\sigma^2(w) + \frac{\gamma}{\sqrt{2}}\left(\sqrt{1-x^2}(\mathcal{U}w)(x) + \sqrt{1-y^2}(\mathcal{U}w)(y)\right)\right), \label{eq 2.12 in our version: step 2}
\end{eqnarray}
as $N\to\infty$, uniformly for all $\eta_j \in (\varepsilon_N, 1]$, $x_j$ in any compact subset of $(-1, 1)$ for $j=1, \cdots, n$, and $(x, y)$ in a compact subset of $(-1, 1)^2$ with $|x-y| \geq \rho$.

Recalling \cite[(A.9) \& (A.14)]{CFL2021}
\begin{equation}
\mathbb{E} [X_\eta(x) X_{\varepsilon}(y)] = \sigma^2(w_{\eta, x}; w_{\varepsilon, y}), \qquad \mathbb{E} [X_\eta(x) X(y)] = \frac{1}{\sqrt{2}} (\mathcal{U}w_{\eta, x})(y) \sqrt{1-y^2},
\end{equation}
then we have
\begin{equation}
\sigma^2(w) = \sum_{i, j = 1}^n r_i r_j \mathbb{E}[X_{\eta_i}(x_i) X_{\eta_j}(x_j)]
\end{equation}
and
\begin{equation}
\frac{\gamma}{\sqrt{2}}\sqrt{1-x^2}\mathcal{U}w(x) = \sum_{i=1}^n \gamma r_i \mathbb{E} [X(x)X_{\eta_i}(x_i)].
\end{equation}
Substituting the above two identities into \eqref{eq 2.12 in our version: step 2} gives us \eqref{eq 2.12 in our version} when $\varepsilon = \varepsilon' = 0$.

When either $\varepsilon > 0$ or $\varepsilon' > 0$, we incorporate the factors $e^{\gamma \tilde{h}_{N, \varepsilon}(x)}$ or $e^{\gamma \tilde{h}_{N, \varepsilon'}(y)}$ into the term $e^{\sum_{j=1}^n r_j \tilde{h}_{N, \eta_j}(x_j)}$ in \eqref{eq 2.12 in our version}, and replace $w$ by $w + \gamma w_{\varepsilon, x}$ or $w + \gamma w_{\varepsilon', y}$, respectively. Consequently, the Hankel determinant expression in \eqref{eq 2.12: Henkel determinants} becomes 
\begin{equation*}
\frac{D_N(x, y; 0, \gamma; w + \gamma w_{\varepsilon, x}+t)}{D_N(x, y; 0, \gamma; t)} \quad \textrm{or} \quad \frac{D_N(x, y; \gamma, 0; w + \gamma w_{\varepsilon', y}+t)}{D_N(x, y; \gamma, 0; t)} 
\end{equation*}
Using a similar approximation as above, we obtain \eqref{eq 2.12 in our version} from the asymptotics of these Hankel determinants. The case where both $\varepsilon$ and $\varepsilon'$ are positive is estimated in the same way. This completes the verification of \eqref{eq 2.12 in our version}.
\end{proof}

\subsection{Proof of Theorem \ref{main theorem for hN}: The Maximum of $h_N(x)$}

Now we are ready to prove Theorem \ref{main theorem for hN}. First we show that $d\mu_N^\gamma(x)$ defined in \eqref{eq:dmu-N} to the GMC measure.

\begin{proposition} \label{prop-GMC-convergence}
Let $0<\gamma<\sqrt{2}$, the eigenvalue counting function $h_N(x)$ be defined in \eqref{hN}, and $X(x)$ be the log-correlated Gaussian field with correlation kernel \eqref{correlation kernel of X}. Then, the measure $d\mu_N^\gamma(x)$  in \eqref{eq:dmu-N}  converges to the GMC measure $d\mu^\gamma(x) = \frac{e^{\gamma X(x)}}{\mathbb{E}e^{\gamma X(x)}}$ as $N \to \infty$ in law, with respect to the weak topology. 
\end{proposition}

\begin{proof}
Since $h_N(x)$ and $\tilde{h}_N(x)$ only differ by a deterministic quantity (cf. \eqref{tilde hN}), we have
\begin{equation} \label{eq: GMC and two h}
d\mu_N^\gamma(x) = \frac{e^{\gamma h_N(x)}}{\mathbb{E} [e^{\gamma h_N(x)}]}dx = \frac{e^{\gamma \tilde{h}_N(x)}}{\mathbb{E} [e^{\gamma \tilde{h}_N(x)}]}dx.
\end{equation}
By Proposition \ref{all the asymptotics of hN}, $\tilde{h}_N(x)$ satisfies all the conditions in \cite[Assumption 2.5]{CFL2021}.  Hence by \cite[Theorem 2.6]{CFL2021}, the convergence of $d\mu_N^\gamma(x)$  to the GMC $d\mu^\gamma(x)$ measure is justified, which completes the proof of the proposition.
\end{proof}

The convergence of $d\mu_N^\gamma(x)$ to the GMC measure gives us the lower bound for the maximum of $h_N(x)$. 

\begin{proposition} \label{lower bound of hN}
Let $h_N(x)$ be the eigenvalue counting function  defined in \eqref{hN}. Then, for any compact set $A\subset (-1, 1)$ with a positive Lebesgue measure, we have, for any $\delta>0$,
\begin{equation} \label{eq:hn-lowerbound}
\lim_{N\to\infty} \mathbb{P} \left[ \max_{ x \in A} h_N(x) \geq (\sqrt{2}-\delta) \log N\right] = 1.
\end{equation}
\end{proposition}

\begin{proof}
By Propositions \ref{Assumption 3.1} and \ref{prop-GMC-convergence}, the function $h_N(x)$ satisfies all the conditions of \cite[Theorem 3.4]{CFL2021}. Then, the lower bound \eqref{eq:hn-lowerbound} follows directly.
\end{proof}

\begin{remark}
The above proposition provides a lower bound for the maximum of $h_N(x)$. This result is much stronger than Theorem \ref{main theorem for hN} because it implies that the lower bound of $\max |h_N(x)|$ holds in probability on any compact subset of $[-1, 1]$ with positive Lebesgue measure, not just on the entire interval.
\end{remark}

The upper bound for the maximum of $h_N(x)$ can be obtained from its definition, without using the convergence to the GMC measure. 
\begin{proposition} \label{prop: upper bound of hN}
Let $h_N(x)$ be the eigenvalue counting function  defined in \eqref{hN}. Then, we have, for any $\delta>0$, 
\begin{equation} \label{eq:hn-upperbound}
\lim_{N\to\infty} \mathbb{P} \left[ \max_{ x \in [-1, 1]} h_N(x) \leq (\sqrt{2}+\delta) \log N\right] = 1.
\end{equation}
\end{proposition}

\begin{proof}
Recall the definition of the $j$-th percentile $\kappa_j$ defined in \eqref{eq: kappa-j-def}, and set $\kappa_0 = 0$ and $\kappa_{N+1} = 1$. Then, for any $x\in (\kappa_j, \kappa_{j+1})$, $j = 0, \cdots N$, we have
\begin{equation} \label{hN bound}
h_N(\kappa_j) - \sqrt{2}\pi \leq h_N(x) \leq h_N(\kappa_{j+1}) + \sqrt{2}\pi.
\end{equation}
Consider the probability
\begin{eqnarray*}
& & \mathbb{P} \Big[\max_{(-1, 1)} h_N(x) \geq (\sqrt{2}+\delta)\log N \Big] \nonumber\\
&& = \sum_{j} \mathbb{P}(x^*\in (\kappa_j, \kappa_{j+1})) 
\mathbb{P} \Big(h_N(x^*)\geq (\sqrt{2}+\delta) \log N \Big| x^*\in (\kappa_j, \kappa_{j+1}) \Big),
\end{eqnarray*}
where $x^*$ is a random variable such that $h_N(x^*) \geq (\sqrt{2}+\delta) \log N$. As $\mathbb{P}(x^*\in (\kappa_j, \kappa_{j+1})) \leq 1$ for all $j$, we get
\begin{equation*}
\mathbb{P} \Big[  \max_{x\in(-1, 1)} h_N(x) \geq (\sqrt{2}+\delta)\log N \Big] \leq \sum_{j} \mathbb{P} \Big(h_N(x^*)\geq (\sqrt{2}+\delta) \log N \Big| x^*\in (\kappa_j, \kappa_{j+1}) \Big) .
\end{equation*}
It then follows from \eqref{hN bound} that
\begin{equation*}
\sum_{j} \mathbb{P} \Big(h_N(x^*)\geq (\sqrt{2}+\delta) \log N \Big| x^*\in (\kappa_j, \kappa_{j+1}) \Big) \leq \sum_{j} \mathbb{P} (h_N(\kappa_j) \geq (\sqrt{2}+\delta)\log N - \sqrt{2}\pi).
\end{equation*}
By Markov's inequality, there exist positive constant $C$ and $a$, such that $\mathbb{P} (h_N(\kappa_j) \geq (\sqrt{2}+\delta)\log N - \sqrt{2}\pi) \leq C \frac{\mathbb{E}[e^{a h_N(\kappa_j)}]}{e^{a (\sqrt{2}+\delta)\log N}}$. Applying Proposition \ref{Assumption 3.1} and combining the above three formulas, we have 
\begin{equation}
\mathbb{P} \Big[ \max_{x\in (-1, 1)} h_N(x) \geq (\sqrt{2}+\delta)\log N \Big] \leq C N^{\frac{a^2}{2} + 1 - a(\sqrt{2} +\delta)}.
\end{equation}
Choosing $a = \sqrt{2}$, we find
\begin{eqnarray}
\lim_{N\to\infty} \mathbb{P} \left[\max_{(-1, 1)} h_N(x)\geq (\sqrt{2} + \delta)\log N \right] = 0,
\end{eqnarray}
which yields \eqref{eq:hn-upperbound}. This completes the proof of the proposition.
\end{proof}

Combining the above two propositions, we conclude the proof of Theorem \ref{main theorem for hN}.

\subsection{Proof of Theorem \ref{Main Theorem}: The Eigenvalue Rigidity}

We now proceed to the proof of our main result, Theorem \ref{Main Theorem}. Recall the definition of $h_N(x)$ from \eqref{hN}, we have
\begin{equation} \label{hN(lambda)}
h_N(\lambda_j)  = \sqrt{2} \pi N (F(\kappa_j) - F(\lambda_j)) + \frac{\sqrt{2}}{2}\pi.
\end{equation}
When $ j \asymp N$, the above formula yields
\begin{equation} \label{eq: lower-bound-approx}
h_N(\lambda_j)  = \sqrt{2}N F'(\kappa_j)(\kappa_j - \lambda_j) + O(1).
\end{equation}
Then, the lower bound in \eqref{eq:main-theorem} follows directly from Proposition \ref{lower bound of hN}. The rest of this section is therefore dedicated to proving the upper bound in \eqref{eq:main-theorem}. In what follows, we use $C, C'$ to denote positive constants independent of $N$, whose values not necessarily equal.

\subsubsection{Bulk Rigidity: The Method of Iteration}

We begin by establishing the following upper bound in the bulk, extending its validity near the endpoints $\pm 1$.

\begin{proposition} \label{Main theorem: Eigenvalue Rigidity on the Bulk}
For any $\delta > 0$ and $F(x)$ defined in \eqref{distribution of eq measure}, we have
\begin{equation} \label{eq: prop5.7-formula}
\lim_{N\to\infty}\mathbb{P}\left(F'(\kappa_j)|\lambda_j - \kappa_j| \leq \frac{(1+\varepsilon)\log N}{\pi N} \text{ for } j = N^\delta, \cdots, N-N^{\delta}\right) = 1.
\end{equation}
\end{proposition}

\begin{proof}
From \eqref{hN(lambda)} and Proposition \ref{prop: upper bound of hN}, it is obvious that 
\begin{equation} \label{F(kappa)-F(lambda)}
\lim_{N\to\infty}\mathbb{P} \left( |F(\kappa_j)-F(\lambda_j)| \leq \frac{(1+\varepsilon)\log N}{\pi N} \textrm{ for all } j \right) = 1 .
\end{equation}
Define the event
\begin{equation} \label{eq:event-Bn0}
\mathcal{B}_{N, 0} = \left\{|\lambda_j-\kappa_j|\leq \frac{(1+\varepsilon)\log N}{N} \text{ for all } j\right\}.
\end{equation}
Since $F'(x)\geq \frac{1}{\pi}$ for all $x\in [-1,1]$, by mean-value theorem and \eqref{F(kappa)-F(lambda)}, we have
\begin{equation}
\lim_{N\to\infty}\mathbb{P} (\mathcal{B}_{N, 0}) = 1.
\end{equation}
Next, we introduce $l_N = N^{2/3+\delta}$, where $\delta>0$ can be arbitrarily small, and define another event
\begin{equation} \label{eq:event-Bn1}
\mathcal{B}_{N, 1} =\left\{ F'(\kappa_j)|\lambda_j-\kappa_j|\leq \frac{(1+\varepsilon)\log N}{\pi N} \text{ for } j=l_N, \cdots, N-l_N \right\} \bigcap \mathcal{B}_{N, 0}.
\end{equation}
We claim that 
\begin{equation}
\lim_{N\to\infty}\mathbb{P} (\mathcal{B}_{N, 1}) = 1.
\end{equation}
To see this, we recall that
\begin{equation} \label{eq:F'-F""}
F'(x) = \frac{1}{\pi} (1-x^2)^{-1/2}, \quad F''(x) = -\frac{1}{\pi} x (1-x^2)^{-3/2}.
\end{equation}
When $\zeta_j$ lies between $\lambda_j$ and $\kappa_j$, under the condition $\mathcal{B}_{N, 0}$, we have
\begin{eqnarray}
\sqrt{2}\pi N|F''(\zeta_j)|(\lambda_j-\kappa_j)^2 &\leq& C N (l_N^2 N^{-2})^{-3/2}(\log N)^2 N^{-2} \leq C l_N^{-3} (\log N)^2 N^2,
\end{eqnarray}
for some constant $C>0$, which is of order $o(1)$ by our choice of $l_N$. Therefore, from \eqref{hN(lambda)}, we get
\begin{eqnarray}\label{Taylor expansion of hN}
|h_N(\lambda_j)| &=& \left|\sqrt{2}\pi N F'(\kappa_j)(\lambda_j-\kappa_j) + \sqrt{2}\pi N \frac{F''(\zeta_j)}{2}(\lambda_j-\kappa_j)^2 + \frac{\sqrt{2}}{2}\pi\right| \nonumber\\
&\geq & \sqrt{2}\pi N F'(\kappa_j)|\lambda_j-\kappa_j| - \sqrt{2}\pi N \frac{|F''(\zeta_j)|}{2}(\lambda_j-\kappa_j)^2 - \frac{\sqrt{2}}{2}\pi,
\end{eqnarray}
Using \eqref{eq:hn-upperbound}, we obtain from the above two formulas that $\lim_{N\to\infty}\mathbb{P} (\mathcal{B}_{N, 1}) = 1$. 

This approximation provides us with the upper bound in \eqref{eq:main-theorem} for a region slightly larger than the bulk. We now extend this result towards the endpoints $\pm 1$ as much as possible. By symmetry, we will focus on the left endpoint $-1$, corresponding to $j\leq l_N$. The results for $j\geq N-l_N$ can be derived similarly. Furthermore, if $\lambda_j \leq \kappa_j$, we have
\begin{equation} \label{eq 1: only need one side}
|F(\kappa_j) - F(\lambda_j)| = |F'(\xi_j)|(\kappa_j - \lambda_j), \qquad \xi_j\in (\lambda_j, \kappa_j).
\end{equation}
Since $F'(x)>0$ is monotonically decreasing near $-1$ (cf. \eqref{eq:F'-F""}), it follows that
\begin{equation}\label{eq 2: only need one side}
F'(\kappa_j) (\kappa_j - \lambda_j) \leq F'(\xi_j)(\kappa_j - \lambda_j). 
\end{equation}
The desired bound \eqref{eq: prop5.7-formula} then follows directly from \eqref{F(kappa)-F(lambda)} and the above inequality. Thus,  we only need to consider the case $\lambda_j>\kappa_j$ for the remainder of the proof.

Let us first consider $j = l_N$. From \eqref{eq: kappa-j-def}, we know $\kappa_j = -\cos \frac{(j-1/2)\pi}{N}$. Conditioning on the event $\mathcal{B}_{N, 1}$ from \eqref{eq:event-Bn1} at $j=l_N$, we have
\begin{eqnarray}
\lambda_{l_N} - \kappa_{l_N} \leq \frac{1}{F'(\kappa_{l_N})} \cdot \frac{(1+\varepsilon)\log N}{\pi N} = \pi \sin \frac{(l_N-1/2)\pi}{N} \frac{(1+\varepsilon)\log N}{\pi N},
\end{eqnarray}
which is of order $O(\frac{\log N}{N^{4/3 - \delta}})$ based on our choice of $l_N$. Next, for $j\leq l_N$, since there exists $\xi_j \in (\kappa_j, \lambda_j)$ such that $F'(\xi_j)(\lambda_j - \kappa_j) = F(\kappa_j) - F(\lambda_j)$, we get
\begin{multline}
\xi_j < \lambda_j \leq \lambda_{l_N} \leq  \kappa_{l_N} + \frac{C\log N}{N^{4/3-\delta}} =-\cos \frac{(l_N-1/2)\pi}{N} + \frac{C\log N}{N^{4/3-\delta}} \\
= -1 + C' N^{-2/3+2\delta} + O(N^{4\delta- 4/3}) \qquad \text{ as } N\to\infty,
\end{multline}
Then, we have
\begin{equation}
F'(\xi_j) = \frac{1}{\pi} (1-\xi_j^2)^{-1/2} \geq 
C' N^{1/3 - \delta} .
\end{equation}
Then, conditioned on the event $\mathcal{B}_{N, 1}$, we obtain
\begin{equation}\label{lambda-kappa}
\lambda_j - \kappa_j \leq C(F'(\xi_j))^{-1}\frac{\log N}{ N} \leq \frac{C' \log N}{N^{4/3-\delta}} \qquad \textrm{for } j\leq l_N,
\end{equation}
which improves the initial bound from \eqref{eq:event-Bn0} when $j \leq l_N$.

We now iterate this refinement. Set $l^{(1)}_N = N^{4/9+\delta}$ and define
\begin{equation} \label{eq:event-Bn2}
\mathcal{B}_{N, 2} = \left\{F'(\kappa_j)|\lambda_j-\kappa_j| \leq \frac{(1+\varepsilon)\log N}{\pi N} \text{ for } j= l^{(1)}_N, \cdots, l_N.\right\} \bigcap \mathcal{B}_{N, 1}.
\end{equation}
By \eqref{eq:F'-F""} and the fact that $\zeta_j$ is between $\kappa_j$ and $\lambda_j$, we have
\begin{equation}
|F''(\zeta_j)| \leq C \left((l_N^{(1)})^2N^{-2}\right)^{-3/2} \qquad \textrm{for } j\geq l_N^{(1)}.
\end{equation}
Using this bound together with \eqref{lambda-kappa}, we get, for some constant $C>0$,
\begin{eqnarray}
\sqrt{2}\pi N |F''(\zeta_j)|(\lambda_j - \kappa_j)^2 &\leq & C N ((l^{(1)}_N)^{2} N^{-2})^{-3/2} (\log N)^{2} N^{-8/3+2\delta} \nonumber\\
&\leq& C (l^{(1)}_N)^{-3}(\log N)^{2} N^{4/3+2\delta}. 
\end{eqnarray}
Since $l^{(1)}_N = N^{4/9+\delta}$, the right-hand side is of order $o(1)$ as $N \to \infty$. Combining this estimate with \eqref{lambda-kappa}, \eqref{Taylor expansion of hN} and the upper bound of $h_N(x)$ in \eqref{eq:hn-upperbound}, we obtain $\lim_{N\to\infty} \mathbb{P}(\mathcal{B}_{N, 2}) = 1$.

By repeating this procedure $m-1$ times, we arrive at
\begin{equation} \label{eq:event-BnM-limit}
\lim_{N\to\infty}\mathbb{P} (\mathcal{B}_{N, m}) = 1,
\end{equation}
where 
\begin{equation}
\mathcal{B}_{N, m+1} = \left\{F'(\kappa_j)|\lambda_j-\kappa_j| \leq \frac{(1+\varepsilon)\log N}{\pi N} \text{ for } j= l^{(m)}_N, \cdots, l^{(m-1)}_N.\right\} \bigcap \mathcal{B}_{N, m},
\end{equation}
and $l^{(m)}_N = N^{(\frac{2}{3})^m + \delta}$. Since $\delta > 0$ can be chosen arbitrarily small and $(2/3)^m \to 0$, the iteration implies that \eqref{eq: prop5.7-formula} holds for $N^\delta \leq j \leq N-l_N$. Similarly, we can show that the result holds for  $ N-l_N \leq j \leq N-N^\delta$. This concludes the proof of the proposition.
\end{proof}

\subsubsection{Edge Rigidity: Refinement of the Bound of $h_N(x)$}

In this section, we establish the upper bound \eqref{eq:main-theorem} near the edge. As a preliminary step, we require the following two lemmas concerning the bounds of the eigenvalue counting function $h_N(x)$.

\begin{lemma} \label{hN near edge first lemma}
Let $h_N(x)$ be the eigenvalue counting function  defined in \eqref{hN}. Then, for any constant $C>0$, there exists a sufficiently small $\delta>0$ such that
\begin{equation}
\lim_{N\to\infty} \mathbb{P} [\max_{x\leq \kappa_{2N^{\delta}}} h_N(x) \geq C\log N] = 0.
\end{equation}
By symmetry, similar estimates hold for $-h_N(x)$.
\end{lemma}

\begin{proof}
We start with \eqref{hN bound}, which implies
\begin{eqnarray}
\mathbb{P} \left(\max_{x\leq \kappa_{2N^{\delta}}} h_N(x) \geq C\log N \right) &\leq& \mathbb{P} \left(\max_{j\leq 2N^\delta} h_N(\kappa_j) \geq C\log N - \sqrt{2}\pi\right) \nonumber \\
&\leq& \sum_{j\leq 2N^{\delta}} \mathbb{P} \left(h_N(\kappa_j)\geq C\log N -\sqrt{2}\pi\right).
\end{eqnarray}
By Markov's inequality, we have, for a constant $\gamma > 0$,
\begin{equation}
\mathbb{P} \left(\max_{x\leq \kappa_{2N^{\delta}}} h_N(x) \geq C\log N \right)
\leq  \sum_{j\leq 2N^{\delta}} C' \frac{\mathbb{E} [e^{\gamma h_N(\kappa_j)}]}{N^{\gamma C}}.
\end{equation}
We split the above sum into two parts based on the location of $\kappa_j$. Define the threshold $t_N = -1 + N^{-2} \log\log N$, and let
\begin{align*}
\mathcal{P}_1 = \sum_{j : \, t_N \leq \kappa_j < \kappa_{2N^\delta} } C' \frac{\mathbb{E} e^{\gamma h_N(\kappa_j)}}{N^{\gamma C}}, \qquad 
\mathcal{P}_2 = \sum_{j : \, \kappa_j < t_N} C' \frac{\mathbb{E} e^{\gamma h_N(\kappa_j)}}{N^{\gamma C}}.
\end{align*}
We now bound $\mathcal{P}_1$ and $\mathcal{P}_2$ separately. For $\mathcal{P}_1$, using \eqref{hN moment bulk}, we have
\begin{eqnarray}
\mathcal{P}_1 \leq  \sum_{j : \, t_N \leq \kappa_j < \kappa_{2N^\delta} }  C' \frac{(N F(\kappa_j))^{\gamma^2 / 2}}{N^{\gamma C}}.
\end{eqnarray}
By letting $\gamma = \sqrt{2}$, we obtain
\begin{eqnarray}
\mathcal{P}_1 \leq \sum_{j\leq 2N^{\delta}} C' \frac{j-1/2}{N^{\sqrt{2} C}} \leq C' N^{2\delta - \sqrt{2}C}.
\end{eqnarray}
Hence, by choosing $0 < \delta < \frac{\sqrt{2}C}{2}$, we have $\mathcal{P}_1 \to 0$ as $N\to\infty$. 

For $\mathcal{P}_2$, using \eqref{hN moment edge}, we have
\begin{eqnarray}
\mathcal{P}_2 \leq \sum_{j : \, \kappa_j < t_N} C' \frac{e^{\sqrt{2}\pi\gamma C (\log\log N)^{1/2}} R_{\gamma} (C(\log\log N)^{1/2})^{\gamma^2 / 2}}{N^{\gamma C}}.
\end{eqnarray}
Since the number of such indices $j$ is of the order $ O\left( N^{-1} (\log\log N)^{1/2} \right)$, we have 
\begin{eqnarray}
\mathcal{P}_2 \leq  C' N^{-1}(\log\log N)^{1/2} \frac{e^{\sqrt{2}\pi\gamma C (\log\log N)^{1/2}} R_{\gamma} (C(\log\log N)^{1/2})^{\gamma^2 / 2}}{N^{\gamma C}},
\end{eqnarray}
which tends to $0$ as $N\to\infty$. This concludes the proof of the lemma.
\end{proof}

When $x$ is closer to the edge $-1$, the approximation in the above lemma can be further improved.

\begin{lemma} \label{hN near edge second lemma}
Let $h_N(x)$ be the eigenvalue counting function  defined in \eqref{hN}. Then, for any constant $C>0$, we have
\begin{equation}
\lim_{N\to\infty} \mathbb{P} \left[\max_{x\leq \kappa_{2\log N}} h_N(x) \geq C (\log N)^{1/2}\right] = 0.
\end{equation}
By symmetry, similar estimates hold for $-h_N(x)$.
\end{lemma}

\begin{proof} 
The proof follows a similar approach to that of the above lemma. From \eqref{hN bound}, we have
\begin{eqnarray}
\mathbb{P} \left(\max_{x\leq 2 \log N} h_N(x) \geq C (\log N)^{1/2} \right) 
\leq \sum_{j\leq 2\log N} \mathbb{P} \left(h_N(\kappa_j)\geq C (\log N)^{1/2} -\sqrt{2}\pi\right).
\end{eqnarray}
By Markov's inequality, we have, for a constant $\gamma > 0$,
\begin{multline}
\mathbb{P} \left(\max_{x\leq \kappa_{2\log N}} h_N(x) \geq C (\log N)^{1/2} \right)
\leq  \sum_{j\leq 2\log N} C' \frac{\mathbb{E} e^{\gamma h_N(\kappa_j)}}{e^{\gamma C (\log N)^{1/2}}} \\
=\sum_{j: \, t_N \leq \kappa_j \leq \kappa_{2\log N}} C' \frac{\mathbb{E} e^{\gamma h_N(\kappa_j)}}{e^{\gamma C (\log N)^{1/2}}} + \sum_{j : \, \kappa_j \leq t_N} C' \frac{\mathbb{E} e^{\gamma h_N(\kappa_j)}}{e^{\gamma C (\log N)^{1/2}}}
\end{multline}
with $t_N=-1+N^{-2}\log\log N$. The rest of the proof is similar to that of Lemma \ref{hN near edge first lemma}, and we omit the details here.
\end{proof}

Using the two lemmas above, we now prove the following proposition concerning eigenvalue rigidity near the edge.

\begin{proposition} \label{main proposition for edge rigidity}
For $F(x)$ defined in \eqref{distribution of eq measure}, there exists $\delta > 0$ such that
\begin{equation} \label{eq: main proposition for edge rigidity}
\lim_{N\to\infty} \mathbb{P} \left(F'(\kappa_j)|\lambda_j-\kappa_j| \leq \frac{(1+\varepsilon)\log N}{\pi N} \text{ for } 1\leq j \leq N^\delta \right) = 1.
\end{equation}
\end{proposition}

\begin{proof}
Recalling the discussion in \eqref{eq 1: only need one side} and \eqref{eq 2: only need one side}, we only need to consider the case $\lambda_j>\kappa_j$. To establish the bound in \eqref{eq: main proposition for edge rigidity}, we claim that it suffices to show there exists a $\delta>0$ such that
\begin{equation}\label{edge: Prob: lambda-kappa}
\lim_{N\to\infty} \mathbb{P} \left(\lambda_j-\kappa_j\leq \frac{k\log N}{4N^2} \quad\textrm{ for } j\leq N^\delta\right) = 1.
\end{equation}
To see this, note that
\begin{equation}
F'(\kappa_j) = \frac{1}{\pi} (1-\kappa_j^2)^{-1/2} = \frac{1}{\pi \sin \frac{(j-1/2)\pi}{N}} \sim \frac{N}{(j-1/2)\pi^2}
\end{equation}
as $N\to\infty$. Then, for arbitrarily small $\varepsilon>0$, we have
\begin{equation}\label{Edge estimation for F prime}
F'(\kappa_j) \leq \frac{(1+\varepsilon)N}{(j-1/2)\pi^2}
\end{equation}
for sufficiently large $N$. Combining this with \eqref{edge: Prob: lambda-kappa}, we obtain \eqref{eq: main proposition for edge rigidity}. Therefore, in the remaining part of the proof, we will focus on demonstrating \eqref{edge: Prob: lambda-kappa}. Based on the preceding two lemmas, we break this proof into two steps.

\textbf{Step 1: Consider the regime where $\log N\leq j\leq N^\delta$.} Assume that there exists $j$ such that $\lambda_j - \kappa_j > \frac{j\log N}{4N^2}$. From the definition of $h_N(x)$ in \eqref{hN}, we have
\begin{equation}\label{same}
\frac{1}{\sqrt{2}\pi N} h_N\left(\kappa_j + \frac{j\log N}{4N^2}\right) \leq F(\kappa_j) - F\left(\kappa_j+\frac{j\log N}{4N^2}\right)
\end{equation}
for sufficiently large $N$. Recall that
\begin{equation}\label{same2}
\kappa_j + \frac{j\log N}{4N^2} = -1+\frac{(j-1/2)^2\pi^2}{2N^2}+\frac{j\log N}{4N^2} + O(j^4/N^4)
\end{equation}
as $N\to\infty$. Since $\log N \leq j \leq N^\delta$, the term of order $j^2/N^2$ dominates the expansion. Therefore, by the mean value theorem and the behavior of $F(x)$ near $-1$, there exists a constant $C > 0$ such that
\begin{equation} \label{same3}
F\left(\kappa_j+\frac{j\log N}{4N^2}\right) - F(\kappa_j) \geq C\cdot \frac{N}{j}\cdot \frac{j\log N}{4N^2} = C\frac{\log N}{4N}.
\end{equation}
Substituting this into \eqref{same} yields
\begin{equation}\label{star}
\frac{1}{\sqrt{2}\pi N} h_N\left(\kappa_j + \frac{j\log N}{N^2}\right) \leq -C\frac{\log N}{N}.
\end{equation}
As $\kappa_j = -\cos \frac{(j-1/2) \pi}{N}$, we have $
\kappa_j+\frac{j\log N}{N^2} \leq \kappa_{2j}\leq \kappa_{2N^\delta},
$
 which implies that $\min_{x\leq \kappa_{2N^\delta}} h_N(x) \leq -C\log N$ for some $C>0$. By Lemma \ref{hN near edge first lemma}, the probability of this event tends to $0$. Hence, we can conclude that for sufficiently small $\delta>0$, the inequality \eqref{edge: Prob: lambda-kappa} holds for  $\log N\leq j\leq N^\delta$.

\textbf{Step 2: Consider the regime where $1\leq j\leq \log N$.} Again, assume that there exists $j$ such that $\lambda_j - \kappa_j > \frac{j\log N}{4N^2}$. The formulas \eqref{same} and \eqref{same2} still hold. However, since $1\leq j\leq \log N$, the term $\frac{j\log N}{4N^2}$ in expansion \eqref{same2} dominates over $\frac{j^2}{N^2}$. Consequently, the approximation in \eqref{same3} is modified to
\begin{equation}
F\left(\kappa_j+\frac{j\log N}{4N^2}\right) - F(\kappa_j) \geq C\cdot \frac{2N}{\sqrt{j\log N}}\cdot \frac{j\log N}{4N^2} = C\frac{\sqrt{j\log N}}{2N}.
\end{equation}
Substituting this result into \eqref{same} yields
\begin{equation}\label{star2}
\frac{1}{\sqrt{2}\pi N} h_N\left(\kappa_j + \frac{j\log N}{4N^2}\right) \leq -C\frac{\sqrt{j\log N}}{4N} \leq -C\frac{\sqrt{\log N}}{4N}.
\end{equation}
From \eqref{same2}, we have $
\kappa_j + \frac{j\log N}{4N^2} \leq \kappa_{2\log N}
$
for sufficiently large $N$. Therefore, inequality \eqref{star2} implies that $\min_{x\leq \kappa_{2\log N}} h_N(x) \leq -C\sqrt{\log N}$. By Lemma \ref{hN near edge second lemma}, the probability of this event tends to $0$ as $N \to \infty$. Therefore, we conclude that \eqref{edge: Prob: lambda-kappa} holds for $1\leq j\leq \log N$.

\smallskip

Combining the above two steps, we finalize the proof of \eqref{edge: Prob: lambda-kappa} and, simultaneously, Proposition \ref{main proposition for edge rigidity}.
\end{proof}

Since the lower bound in \eqref{eq:main-theorem} is provided by \eqref{eq: lower-bound-approx} and Proposition \ref{lower bound of hN}, and the upper bound follows from Propositions \ref{Main theorem: Eigenvalue Rigidity on the Bulk} and \ref{main proposition for edge rigidity}, the proof of Theorem \ref{main theorem for hN} is complete.

\appendix

\section{Several Model RH problems}

\subsection{Bessel Parametrix} \label{section: RHP for Bessel}

The Bessel parametrix $\Phi_{\text{Bes}}^{(\alpha)}(z)$ with $\alpha$ being a parameter is a solution to the following RH problem.

   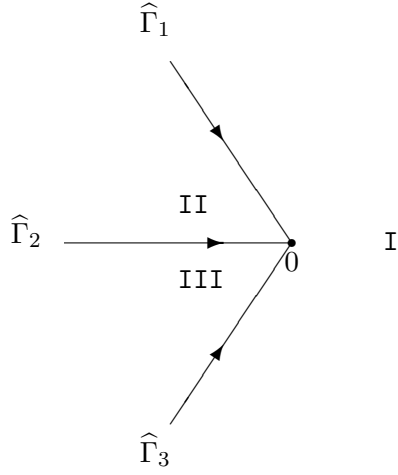
\begin{figure}[t]
\begin{center}
   \setlength{\unitlength}{1truemm}
   \begin{picture}(80,70)(-5,2)
       \put(40,40){\line(-2,-3){16}}
       \put(40,40){\line(-2,3){16}}
       \put(40,40){\line(-1,0){30}}

       \put(30,55){\thicklines\vector(2,-3){1}}
       \put(30,40){\thicklines\vector(1,0){1}}
       \put(30,25){\thicklines\vector(2,3){1}}

       \put(39,36.3){$0$}
       \put(20,11){$\widehat \Gamma_3$}
       \put(20,68){$\widehat \Gamma_1$}
       \put(3,40){$\widehat  \Gamma_2$}

       \put(52,39){$\texttt{I}$}
       \put(25,44){$\texttt{II}$}
       \put(25,34){$\texttt{III}$}

       \put(40,40){\thicklines\circle*{1}}

   \end{picture}
   \caption{The jump contours and regions for the RH problem for $\Phi_{\text{Bes}}^{(\alpha)}$.}
   \label{fig:jumps-Phi-B}
\end{center}
\end{figure}

\begin{rhp}\label{Appendix: Model RHP for Bessel}
\hfill

\begin{itemize}
\item[(a)] $\Phi_{\text{Bes}}^{(\alpha)}(z)$ is defined and analytic in $\mathbb{C}\setminus \{\cup^3_{j=1}\widehat \Gamma_j\cup\{0\}\}$, where the contours $\widehat \Gamma_j$, $j=1,2,3$,  are indicated  in Figure \ref{fig:jumps-Phi-B}.

\item[(b)] $\Phi_{\text{Bes}}^{(\alpha)}(z)$ satisfies the jump condition
\begin{equation} \label{Bessel-jump}
 \Phi_{\text{Bes}, +}^{(\alpha)}(z)=\Phi_{\text{Bes}, -}^{(\alpha)}(z)
 \left\{
 \begin{array}{ll}
   \begin{pmatrix}
                                1 & 0\\
                               e^{\alpha \pi i} & 1
                                \end{pmatrix},  &  \qquad z \in \widehat \Gamma_1, \\
   \begin{pmatrix}
                                0 & 1\\
                               -1 & 0
                                \end{pmatrix},  &  \qquad z \in \widehat \Gamma_2, \\
   \begin{pmatrix}
                                 1 & 0 \\
                                 e^{-\alpha\pi i} & 1 \\
                                \end{pmatrix}, &   \qquad z \in \widehat \Gamma_3.
 \end{array}  \right .
 \end{equation}

\item[(c)] $\Phi_{\text{Bes}}^{(\alpha)}(z)$ satisfies the following asymptotic behavior at infinity:
\begin{multline}\label{eq:Besl-infty}
 \Phi_{\text{Bes}}^{(\alpha)}(z)=
 \frac{( \pi^2 z )^{-\frac{1}{4} \sigma_3}}{\sqrt{2}}
 \begin{pmatrix}
 1 & i
 \\
 i & 1
 \end{pmatrix}
 \\
 \times  \left( I + \frac{1}{8 z^{1/2}}  \begin{pmatrix}
   -1-4\alpha^2 & -2i \\
   -2i & 1+4\alpha^2
 \end{pmatrix} +  O\left(\frac{1}{z}\right)
 \right)e^{z^{1/2}\sigma_3},\quad z\to \infty.
   \end{multline}

\item[(d)] $\Phi_{\text{Bes}}^{(\alpha)}(z)$ satisfies the  following asymptotic behaviors near the origin:
\newline
    If $\alpha<0$,
    \begin{equation}
\Phi_{\text{Bes}}^{(\alpha)}(z)=
O \begin{pmatrix}
|z|^{\alpha/2} & |z|^{\alpha/2}
\\
|z|^{\alpha/2} & |z|^{\alpha/2}
\end{pmatrix}, \qquad \textrm{as $z \to 0$}.
\end{equation}
If $\alpha=0$,
    \begin{equation}
\Phi_{\text{Bes}}^{(\alpha)}(z)=
O \begin{pmatrix}
\ln|z| & \ln|z|
\\
\ln|z| & \ln|z|
\end{pmatrix}, \qquad \textrm{as $z \to 0$}.
\end{equation}
If $\alpha>0$,
     \begin{equation}
\Phi_{\text{Bes}}^{(\alpha)}(z)= \left\{
                             \begin{array}{ll}
                                O\begin{pmatrix}
|z|^{\alpha/2} & |z|^{-\alpha/2}
\\
|z|^{\alpha/2} & |z|^{-\alpha/2}
\end{pmatrix}, & \hbox{as $z \to 0$ and $z\in \texttt{I}$,}
\\
 O \begin{pmatrix}
|z|^{-\alpha/2} & |z|^{-\alpha/2}
\\
|z|^{-\alpha/2} & |z|^{-\alpha/2}
\end{pmatrix}, & \hbox{as $z \to 0$ and $z\in \texttt{II}\cup \texttt{III} $.}
                             \end{array}
                           \right.
\end{equation}
\end{itemize}
\end{rhp}

From \cite{ABJ2004}, it follows that the above RH problem can be solved explicitly by defining
\begin{equation}\label{Phi-B-solution}
\Phi_{\text{Bes}}^{(\alpha)}(z)=\left\{
                             \begin{array}{ll}
                               \begin{pmatrix}
I_{\alpha}(z^{1/2}) & \frac{i}{\pi}K_{\alpha}(z^{1/2}) \\
\pi iz^{1/2}I'_{\alpha}(z^{1/2}) &
-z^{1/2}K_{\alpha}'(z^{1/2})
\end{pmatrix}, & z\in \texttt{I},\\ [4mm]
                              \begin{pmatrix}
I_{\alpha}(z^{1/2}) & \frac{i}{\pi}K_{\alpha}(z^{1/2}) \\
\pi iz^{1/2}I'_{\alpha}(z^{1/2}) &
-z^{1/2}K_{\alpha}'(z^{1/2})
\end{pmatrix}\begin{pmatrix}
                                1 & 0\\
                               -e^{\alpha \pi i} & 1
                                \end{pmatrix}, & z\in \texttt{II}, \\ [4mm]
                                \begin{pmatrix}
I_{\alpha}(z^{1/2}) & \frac{i}{\pi}K_{\alpha}(z^{1/2}) \\
\pi iz^{1/2}I'_{\alpha}(z^{1/2}) &
-z^{1/2}K_{\alpha}'(z^{1/2})
\end{pmatrix}\begin{pmatrix}
1 & 0 \\
e^{-\alpha\pi i} & 1
\end{pmatrix}, &  z\in \texttt{III},
                             \end{array}
                           \right.
\end{equation}
where $I_\alpha(z)$ and $K_\alpha(z)$ denote the modified Bessel functions (cf. \cite[Chapter 10]{DLMF}), the principal branch is taken for $z^{1/2}$ and the regions $\texttt{I-III}$ are illustrated in Fig. \ref{fig:jumps-Phi-B}.

\subsection{Confluent Hypergeometric Parametrix}\label{Appendix: Section: HG model RHP}

The confluent hypergeometric parametrix $\Phi_{\text{HG}}(z)=\Phi_{\text{HG}}(z;\beta)$ with $\beta$ being a parameter is a solution to the following RH problem.

\begin{rhp} \label{Appendix: HG model RHP}
\hfill

\begin{itemize}
  \item[(a)]   $\Phi_{\text{HG}}(z)$ is analytic in $\mathbb{C}\setminus \{\cup^6_{j=1}\widehat\Sigma_j\cup\{0\}\}$, where the contours $\widehat\Sigma_j$, $j=1,\ldots,6,$ are indicated in Figure. \ref{fig:jumps-Phi-C}.

  \begin{figure}[h]
\centering
   \setlength{\unitlength}{1truemm}
   \begin{picture}(100,70)(-5,2)
       \put(40,40){\line(-2,-3){18}}
       \put(40,40){\line(-2,3){18}}
       \put(40,40){\line(-1,0){30}}
       \put(40,40){\line(1,0){30}}
  \put(40,40){\line(2,-3){18}}
    \put(40,40){\line(2,3){18}}

       \put(30,55){\thicklines\vector(2,-3){1}}
       \put(30,40){\thicklines\vector(1,0){1}}
       \put(50,40){\thicklines\vector(1,0){1}}
       \put(30,25){\thicklines\vector(2,3){1}}
      \put(50,25){\thicklines\vector(2,-3){1}}
       \put(50,55){\thicklines\vector(2,3){1}}


       \put(42,36.9){$0$}
         \put(72,40){$\widehat \Sigma_1$}
           \put(60,69){$\widehat \Sigma_2$}
             \put(20,69){$\widehat \Sigma_3$}
              \put(3,40){$\widehat \Sigma_4$}
       \put(18,10){$\widehat \Sigma_5$}
          \put(60,11){$ \widehat \Sigma_6$}

%

       \put(40,40){\thicklines\circle*{1}}
\end{picture}
   \caption{The jump contours for the RH problem for $\Phi_{\text{HG}}$.}
   \label{fig:jumps-Phi-C}

\end{figure}
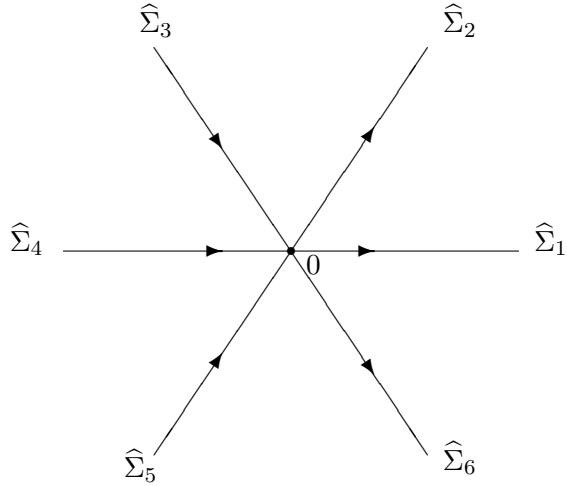

  \item[(b)] $\Phi_{\text{HG}}$ satisfies the following jump condition:
  \begin{equation}\label{HJumps}
  \Phi_{\text{HG}, +}(z)=\Phi_{\text{HG}, -} \widehat J_i(z), \quad z \in \widehat\Sigma_i,\quad j=1,\ldots,6,
  \end{equation}
  where
  \begin{equation*}
 \widehat J_1(z) = \begin{pmatrix}
    0 &   e^{-\beta \pi i} \\
    -  e^{\beta \pi i} &  0
    \end{pmatrix}, \qquad \widehat J_2(z) = \begin{pmatrix}
    1 & 0 \\
    e^{ \beta \pi i } & 1
    \end{pmatrix}, \qquad
    \widehat J_3(z) = \begin{pmatrix}
    1 & 0 \\
    e^{ -\beta\pi i} & 1
    \end{pmatrix},                                                         
  \end{equation*}
  \begin{equation*}
  \widehat J_4(z) = \begin{pmatrix}
    0 &   e^{\beta\pi i} \\
     -  e^{-\beta\pi i} &  0
     \end{pmatrix}, \qquad
      \widehat J_5(z) = \begin{pmatrix}
     1 & 0 \\
     e^{- \beta\pi i} & 1
     \end{pmatrix},\qquad
     \widehat J_6(z) = \begin{pmatrix}
   1 & 0 \\
   e^{\beta\pi i} & 1
   \end{pmatrix}.
  \end{equation*}

  \item[(c)] $\Phi_{\text{HG}}$ satisfies the following asymptotic behavior at infinity:
  \begin{equation} \label{Appendix: model RHP HG asymptotics}
\Phi_{\text{HG}}(z) = e^{-\frac{\pi i}{2}\beta} (I + O(z^{-1}) (z)^{\beta\sigma_3} e^{-i\frac{z}{2}\sigma_3} \tilde{\chi}(z), \qquad z\to\infty
\end{equation}
where $\tilde{\chi}(z)$ is given by
\begin{eqnarray}
\tilde{\chi}(z) = \left\{\begin{array}{ll}
e^{\pi i \beta\sigma_3}, & \arg z\in (0, \pi)\\
\left(
  \begin{array}{cc}
    0 & -1 \\
    1 & 0 \\
  \end{array}
\right), & \arg z \in (\pi, 2\pi).
\end{array}
\right.
\end{eqnarray}
  
\item[(d)] As $z\to 0$, we have $\Phi_{\text{HG}}(z)=\Boh(\log |z|)$.

\end{itemize}

\end{rhp}

From \cite{IK}, it follows that the above RH problem can be solved explicitly in the following way. For $z$ belonging to the region bounded by the rays $\widehat \Sigma_1$ and $\widehat \Sigma_2$,
\begin{equation}\label{Hsolution}
\Phi_{\text{HG}}(z)=C_1\left(\begin{array}{ll}
\psi(\beta,1,e^{\frac{\pi i}{2}}z)e^{2 \beta \pi i}e^{-\frac{iz}{2}}&-\frac{\Gamma(1-\beta)}{\Gamma(\beta)}\psi(1-\beta,1,e^{-\frac{\pi i}{2}}z)e^{\beta \pi i}e^{\frac{iz}{2}}\\
-\frac{\Gamma(1+\beta)}{\Gamma(-\beta)}\psi(1+\beta,1,e^{\frac{\pi i}{2}}z)e^{\beta \pi i}e^{-\frac{iz}{2}}
&\psi(-\beta,1,e^{-\frac{\pi i}{2}}z)e^{\frac{iz}{2}}\end{array}\right),
\end{equation}
where the confluent hypergeometric function $\psi(a,b;z)$ is the unique solution to the Kummer's equation
\begin{equation} \label{Kummer-equation}
z\frac{\ud^2y}{\ud z^2}+(b-z)\frac{\ud y}{\ud z}-ay=0
\end{equation}
satisfying the boundary condition $\psi(a,b,z)\sim  z^{-a}$ as $z\to \infty$ and $-\frac{3\pi }{2} < \arg z < \frac{3\pi}{2}$;
see \cite[Chapter 13]{DLMF}.  The branches of the multi-valued functions are chosen such that $-\frac{\pi}{2}<\arg z<\frac{3\pi}{2}$ and
$$
C_1=
\begin{pmatrix}
e^{-\frac32 \beta \pi i} & 0
\\
0 & e^{\frac12 \beta \pi i}
\end{pmatrix}
$$
is a constant matrix. The explicit formula of $\Phi_{\text{HG}}(z)$ in the other sectors is then determined by using the jump condition \eqref{HJumps}. 

From \cite[Lemma C.1]{CFL2021}, one can show that for $\Re\beta = 0$, taking the limit where $z$ approaches $0$ with $\arg z\in (\pi/4,\pi/2)$, one has
\begin{equation}\label{Appendix: HG specific value}
\lim_{z\to 0} \left(\Phi_{HG}(z)\right)_{1,1} = \Gamma (1-\beta),\qquad \lim_{z\to 0} \left(\Phi_{HG}(z)\right)_{2,1} = \Gamma (1-\beta)
\end{equation}
and
\begin{equation} \label{Appendix: HG Lemma}
\lim_{z\to 0}\left(\Phi_{\text{HG}}^{-1}(z)\frac{d}{dz}\Phi_{\text{HG}}(z)\right)_{2,1}={ -\frac{2\pi  \beta}{e^{\pi i \beta}-e^{-\pi i \beta}}}.
\end{equation}

\subsection{A Modified Painlev\'e V Model RH Problem.} \label{Appendix: Section: PV model RHP}

The modified Painlev\'e V parametrix $\widehat{\Phi}_{PV}(z; s)$ with $s$ being a parameter is a solution to the following RH problem. 

\begin{rhp} \label{Appendix: PV RHP}
\hfill
\begin{itemize}
    \item[(a)] $\widehat{\Phi}_{PV}:\mathbb C\setminus \{ \cup_{j=1}^7 \widehat{\Gamma}_j \cup \{0,1\}\} \to \mathbb C^{2\times 2}$ is analytic, where $\widehat{\Gamma}_j$, $j=1,2,...,7$ are depicted in Figure \ref{fig:hatPsi}.
    \item[(b)] $\widehat{\Phi}_{PV}$ satisfies the jump conditions
    \begin{equation}\label{jump hatPsi}\widehat{\Phi}_{PV, +}(z)=\widehat{\Phi}_{PV, -}(z)\widehat J_k,\qquad
    z\in\widehat\Gamma_k,\end{equation}  where
                \begin{align*}
                &\widehat J_1=\begin{pmatrix}1&0\\e^{-{\sqrt{2}\pi}(\gamma_1+\gamma_2)}&1\end{pmatrix},
                &&\widehat J_2=\begin{pmatrix}1&0\\e^{-{\sqrt{2}\pi}(\gamma_1+\gamma_2)}&1\end{pmatrix},\\
                &\widehat J_3=\begin{pmatrix}1&0\\1&1\end{pmatrix},
                &&\widehat J_4=\begin{pmatrix}1&0\\ 1&1\end{pmatrix},\\
                &\widehat J_5=\begin{pmatrix}1&e^{{\sqrt{2}\pi}\gamma_2}\\0&1\end{pmatrix},
&&\widehat J_6=\begin{pmatrix}0&e^{{\sqrt{2}\pi}\gamma_1+{\sqrt{2}\pi}\gamma_2}\\ -e^{-{\sqrt{2}\pi}\gamma_1-{\sqrt{2}\pi}\gamma_2}&0\end{pmatrix},\\
                &\widehat J_7=\begin{pmatrix}0&1\\-1&0\end{pmatrix},                
                \end{align*}
    \item[(c)] $\Phi_{\text{PV}}$ satisfies the following asymptotic behavior at infinity:
    \begin{equation}\label{Psi ashat}
    \widehat{\Phi}_{PV}(z)=\left(I+\frac{\widehat{\Phi}_1}{z}+\frac{\widehat{\Phi}_2}{z^2}+O(z^{-3})\right)
    \widehat{\Phi}^\infty(z)e^{\pm \frac{s}{2}z\sigma_3} \qquad  \mbox{ as $z\to \infty$ \text{ and } $\pm \mathrm{Im}z>0$,}
    \end{equation}
    where
    $\widehat{\Phi}^\infty$ is defined in \eqref{hatpsiinfty}.
                \item[(d)] As $z \to x\in\{0,1\}$, we have
                \begin{equation*}
                \widehat{\Phi}_{PV}(z;s)=O(\log(z-x)).
                \end{equation*}
    \end{itemize}

\begin{figure}[t]
\centering
    \setlength{\unitlength}{0.8truemm}
    \begin{picture}(75,55)(5,10)
    \put(45,50){\thicklines\circle*{.8}}
    \put(60,50){\thicklines\circle*{.8}}
    \put(30,50){\thicklines\circle*{.8}}
    \put(55,50){\thicklines\vector(1,0){.0001}}
    \put(40,50){\thicklines\vector(1,0){.0001}}
    \put(15,50){\thicklines\vector(1,0){.0001}}
    \put(75,50){\thicklines\vector(1,0){.0001}}
    \put(0,50){\line(1,0){90}}
    \put(60,50){\line(1,1){25}}
    \put(30,50){\line(-1,1){25}}
    \put(60,50){\line(1,-1){25}}
    \put(30,50){\line(-1,-1){25}}
    \put(75,65){\thicklines\vector(1,1){.0001}}
    \put(15,65){\thicklines\vector(1,-1){.0001}}
    \put(75,35){\thicklines\vector(1,-1){.0001}}
    \put(15,35){\thicklines\vector(1,1){.0001}}

    \put(75,75){\small $\widehat{\Gamma}_4$}
    \put(15,75){\small $\widehat{\Gamma}_1$}
    \put(75,20){\small $\widehat{\Gamma}_3$}
    \put(15,20){\small $\widehat{\Gamma}_2$}
    \put(45,54){\small $\widehat{\Gamma}_5$}
    \put(15,54){\small $\widehat{\Gamma}_6$}
    \put(75,54){\small $\widehat{\Gamma}_7$}

    \put(30,53){\small $0$}
    \put(59,53){\small $1$}

    \end{picture}

    \caption{The jump contour for $\widehat{\Phi}_{PV}$.}
    \label{fig:hatPsi}
\end{figure}
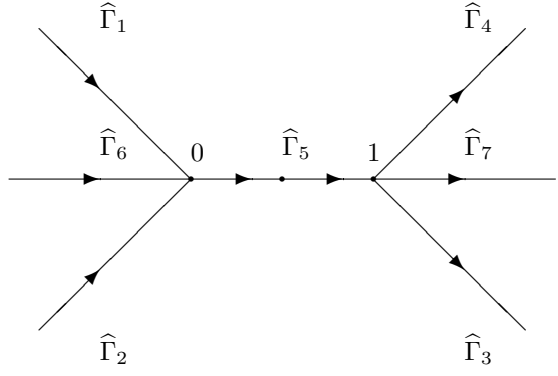
    
\end{rhp}

It is shown in \cite[Lemma 6.1]{CFL2021} that the above RH problem is uniquely solvable when $\gamma_1, \gamma_2 \in \mathbb{R}$ and $s \in -i \mathbb{R}_+$. The solution can be constructed by using the $\psi$-functions associated with the Painlev\'e V equation; see also \cite[Sec. 3]{Cla:Kra2015}.

\section*{Acknowledgements}
Dan Dai was partially supported by grants from the Research Grants Council of the Hong Kong Special Administrative Region, China [Project No. CityU 11311622, CityU 11306723 and CityU 11301924].

\end{document}